\documentclass{article}
\usepackage{algorithm}
\usepackage{algorithmic}

\usepackage{enumerate}
\usepackage{mathbbold}
\usepackage{natbib}
\usepackage{bbm}
\usepackage[T1]{fontenc}
\usepackage{latexsym, amssymb, amsmath, amsfonts, amsthm}
\usepackage{graphicx, color}
\usepackage{subfigure}
\usepackage{epsf}
\usepackage{epstopdf}
\usepackage{authblk} 
\usepackage{multirow}
\usepackage{diagbox}
\usepackage{stmaryrd}
\usepackage{mathrsfs}
\usepackage{hyperref}
\hypersetup{
colorlinks =true, 
linkcolor=black, 
anchorcolor=black,
citecolor=black}

\numberwithin{equation}{section}

\newcommand{\R}{{\mathbb R}}

\newcommand{\C}{{\mathbb C}}

\newcommand{\be}{\begin{eqnarray}}
\newcommand{\ben}{\begin{eqnarray*}}
\newcommand{\en}{\end{eqnarray}}
\newcommand{\enn}{\end{eqnarray*}}

\newcommand{\pa}{\partial}

\newcommand{\hx}{\hat{x}}

\newtheorem{theorem}{Theorem}[section]

\theoremstyle{definition}
\newtheorem{example}{Example}[section]

\graphicspath{{./Figure}}

\begin{document}
\fontsize{9}{11}\selectfont
\title{\bf A quantitative direct sampling method for inhomogeneities from multi-frequency backscattering measurements}

\author[1]{Yukun Guo}
\author[2]{Xiaodong Liu} 
\makeatletter 
\renewcommand{\AB@affilnote}[1]{\footnotemark[#1]} 
\renewcommand{\AB@affillist}[1]{}
\renewcommand{\AB@affillist}[2]{} 
\makeatother

\date{}
\maketitle
\footnotetext[1]{School of Mathematics, Harbin Institute of Technology, Harbin, China, Email: ykguo@hit.edu.cn} 
\footnotetext[2]{Corresponding author. State Key Laboratory of Mathematical Sciences, Academy of Mathematics and Systems Science, Chinese Academy of Sciences, Beijing 100190, China. Email: xdliu@amt.ac.cn}

\begin{abstract}
The inverse scattering problem from the multi-frequency backscattering data is a long-standing open problem. We advance the theory by proving a local uniqueness result. Moreover,  we introduce a direct sampling method for quantitatively reconstructing unknown inhomogeneities. Comprehensive numerical experiments validate the robustness, accuracy, and computational effectiveness of the proposed quantitative direct sampling method. 

\vspace{.2in}
{\bf Keywords:} inverse medium scattering; quantitative direct sampling method; multi-frequency backscattering data; Helmholtz equation; Born approximation.

\end{abstract}

\section{Introduction}
Inverse scattering from backscattering data is a foundational and impactful research area at the intersection of applied mathematics, computational physics, and engineering, with critical real-world applications such as non-destructive testing, medical imaging, radar systems, and remote sensing. This research direction focuses on retrieving the geometric and  physical properties of unknown scatterers (e.g., obstacles, inhomogeneous media) using only backscattered wave measurements—where the source and receiver are co-located or aligned in a backscattering configuration—a setup that is widely adopted in practical scenarios due to its simplicity and operational efficiency. A central bottleneck in this field is the lack of theoretical guarantees for uniqueness and stable reconstruction from backscattering data.

Denote by $q : \mathbb{R}^n \to \mathbb{C}$ the compactly supported contrast and by $\operatorname{supp}(q) = D$ its support. We assume that $\Re(q)>-1$ and $\Im(q)\geq 0$ throughout $\R^n$. Then the total field $u^\mathrm{t}$ solves
\begin{equation}
    \Delta u^\mathrm{t} + k^2(1+q)u^\mathrm{t} = 0 \quad \text{in } \mathbb{R}^n,\quad n=2,3.
\end{equation}
Furthermore, the scattered wave $u^{s} = u^\mathrm{t} - u^{inc}$ is required to  satisfy the Sommerfeld radiation condition
\begin{equation*}
\lim_{r\to\infty} r^{\frac{n-1}{2}}\Bigl(\frac{\partial}{\partial r} - \mathrm{i}k\Bigr)u^{s}= 0,
\end{equation*}
where $r:=|x|$ and the limit holds uniformly in all directions $\hat{x} = x/|x|\in\mathbb{S}^{n-1}:=\{x\in\mathbb{R}^n: |x|=1\}$. 
Note that every radiating solution of the Helmholtz equation has the following asymptotic behavior at infinity:
\begin{equation}
 u^s(x)
=\frac{e^{{\rm i}kr}}{r^{\frac{n-1}{2}}}\left\{u^{\infty}(\hat{x})+\mathcal{O}\left(\frac{1}{r}\right)\right\}\quad\mbox{as }\,r:=|x|\rightarrow\infty
\end{equation}
uniformly with respect to all directions $\hx=x/|x|\in \mathbb{S}^{n-1}$. The complex valued function $u^\infty(\hat{x})$ defined on the unit sphere $\mathbb{S}^{n-1}$ is known as the far-field pattern with $\hat{x}\in \mathbb{S}^{n-1}$ denoting the observation direction.

The inverse medium problem (IMP) aims to reconstruct the spatially varying contrast function $q$ of an inhomogeneous medium from boundary or far-field measurements of scattered waves. To tackle this challenging problem, two main categories of numerical approaches have been developed in the literature. One category is nonlinear iterative methods, which minimize the data-misfit functional with regularization to mitigate the ill-posedness of the problem.
Among classical iterative methods, representative ones include the contrast source inversion \cite{vandenberg1997contrast}, Newton-type methods \cite{hanke1999iterative}, subspace-based optimization method \cite{Chen2010SOMInhom}, and the recursive linearization method \cite{BaoTriki}. The other category consists of qualitative methods, which focus on reconstructing the support of the scatterer without the need for full quantitative inversion.  Specifically, we refer to the linear sampling method \cite{colton1996simple}, the factorization method \cite{kirsch1999factorization}, and direct sampling methods \cite{ItoJinZou,LiuIP17}; these methods construct indicator functions from far-field data, enabling non-iterative and initial-guess-free shape recovery. While these methods are robust, they only provide qualitative profiles.
For a comprehensive review of the state-of-the-art mathematical theory and numerical approaches to solving inverse time-harmonic medium scattering problems up to 2019, we  refer to Chapter 11 of the standard monograph \cite{colton2019inverse}. In recent years, significant progress has also been made in applying deep learning to IMPs \cite{desai2025neural,khoo2019switchnet,meng2024kernel,li2024reconstruction,zhou2026exploring,zhou2025recovery}. For the uniqueness and stability analyses of the time domain IMP from backscattering data, we refer to \cite{StefUhl} and \cite{wang1998stability}, respectively.

However, to the best of our knowledge, no practical numerical algorithms based on backscattering data have been reported in the literature. 
The work presented in this paper fills this gap by proposing a quantitative direct sampling method (QDSM) based on the Born approximation, which inherits the efficiency of the linear Born approximation while enabling quantitative contrast reconstruction from backscattering data, without requiring iterative forward solves or additional correction terms.
We first recall two typical time-harmonic acoustic incident fields. The first incident field is the plane wave
\begin{equation}
    u^{inc}(x,\theta, k)=e^{{\rm i}k x\cdot \theta},\quad x\in\R^n,\, k\in \R^+,
\end{equation}
where  $k=\omega/c$ is the wavenumber, $\omega$
 is the angular frequency,  $c$ is the speed of sound and $\theta\in \mathbb{S}^{n-1}$ is the direction of propagation. In the sequel, for an incident plane wave $u^{inc}(x,\theta, k)$, we indicate dependence of the  far-field pattern on the incident direction $\theta$ and the wave number $k$ by writing $u^\infty(\hat{x}, \theta, k)$.
 The other incident field is the point source
 \begin{equation}\label{point sources}
u^{inc}(x,z, k)=\Phi(x,z, k) 
 := \begin{cases}
   \frac{e^{{\rm i}k|x-z|}}{4\pi|x-z|}, & n = 3;\medskip \\
   \frac{{\rm i}}{4} H_0^{(1)}(k|x-z|), & n = 2,
   \end{cases} \quad x\in\R^n\backslash\{z\},\, k\in \R^+,
\end{equation}
where $H_0^{(1)}$ is the Hankel function of the first kind and order zero. For the corresponding scattered field we write $u^s(\cdot, z,k)$ to indicate the dependence of point-source $z\in\R^n$ and the wavenumber $k$.

The inverse medium scattering problems of interest in this paper are to quantitatively determine the contrast $q$ from one of the following two backscattering data sets, which possess a significant application value:
\begin{equation*}
    \mathcal{M}_f:=\{u^\infty(-\theta, \theta, k):\,\, \theta\in \mathrm{S}^{n-1},\,k\in\R^+\} \quad\mbox{and}\quad 
    \mathcal{M}_s:=\{u^s(x,x,k):\,\, x\in\pa B_R,\,k\in\R^+\},
\end{equation*}
where  $\pa B_R$ is a sphere centered at the origin with radius $R$ large enough to contain the contrast support $D$. 
It is worth noting that the uniqueness of these inverse backscattering problems remains a long-standing open issue. In this paper, we make a first step from a numerical perspective by designing a QDSM.
The QDSM was first proposed for acoustic sources \cite{liu2026direct}. A key feature of this method is that it inherits the advantages of qualitative direct sampling methods while enabling quantitative amplitude recovery without the need for iterative forward solvers. We refer to several applications of the quantitative direct sampling method, including those for elastic and electromagnetic sources \cite{LiuShi}, biharmonic sources \cite{liu2025radon}, and convex obstacles \cite{LiLiuShi2025}. The QDSM is based on the linear representation of measurements with respect to the unknowns; therefore, linearization methods are crucial to the design of its indicators. Unlike obstacle scattering, which relies on the physical optics approximation, we employ the Born approximation for medium scattering. In this scenario, the scattering inhomogeneity is weak, such that the total field inside the scatterer can be approximated by the incident field alone. Note that \cite{kirsch2017remarks} established that, for a fixed frequency, the general nonlinear problem and its Born approximation have the same order of ill-posedness. 

The remainder of the paper is organized as follows. 
In Section \ref{Local-uni}, we establish a local uniqueness result for the inhomogeneity from the multi-frequency backscattering far-field patterns.
In Section \ref{section-field-representation}, we derive the representations of the backscattering scattered fields and the corresponding far-field patterns.  In Section \ref{section-QDSM}, we design the indicators of QDSM. Finally, in Section \ref{section-Numerical examples}, comprehensive numerical examples are presented to show the validity and robustness of the proposed QDSM. 

\section{A local uniqueness result}\label{Local-uni}
The uniqueness of recovering the contrast function from the multi-frequency backscattering data is still a long-standing open problem. In this section, we make a first step by showing that the uniqueness from multi-frequency backscattering far-field patterns holds when the contrast function is at most a first-order polynomial perturbation.

‌The point source \eqref{point sources} is exactly the fundamental solution to the Helmholtz equation. Since the radiating scattered field $u^\mathrm{s}$ satisfies
\begin{equation*}
\Delta u^\mathrm{s} + k^2u^\mathrm{s} = -k^2 q u^\mathrm{t},
\end{equation*}
we obtain the following integral representation:
\begin{equation}\label{Lippman-Schwinger}   u^s(x)=k^2\int_{\mathbb{R}^n}q(y)\Phi(x,y,k)u^t(y)\,\mathrm{d}y,\quad x\in\mathbb{R}^n,
\end{equation}
a relation widely known as the Lippmann–Schwinger equation. 
Using the asymptotic behavior of the fundamental solution
\begin{equation}\label{Phi-asy}
    \Phi(x,y,k)=\gamma_n(k)\frac{e^{{\rm i}k|x|}}{|x|^{\frac{n-1}{2}}} \left\{ e^{-{\rm i}k\hat{x}\cdot y} + \mathcal{O}\left( \frac{1}{|x|} \right) \right\} \quad \text{as } |x| \to \infty,
\end{equation}
where
\begin{equation*}
    \gamma_n(k)=\frac{e^{{\rm i}\frac{\pi}{4}}}{\sqrt{8k\pi}} \left( e^{-{\rm i}\frac{\pi}{4}} \sqrt{\frac{k}{2\pi}} \right)^{n-2}, 
    \quad k\in \R^+,
\end{equation*}
we deduce that the far field pattern takes the form
\begin{equation*}
    u^\infty(\hat{x}, \theta, k)=k^2\gamma_n(k)\int_{\mathbb{R}^n}q(y)e^{-{\rm i}k\hat{x}\cdot y}u^t(y, \theta, k)\,\mathrm{d}y,\quad \hat{x}\in\mathbb{S}^{n-1}.
\end{equation*}
Suppose that $q(x)=0$ for $|x|\geq R$ with some $R>0$. 
From \eqref{Lippman-Schwinger}, we further derive that the total field has the following representation
\begin{equation}\label{Lippman-Schwinger2}
    u^t(x,\theta, k)=u^{inc}(x,\theta, k)+k^2\int_{\mathbb{R}^n}q(y)\Phi(x,y,k)u^t(y, \theta, k)\,\mathrm{d}y,\quad x\in\mathbb{R}^n.
\end{equation}
In the Banach space $C(\overline{B_R})$, define the operator $T_q: C(\overline{B_R})\rightarrow C(\overline{B_R})$ by
\begin{equation*}
    (T_q u)(x):=\int_{B_R}\Phi(x,y,k)q(y)u(y)\mathrm{d}y, \quad x\in \overline{B_R}.
\end{equation*}
Rearranging \eqref{Lippman-Schwinger2}, the total field can be written as
\begin{equation*}
    u^t(x,\theta, k)=(I-k^2T_q)^{-1}u^{inc}(x,\theta, k),
\end{equation*}
where $I$ denotes the identity operator on $C(\overline{B_R})$. For sufficiently small $k>0$, the Neumann series expansion of the inverse operator converges, giving
\begin{equation*}
(I - k^2 T_q)^{-1} = I + k^2 T_q + \mathcal{O}(k^4) \quad \text{as } k\rightarrow 0.
\end{equation*}
Substituting the Taylor expansion of the incident plane wave
\begin{equation*}
u^{inc}(x,\theta,k) = e^{{\rm i}k\theta\cdot x} = 1 + {\rm i}k\theta\cdot x - k^2(\theta\cdot x)^2 + \mathcal{O}(k^3) \quad \text{as } k\rightarrow 0,
\end{equation*}
we obtain the following asymptotic behavior for the total field as $k\rightarrow 0$:
\begin{align}\label{ut-asym}
u^t(x,\theta, k)
= \left(I + k^2 T_{q} + \mathcal{O}(k^4)\right)\left(1 + {\rm i}k\theta\cdot x - k^2(\theta\cdot x)^2 + \mathcal{O}(k^3)\right) 
= 1 + {\rm i}k\theta\cdot x + \mathcal{O}(k^2).
\end{align}

\begin{theorem}
    Let $q_1$ and $q_2$ be two contrast functions with compact support in $B_R$. Assume that the corresponding backscattering far-field patterns are
    \begin{equation*}
        u^{\infty}_{q_1}(-\theta, \theta, k)=u^{\infty}_{q_2}(-\theta, \theta, k),\quad \mbox{for all}\,\, \theta\in \mathbb{S}_0, \, k\in (k_{\rm min}, k_{\rm max}),
    \end{equation*}
    where  $\mathbb{S}_0\subset\mathbb{S}^{n-1}$ and $(k_{\rm min}, k_{\rm max})\subset\R^+$. If 
    \begin{equation}\label{q-q-assumption}
        q_1(y)-q_2(y)=\alpha\cdot y +c,\quad y\in B_R
    \end{equation}
    holds for some $\alpha\in\C^n$ and $c\in\C$, then $q_1=q_2$.
\end{theorem}
\begin{proof}
Recall that the backscattering far-field pattern is given by 
    \begin{equation}\label{bk-far-1}
        u^\infty_q(-\theta, \theta, k)=k^2\gamma_n(k)\int_{\mathbb{R}^n}q(y)e^{{\rm i}k\theta\cdot y}u^t(y, \theta, k)\,\mathrm{d}y
    \end{equation}
for all $\theta\in \mathbb{S}_0$, $k\in (k_{\min}, k_{\max})$, and $q=q_1,q_2$.
This implies that the backscattering far-field pattern depends analytically on the incident direction $\theta$ and the wave number $k$.
Consequently, we have 
\begin{equation}\label{bk-far-2}
        u^{\infty}_{q_1}(-\theta, \theta, k)=u^{\infty}_{q_2}(-\theta, \theta, k),\quad \mbox{for all}\,\, \theta\in \mathbb{S}^{n-1}, \, k\in \R^+.
\end{equation}
Letting $k\rightarrow 0$, employing the asymptotic \eqref{ut-asym}, and noting that both $q_1$ and $q_2$ are compactly supported in $B_R$, we deduce from \eqref{bk-far-1} and \eqref{bk-far-2} that
\begin{equation}\label{q1-q2}
    \int_{B_R}[q_1(y)-q_2(y)] \mathrm{d}y=0
\end{equation}
and
\begin{equation}\label{q1-q2theta}
\int_{B_R}[q_1(y)-q_2(y)]\theta\cdot y \,\mathrm{d}y=0\quad \mbox{for all}\,\,\theta\in\mathbb{S}^{n-1}.
\end{equation}
Let $r(y) := q_1(y) - q_2(y)$. By assumption \eqref{q-q-assumption},
\begin{equation*}
r(y) = \alpha \cdot y + c
\end{equation*}
for some $\alpha \in \mathbb{C}^n$ and $c \in \mathbb{C}$. We will show that $r \equiv 0$.

Substituting the form of $r$ into the integral condition \eqref{q1-q2} yields
\begin{equation*}
0=\int_{B_R} (\alpha \cdot y + c) \, \mathrm{d}y
= \int_{B_R} \alpha \cdot y \, \mathrm{d}y + c \int_{B_R} \mathrm{d}y.
\end{equation*}
Since each coordinate function $y_i$ is odd on $B_R$, it follows that
\begin{equation*}
\int_{B_R} \alpha \cdot y \, \mathrm{d}y = \sum_{i=1}^n \alpha_i \int_{B_R} y_i \, \mathrm{d}y = 0.
\end{equation*}
Thus
\begin{equation*}
c \int_{B_R} \mathrm{d}y=c|B_R|=0,
\end{equation*}
which implies $c = 0$. Therefore $r(y) = \alpha \cdot y$.

Substituting this into the integral condition \eqref{q1-q2theta} gives
\begin{equation*}
\int_{\mathbb{R}^n} (\alpha \cdot y)(\theta \cdot y) \, \mathrm{d}y = 0 \quad \mbox{for all}\,\,\theta \in \mathbb{S}^{n-1}.
\end{equation*}
Expanding the product yields
\begin{equation*}
(\alpha \cdot y)(\theta \cdot y)
= \left( \sum_{i=1}^n \alpha_i y_i \right)
\left( \sum_{j=1}^n \theta_j y_j \right)
= \sum_{i=1}^n \sum_{j=1}^n \alpha_i \theta_j y_i y_j.
\end{equation*}
Integrating term by term, we obtain
\begin{equation*}
\int_{B_R} (\alpha \cdot y)(\theta \cdot y) \, \mathrm{d}y
= \sum_{i=1}^n \sum_{j=1}^n \alpha_i \theta_j \int_{B_R} y_i y_j \, dy.
\end{equation*}
For $i \neq j$, the integral vanishes because $y_i y_j$ is odd. For $i = j$, let
\begin{equation*}
A := \int_{B_R} y_i^2 \, \mathrm{d}y,
\end{equation*}
which is positive and independent of $i$. Hence
\begin{equation*}
\sum_{i=1}^n \alpha_i \theta_i \, A
= (\alpha \cdot \theta)\, A.
\end{equation*}
The condition therefore becomes
\begin{equation*}
(\alpha \cdot \theta)\, A = 0 \quad\mbox{for all}\,\, \theta \in \mathbb{S}^{n-1}.
\end{equation*}
Since $A > 0$, we have $\alpha\cdot\theta=0$ for all $\theta\in\mathbb{S}^{n-1}$, so $\alpha=0$.
We conclude that $r(y) \equiv 0$, i.e., $q_1 = q_2$.
\end{proof}

\section{The field representations under Born approximation} \label{section-field-representation}
The Fourier transform of a function $f$ is defined as
\begin{equation*}
\widehat{f}(\xi) = \mathcal{F} [f](\xi) := \int_{\R^n} f(x) e^{-{\rm i} \xi \cdot x} \, \mathrm{d}x, \quad \xi \in \R^n.
\end{equation*}
Meanwhile, the inverse Fourier transform is given by
\begin{equation*}
f(x) = \mathcal{F}^{-1}[\widehat{f}](x) := \frac{1}{(2\pi)^n} \int_{\R^n} \widehat{f}(\xi) e^{{\rm i} \xi \cdot x} \, \mathrm{d}\xi,\quad x \in \R^n.
\end{equation*}

\subsection{Backscattering far-field patterns of plane waves} \label{subsection-far field}
For an incident plane wave $u^{\text{inc}}(y, \theta,k)=e^{{\rm i}k\theta\cdot y}$, under the Born approximation, the Lippmann–Schwinger equation becomes
\begin{equation*}
u^s_b(x,\theta,k)=k^2\int_{\mathbb{R}^n}q(y)\Phi(x,y,k)u^{\text{inc}}(y, \theta,k)\,\mathrm{d}y,\quad x\in\mathbb{R}^n,\, k\in \R^+.
\end{equation*}
Using the asymptotic behavior \eqref{Phi-asy} of the fundamental solution,
we derive the corresponding far-field pattern
\begin{equation*}
    u^\infty_b(\hat{x}, \theta, k)=k^2\gamma_n(k)\int_{\mathbb{R}^n}q(y)e^{-{\rm i}k\hat{x}\cdot y}e^{{\rm i}k\theta\cdot y}\,\mathrm{d}y,\quad \hat{x}\in \mathbb{S}^{n-1},\, k\in \R^+.
\end{equation*}
Specifically, by setting $\hat{x} = -\theta$, we obtain the backscattering far-field pattern representation
\begin{equation}\label{far field b}
    u^\infty_b(-\theta,\theta,k)=k^2\gamma_n(k)\int_{\mathbb{R}^n}q(y)e^{2{\rm i}k\theta\cdot y}\,\mathrm{d}y,\quad \theta\in \mathbb{S}^{n-1},\, k\in \R^+.
\end{equation}

\subsection{Backscattering scattered fields of point sources} \label{subsection-near field}
For an incident point source $u^{\text{inc}} = \Phi(y, z, k)$, under the Born approximation, the Lippmann–Schwinger equation \eqref{Lippman-Schwinger} is approximated by
\begin{equation*}
    u^s_b(x,z,k)=k^2\int_{\mathbb{R}^n}q(y)\Phi(x,y,k)\Phi(y,z,k)\,dy,\quad x\in\mathbb{R}^n,\, k\in \R^+.
\end{equation*}
Setting $x = z \in \partial B_R$ and applying the reciprocity relation $\Phi(x,y,k)=\Phi(y,x,k)$ for $x\neq y$, we obtain the representation for the backscattering scattered field
\begin{equation}\label{scattered field b}
u^s_b(x,x,k)=k^2\int_{\mathbb{R}^n}q(y)\Phi^2(x,y,k)\,dy,\quad x\in \partial B_R,\, k\in \R^+.
\end{equation}

\section{The quantitative direct sampling method} \label{section-QDSM}

Let $G$ denote the sampling region such that $D\subset G\subset\mathbb{R}^n$. For each sampling point $z\in G$, we define the following two indicators
\begin{equation}\label{I-infty}
   I^\infty(z) := \frac{1}{\pi^n}\int^{\infty}_{0}\gamma^{-1}_n(k)k^{n-3}\int_{\mathbb{S}^{n-1}}u^{\infty}_b(-\theta, \theta, k)e^{-2{\rm i}k\theta\cdot z}\,\mathrm{d}s(\theta)\mathrm{d}k, \quad z\in G
\end{equation}
and
\begin{equation}\label{I-s}
   I^s(z) := \frac{R^{n-1}}{\pi^n}\int^{\infty}_{0}e^{-2{\rm i}kR}\gamma^{-2}_n(k)k^{n-3}\int_{\mathbb{S}^{n-1}}u^s_b(x,x, k)e^{2{\rm i}k\hat{x}\cdot z}\,\mathrm{d}s(\hat{x})\mathrm{d}k, \quad z\in G.
\end{equation}

\begin{theorem}\label{thm:Iinfty=q}
\begin{equation*}
    I^{\infty}(z) = q(z), \quad z\in G.
\end{equation*}
\end{theorem}
\begin{proof}
    Define $\xi:=2k\theta$, then 
    \begin{equation*}
        \theta=\xi/|\xi|\quad\mbox{and}\quad k=|\xi|/2,
    \end{equation*} 
    we rewrite \eqref{I-infty}
    \begin{equation*}
        I^\infty(z) = \frac{1}{(2\pi)^n}\int_{\R^n}\gamma^{-1}_n(k)k^{-2}u^{\infty}_b(-\theta, \theta, k)e^{-{\rm i}\xi\cdot z}\,\mathrm{d}\xi, \quad z\in G
    \end{equation*}
    Inserting \eqref{far field b} into the above equality, we have
\begin{align*}
I^\infty(z) 
&= \frac{1}{(2\pi)^n}\int_{\mathbb{R}^n}\int_{\mathbb{R}^n}q(y)e^{{\rm i}\xi\cdot y}\,\mathrm{d}y e^{-{\rm i}\xi\cdot z}\,\mathrm{d}\xi \notag\\
&= \int_{\mathbb{R}^n}\mathcal{F}^{-1} [q](\xi) e^{-{\rm i}\xi\cdot z}\,\mathrm{d}\xi \notag\\
&= \mathcal{F}\mathcal{F}^{-1} [q](z) \notag\\
&= q(z), \quad z\in G. 
\end{align*}
\end{proof}

\begin{theorem}\label{thm: near-field approximation}
Assume that the measurement surface is $\partial B_R$, then 
\begin{equation*}
    I^{s}(z) = q(z)+\mathcal{O}(R^{-1}), \quad z\in G.
\end{equation*}
\end{theorem}
\begin{proof}
With the help of \eqref{Phi-asy}, we have
\begin{equation*}
    \Phi^2(x,y,k)=\gamma_n^2(k)\frac{e^{2{\rm i}kR}}{R^{n-1}} \left\{ e^{-2{\rm i}k\hat{x}\cdot y} + e^{-{\rm i}k\hat{x}\cdot y}\mathcal{O}\left( \frac{1}{R} \right) \right\} \quad \text{as } R=|x| \to \infty.
\end{equation*}
Combining this with \eqref{scattered field b}, we deduce that
\begin{equation}\label{scattered field b2}
    u^s_b(x,x,k)=k^2\gamma_n^2(k)\frac{e^{2{\rm i}kR}}{R^{n-1}}\int_{\mathbb{R}^n}q(y)\left\{ e^{-2{\rm i}k\hat{x}\cdot y} + e^{-{\rm i}k\hat{x}\cdot y}\mathcal{O}\left( \frac{1}{R} \right)\right\}\,\mathrm{d}y,\quad x\in \partial B_R,\, k\in \R^+.
\end{equation}
    Define $\eta:=2k\hat{x}$, then 
    \begin{equation*}
        \hat{x}=\xi/|\xi|\quad\mbox{and}\quad k=|\xi|/2,
    \end{equation*} 
    we rewrite \eqref{I-s}
    \begin{equation*}
        I^s(z) = \frac{R^{n-1}}{(2\pi)^n}\int_{\R^n}e^{-2{\rm i}kR}\gamma^{-2}_n(k)k^{-2}u^{s}_b(x,x, k)e^{{\rm i}\xi\cdot z}\,\mathrm{d}\xi, \quad z\in G.
    \end{equation*}
    Inserting \eqref{scattered field b2} into the above equality, we have
\begin{align*}
I^s(z) 
&= \frac{1}{(2\pi)^n}\int_{\mathbb{R}^n}\left[\int_{\mathbb{R}^n}q(y)e^{-{\rm i}\xi\cdot y}\,\mathrm{d}y +\mathcal{O}\Big( \frac{1}{R} \Big)\int_{\mathbb{R}^n}q(y)e^{-{\rm i}\frac{\xi}{2}\cdot y}\,\mathrm{d}y\right]e^{{\rm i}\xi\cdot z}\,\mathrm{d}\xi \notag\\
&= \frac{1}{(2\pi)^n}\int_{\mathbb{R}^n}\left[\mathcal{F} [q](\xi)+ \mathcal{O}\Big( \frac{1}{R} \Big)\mathcal{F} [q](\frac{\xi}{2}) \right]e^{{\rm i}\xi\cdot z}\,\mathrm{d}\xi \notag\\
&= \mathcal{F}^{-1}\mathcal{F}^{-1} [q](z) +\mathcal{O}(R^{-1})\mathcal{F}^{-1}\mathcal{F}^{-1} [q](2z)\notag\\
&= q(z)+\mathcal{O}(R^{-1}), \quad z\in G.
\end{align*}
\end{proof}


\section{Numerical examples and discussions}\label{section-Numerical examples}

In this section, several numerical examples are provided to illustrate how the proposed method behaves in typical scenarios. In subsequent numerical examples, synthetic forward data were generated by solving the Lippmann–Schwinger equation \eqref{Lippman-Schwinger} using the \textbf{IPscatt} software package \cite{buergel2019ipscatt}. To test the stability of the algorithm, the synthetic data are perturbed by artificial Gaussian noise via the \textbf{addNoise} function in \textbf{IPscatt} with a relative noise level $\delta\geq 0$. In particular, $\delta=0$ corresponds to the noise-free case. We shall utilize $\delta=5\%$ in what follows unless otherwise specified. 

Note that, from a numerical perspective, it is infeasible to integrate the wave number $k$ from $0$ to $+\infty$. Let $I^{\infty}_{\mathcal{K}}$ and $I^{s}_{\mathcal{K}}$ denote, respectively, the indicators \eqref{I-infty} and \eqref{I-s} with the integration interval $\int_0^{\infty}$ replaced by $\int_{k_{\rm min}}^{k_{\rm max}}$.
Denote by $H^{2}(\R^n)$ the Sobolev space endowed with the norm
\begin{equation*}
\|f\|_{H^{2}(\R^n)}=\frac{1}{(2\pi)^n}\left(\int_{\R^n}\left(1+|\xi|^2\right)^2|\widehat{f}(\xi)|^2{\mathrm d}\xi\right)^{1/2}.
\end{equation*}

Following the framework of source reconstruction in \cite{LiuWang-AML}, we derive an error estimate between $I^{\infty}_{\mathcal{K}}$ and $I^{\infty}$.
\begin{theorem}\label{error analysis}
Let the contrast function $q\in H^{2}(\R^n)$ with compact support $\Omega \subset B_R$, then
\begin{equation*}
|I^{\infty}(z)-I^{\infty}_{\mathcal{K}}(z)|\leq 
\begin{cases}
		\,\sqrt{\pi} \left(k_{\rm max}^{-1}\|q\|_{H^2(\R^2)}+k_{\rm min}\|q\|_{L^2(\R^2)}\right), \quad z\in \R^2, \medskip \\
		  \,2\sqrt{\pi} \left(k_{\rm max}^{-1/2}\|q\|_{H^2(\R^3)}+\frac{\sqrt{3}}{3} k_{\rm min}^{3/2}\|q\|_{L^2(\R^3)}\right), \quad z\in \R^3.
\end{cases}
\end{equation*}
\end{theorem}
\begin{proof}
For simplicity, we present the two-dimensional case. The three-dimensional counterpart can be treated similarly. From the arguments in the proof of Theorem \ref{thm:Iinfty=q}, we obtain 
\begin{align}
|I^{\infty}(z)-I^{\infty}_{\mathcal{K}}(z)|
= \left |\frac{1}{(2\pi)^2}\left(\int_{|\xi|\geq k_{\rm max}}+\int_{|\xi|\leq k_{\rm min}}\right)\int _{\R^2}q(y)e^{{\rm i}k(z-y)\cdot \xi}{\rm d}y {\rm d}{\xi}\right |,\quad z\in \R^2.
\label{truncated-error-analysis}
\end{align}
By applying the triangle inequality, Parseval's identity
\begin{equation*}
    \|q\|_{L^2(\R^2)}=\frac{1}{(2\pi)^2}\|\widehat{q}\|_{L^2(\R^2)},
\end{equation*}
and the Cauchy-Schwarz inequality, we deduce that
\begin{align*}
|I^{\infty}(z)-I^{\infty}_{\mathcal{K}}(z)|
&\leq \Bigg |\frac{1}{(2\pi)^2}\int_{|\xi|\geq k_{\rm max}}\widehat{q}(\xi)e^{{\rm i}z\cdot \xi}{\rm d}\xi\Bigg |+\Bigg |\frac{1}{(2\pi)^2}\int_{|\xi|\leq k_{\rm min}}\widehat{q}(\xi)e^{{\rm i}z\cdot \xi}{\rm d}\xi\Bigg |\cr
&\leq \left(\int_{|\xi|\geq k_{\rm max}}(1+|\xi|^2)^{-2}{\rm d}\xi\right)^{1/2}\|q\|_{H^2(\R^2)}+\left(\int_{|\xi|\leq k_{\rm min}} |e^{{\rm i}z\cdot \xi}|^2{\rm d}\xi\right)^{1/2}\|q\|_{L^2(\R^2)}\cr
&\leq \left(\int_{|\xi|\geq k_{\rm max}}|\xi|^{-4}{\rm d}\xi\right)^{1/2}\|q\|_{H^2(\R^2)}+\left(\int_{|\xi|\leq k_{\rm min}} 1^2{\rm d}\xi\right)^{1/2}\|q\|_{L^2(\R^2)}\cr
&\leq  \sqrt{\pi} k_{\rm max}^{-1}\|q\|_{H^2(\R^2)}+\sqrt{\pi} k_{\rm min}\|q\|_{L^2(\R^2)}, \quad z\in \R^2.
\end{align*}
The proof is complete.
\end{proof}

Combining Theorems \ref{thm:Iinfty=q} and \ref{error analysis}, we obtain 
\begin{equation*}
    |I^\infty_{\mathcal{K}}(z)-q(z)|=\mathcal{O}(k_{\rm max}^{(n-4)/2)})+\mathcal{O}(k_{\rm min}^{n/2}), \quad z\in\R^n.
\end{equation*}
Similarly, we have 
\begin{equation*}
    |I^s_{\mathcal{K}}(z)-q(z)|=\mathcal{O}(k_{\rm max}^{(n-4)/2)})+\mathcal{O}(k_{\rm min}^{n/2})+\mathcal{O}(1/R), \quad z\in\R^n.
\end{equation*}
This indicates that the reconstruction resolution of $q$ can be improved by increasing 
$k_{\rm max}$ and decreasing $k_{\rm min}$. 
Nevertheless, $k_{\rm max}$ cannot be arbitrarily large as the Born approximation breaks down in the high-frequency regime.
Numerically, we consider $N_k$ admissible wavenumbers $\{k_m\}_{m=1}^{N_k}$ that are equidistantly distributed in the interval $[k_{\min}, k_{\max}]$, i.e.,
\begin{equation}\label{eq: discrete_wavenumber}
k_m=k_{\min}+(m-1)\Delta k,\quad m=1,2,\cdots, N_k.
\end{equation}
with step size $\Delta k=(k_{\max}-k_{\min})/(N_k-1)$. The step-by-step workflow of the proposed QSDM scheme is summarized in the following algorithm.
\begin{algorithm}
\caption{QDSM: direct sampling for inverse medium scattering with backscattering data}\label{alg:QDSM}
\begin{algorithmic}[1]
    \STATE Given a set of wavenumbers $\{k_m\}_{m=1}^{N_k}$ as defined in \eqref{eq: discrete_wavenumber}; \\
    \STATE Collect the multifrequency backscattering near- or far-field data due to an underlying contrast $q$;
    \STATE Select a sampling grid $\mathcal{G}$ in the detecting region $G$ such that ${\rm supp}{q}\subset G\subset\partial B_R$;
    \STATE For each sampling point $z\in\mathcal{G}$, evaluate $I^\infty_{\mathcal{K}}(z)$ or $I^s_{\mathcal{K}}(z)$ as the reconstruction of $q(z)$.
\end{algorithmic} 
\end{algorithm}

\subsection{Two-dimensional examples}
We first specify some more details regarding the data generation in 2D. We assume that the $N_\theta$ incident directions $\{\theta_j\}_{j=1}^{N_\theta}$ are uniformly distributed on $[0, 2\pi)$, that is, $\theta_j=(\cos t_j,\sin t_j)$ with incident angles $t_j= (j-1)\Delta\theta, j=1,2,\cdots,N_\theta$, where $\Delta\theta=2\pi/N_\theta$.  Under the above settings, the forward solver produces the noisy data sets
\begin{align*}
    & \text{Far-field data:}\quad \mathbf{M}_f=\{u_b^{\infty, \delta}(-\theta_j, \theta_j, k_m): j=1,2,\cdots,N_\theta,\, m=1,2,\cdots, N_k\},\\
    & \text{Near-field data:}\quad \mathbf{M}_s=\{u_b^{s, \delta}(-x_j, x_j, k_m): x_j=R\theta_j, j=1,2,\cdots,N_\theta,\, m=1,2,\cdots, N_k\}.
\end{align*}

Using these complex-valued $\mathbb{C}^{N_\theta\times N_k}$ measurement matrices $\mathbf{M}_f$ or $\mathbf{M}_s$ as input and noting that $\gamma_2(k)=\mathrm{e}^{\mathrm{i}\pi/4}/\sqrt{8 k\pi}$, the discrete forms of the indicator functions \eqref{I-infty} and \eqref{I-s} are correspondingly written as
$$
{\mathbf I}^{\infty}(z) = \frac{2(1-\mathrm{i})\Delta \theta\Delta k}{\pi^{3/2}}\sum_{m=1}^{N_k} k_m^{-1/2}\sum_{j=1}^{N_\theta} u_b^{\infty, \delta}(-\theta_j, \theta_j, k_m)e^{-2{\rm i}k_m \theta_j\cdot z},
$$
and
$$
{\mathbf I}^s(z) =-\frac{8\mathrm{i}R\Delta\theta\Delta k}{\pi}\sum\limits_{m=1}^{N_k}\sum\limits_{j=1}^{N_\theta}u_b^{s,\delta}(-x_j, x_j, k_m)e^{2{\rm i}k_m (\theta_j\cdot z-R)}.
$$
Note that, for convenience and without ambiguity, we drop the subscript $\mathcal{K}$ in the indicator functions. For the near-field case, we choose $R=5$ in the following examples unless otherwise stated. In all 2D examples, the sampling indicators are calculated on a uniformly distributed grid of size $201\times 201$ over the sampling region $[-0.7, 0.7]$. 

We would like to emphasize that the formulations of the imaging indicators do not rely on any simulation process of the forward problem in \cite{buergel2019ipscatt}, hence the {\it inverse crime} is inherently avoided by the proposed QDSM. 


\begin{example}\label{exple: 2D-1}
In the first example, we consider a predefined contrast function ``cornerBallSparse2D'' in \textbf{IPscatt} with a minor modification of rescaling the original contract value $q$ to $10^{-2}q$. This benchmark contrast is composed of a non-constant $L$-shaped corner, a ball (filled circle) and a broken corner consisting of squares and a thin line. We refer to Figure \ref{fig: 2D-example1-exact} for a 2D view of the exact contrast function. The parameters $N_k=60, k_{\rm min}=1, k_{\rm max}=121$ are used in this example. 
\begin{figure}
    \centering
    \includegraphics[width=0.3\linewidth]{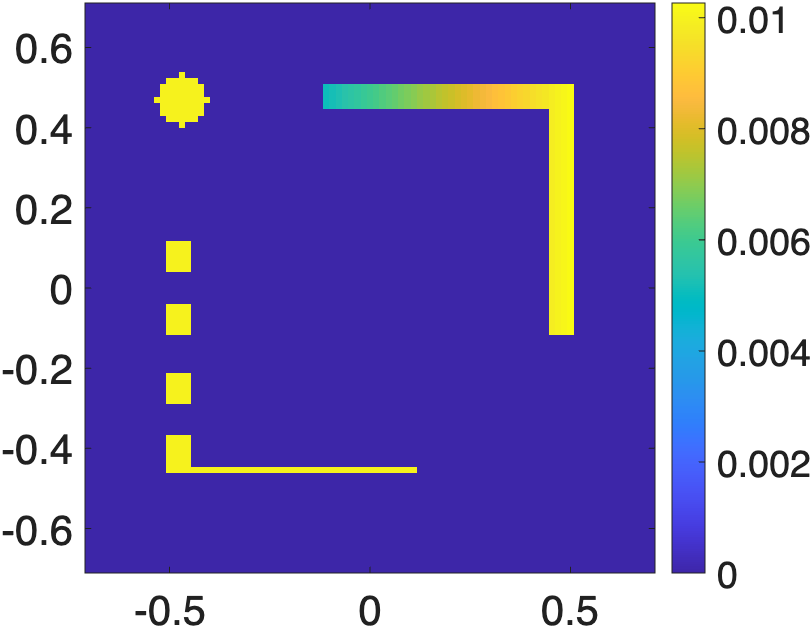}
    \caption{The ground-truth contrast of Example \ref{exple: 2D-1}.}
    \label{fig: 2D-example1-exact}
\end{figure}

\begin{figure}
    \centering
    \subfigure[$N_\theta=8$]{\includegraphics[width=0.3\linewidth]{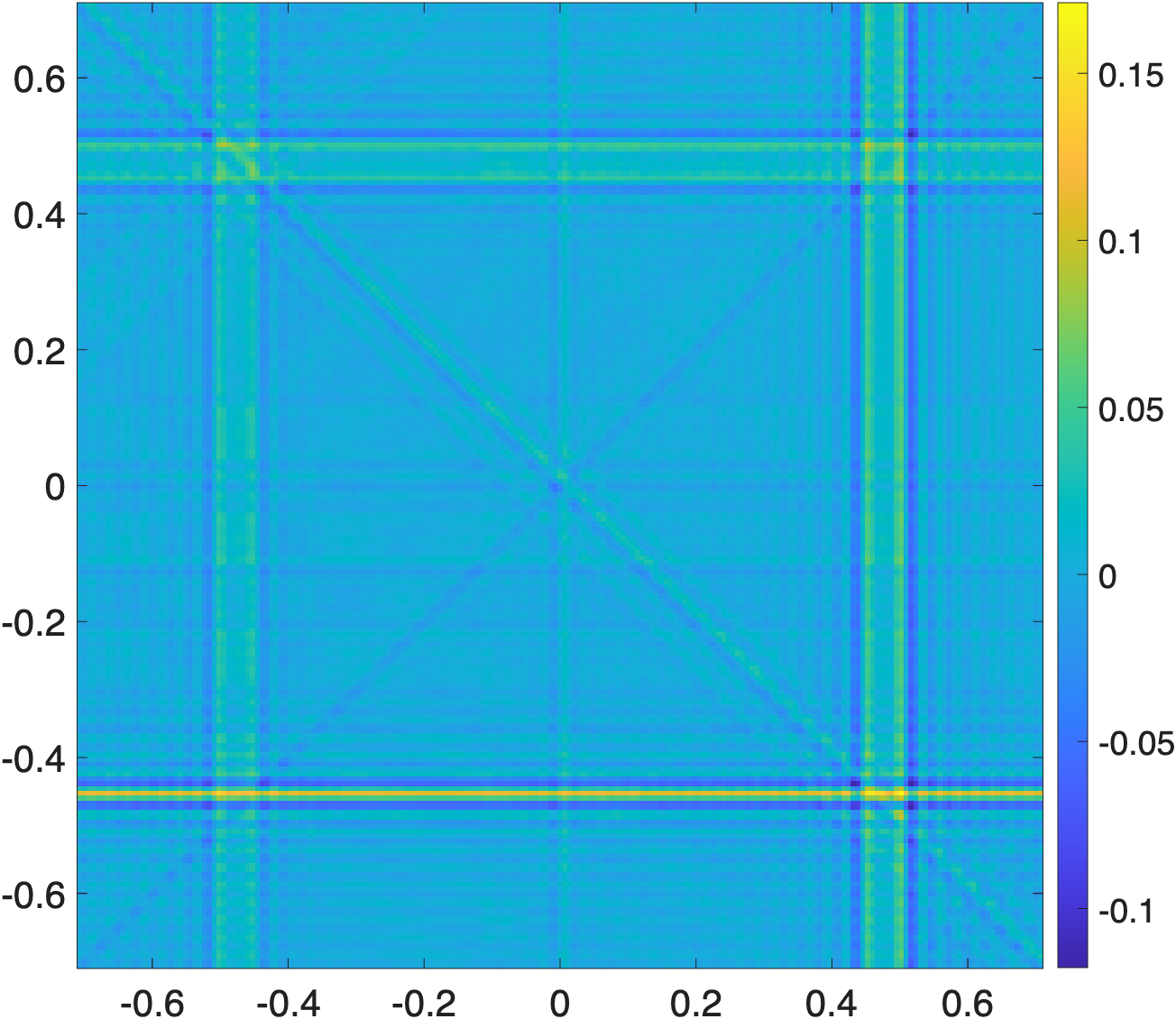}}\quad
    \subfigure[$N_\theta=16$]{\includegraphics[width=0.3\linewidth]{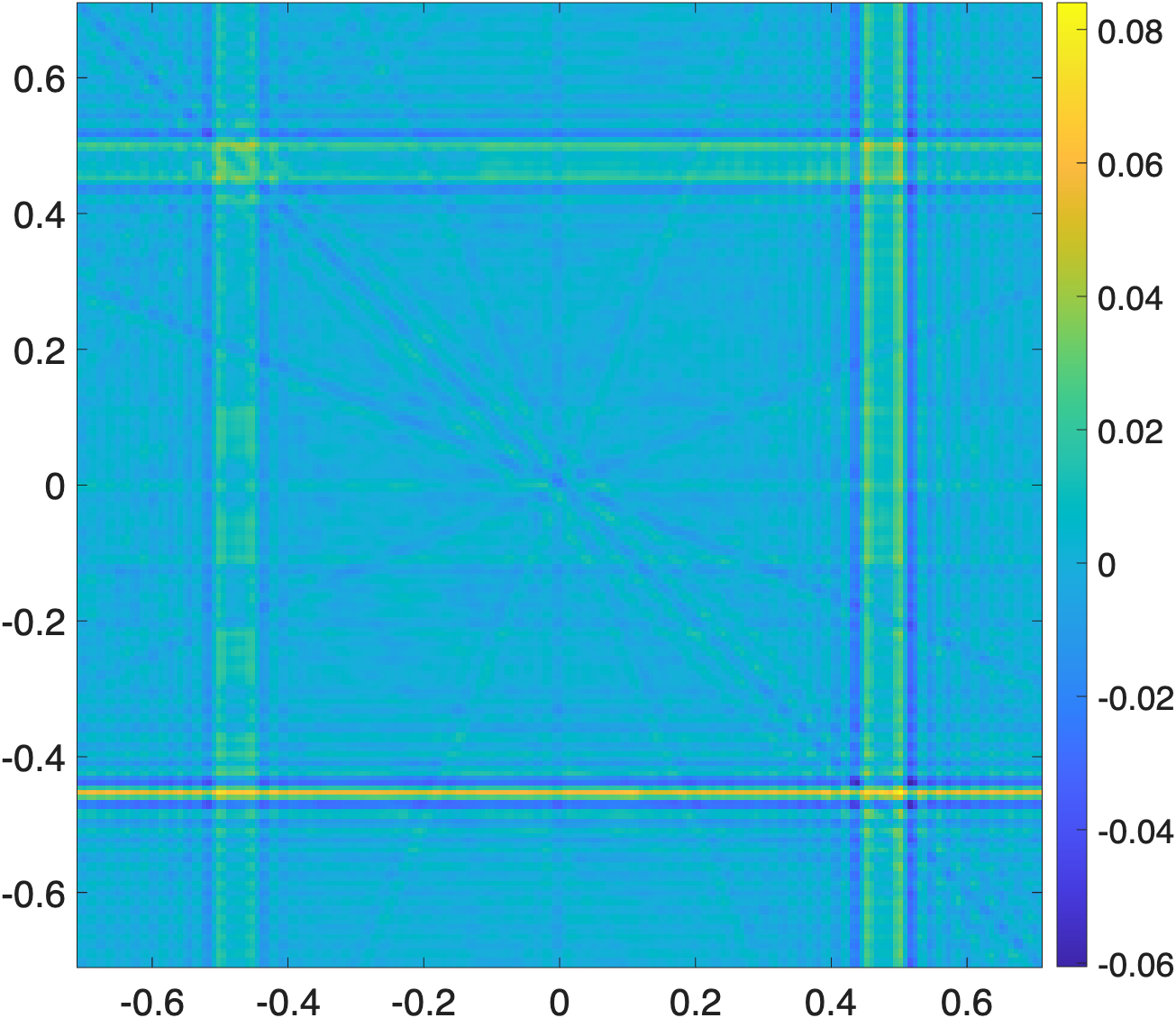}}\quad
    \subfigure[$N_\theta=32$]{\includegraphics[width=0.3\linewidth]{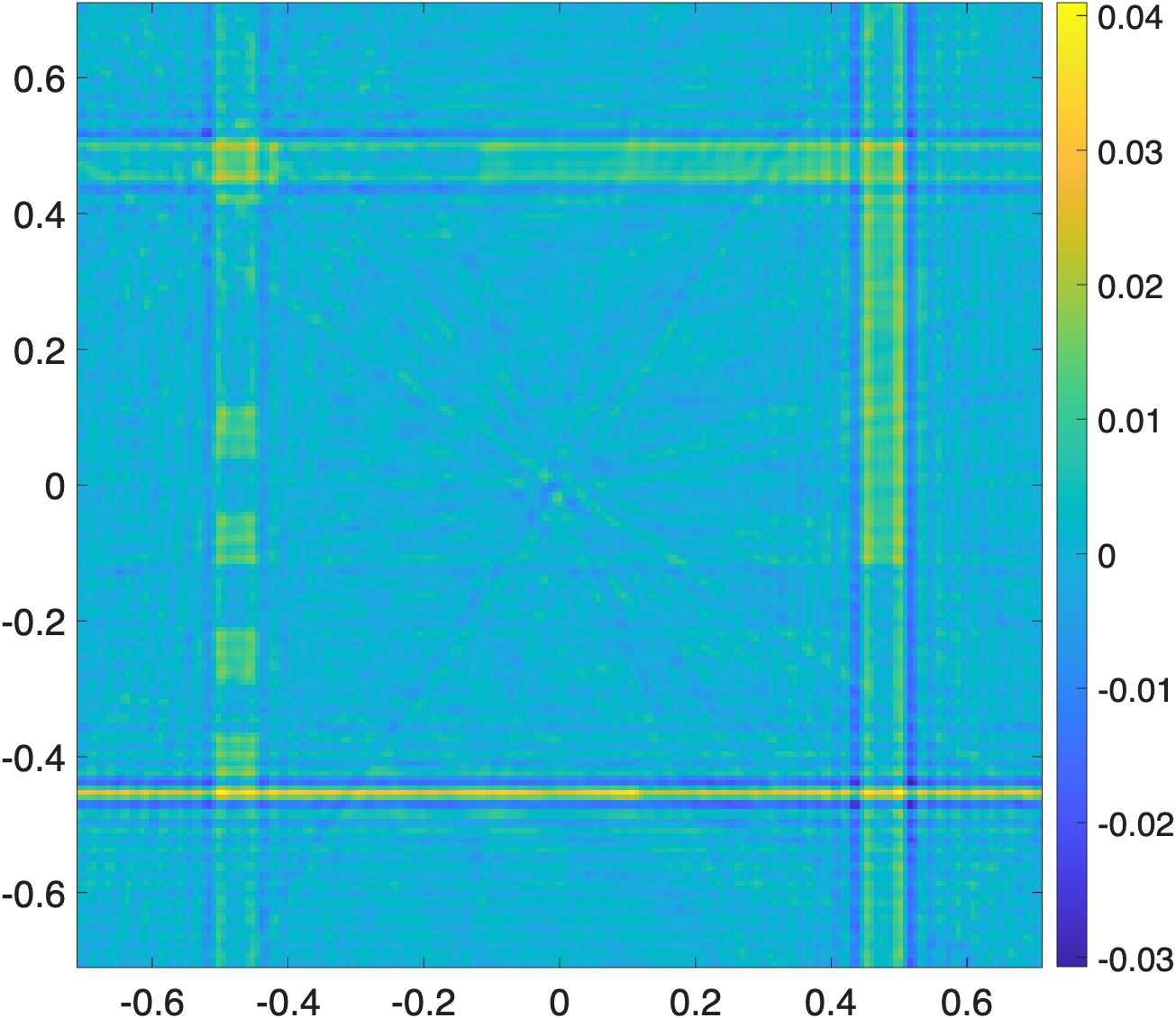}}
    \subfigure[$N_\theta=64$]{\includegraphics[width=0.3\linewidth]{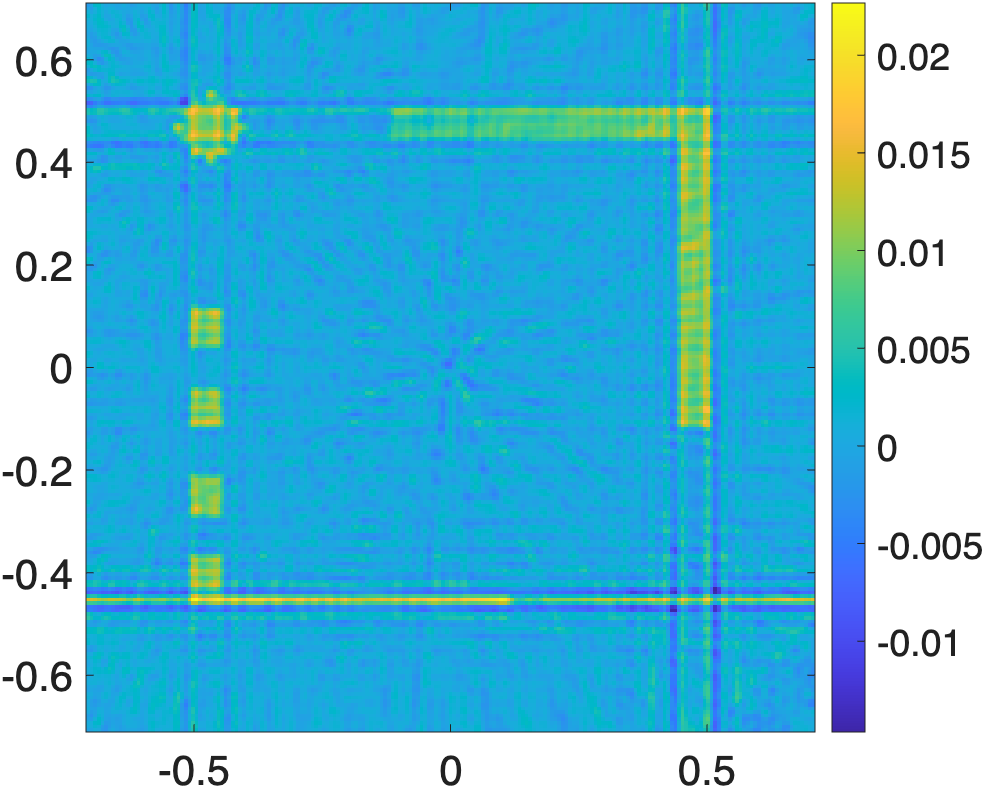}}\quad
    \subfigure[$N_\theta=128$]{\includegraphics[width=0.3\linewidth]{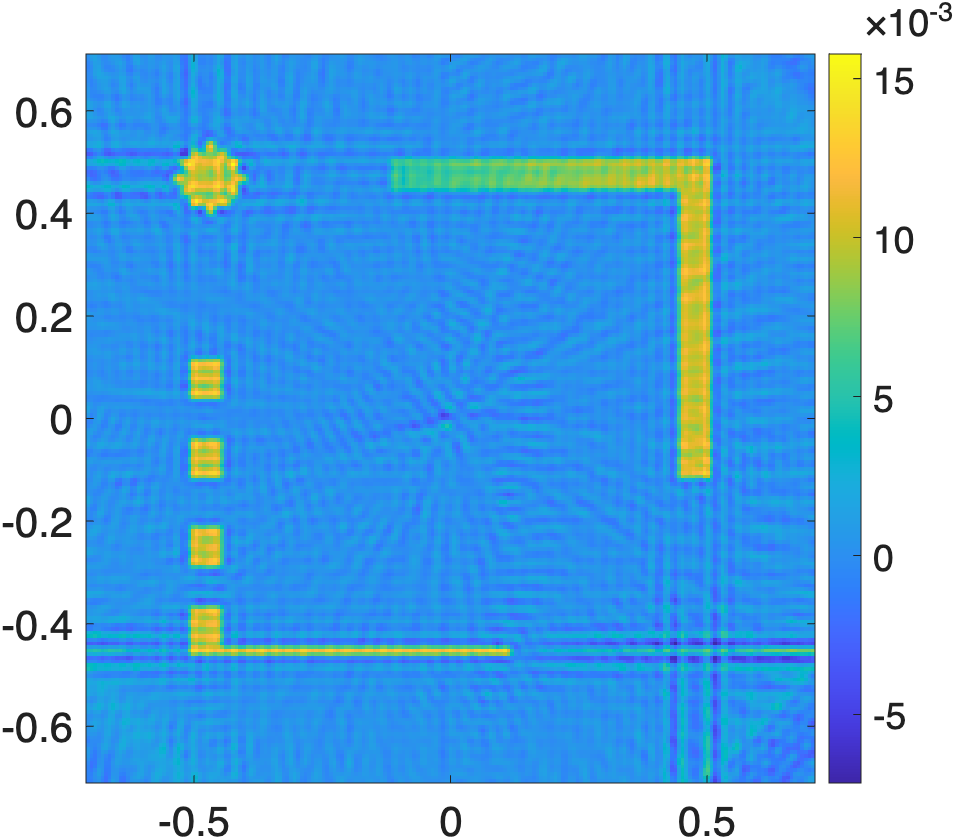}}\quad
    \subfigure[$N_\theta=256$]{\includegraphics[width=0.3\linewidth]{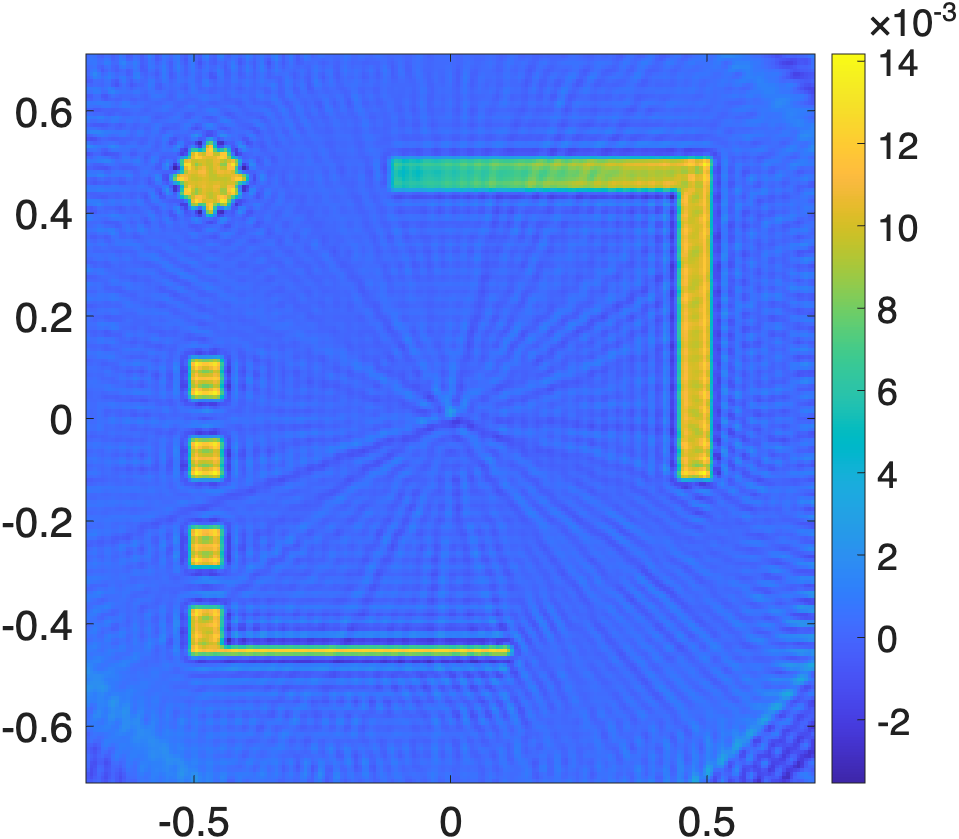}}
    \caption{Imaging of $\Re({\mathbf I}^{\infty})$ with different number of incident directions $N_\theta$.}
    \label{fig: 2D-example1-farField}
\end{figure}

The simulation results for the variance of $N_\theta$ are shown in Figure \ref{fig: 2D-example1-farField}. From Figure \ref{fig: 2D-example1-farField} we find that the number of incident directions significantly account for the accuracy of imaging output ${\mathbf I}^{\infty}$. Analogously, such behaviors can also be observed for the near-field indicator ${\mathbf I}^s$. Therefore, to obtain a satisfactory reconstruction, one should choose a sufficiently large $N_\theta$ to ensure the desired approximation of discrete and continuous indicators. 

\begin{figure}
    \centering
    \subfigure[$R=5$ ]{\includegraphics[width=0.3\linewidth]{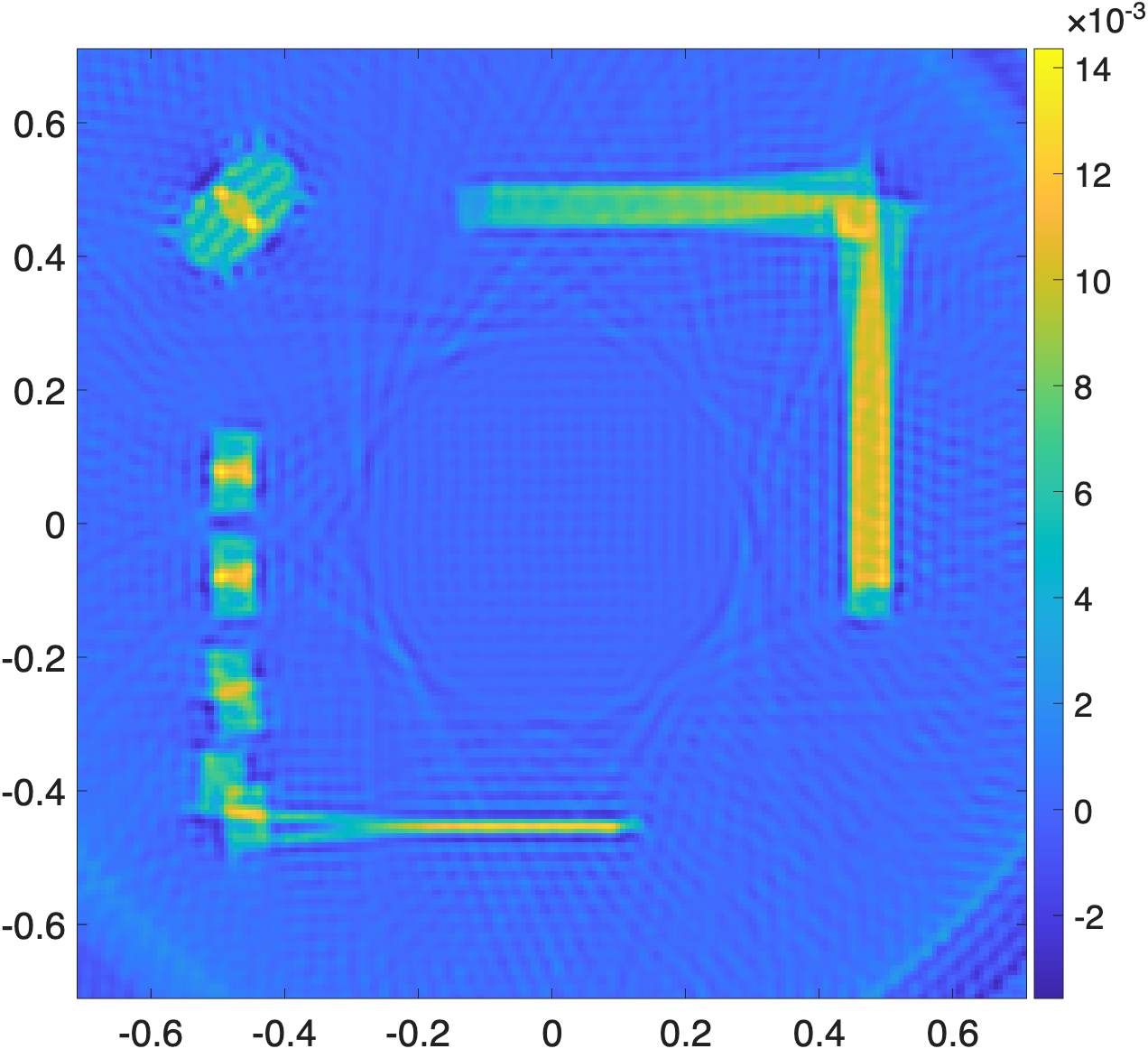}}\quad
    \subfigure[$R=10$]{\includegraphics[width=0.3\linewidth]{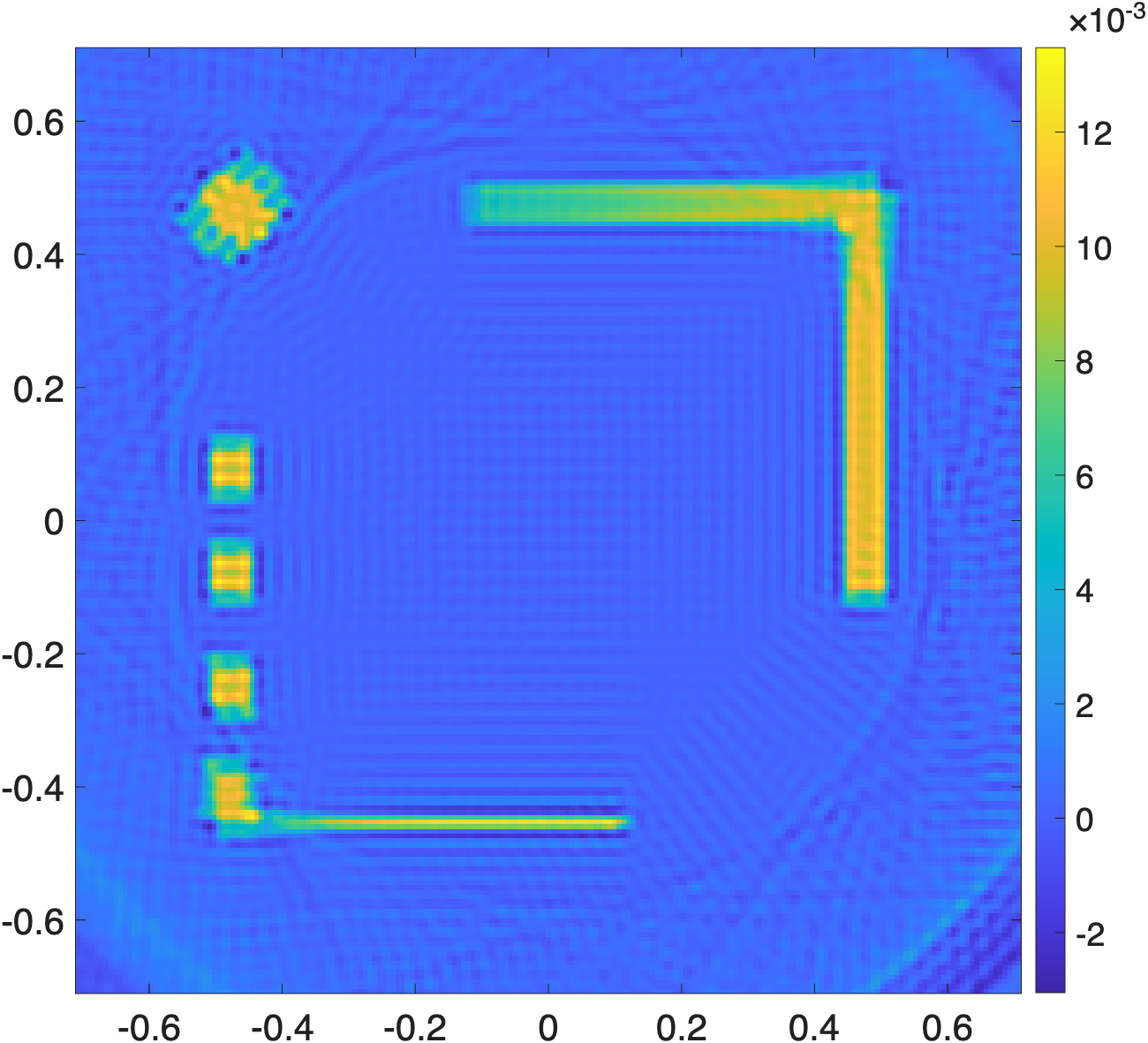}}\quad
    \subfigure[$R=20$]{\includegraphics[width=0.3\linewidth]{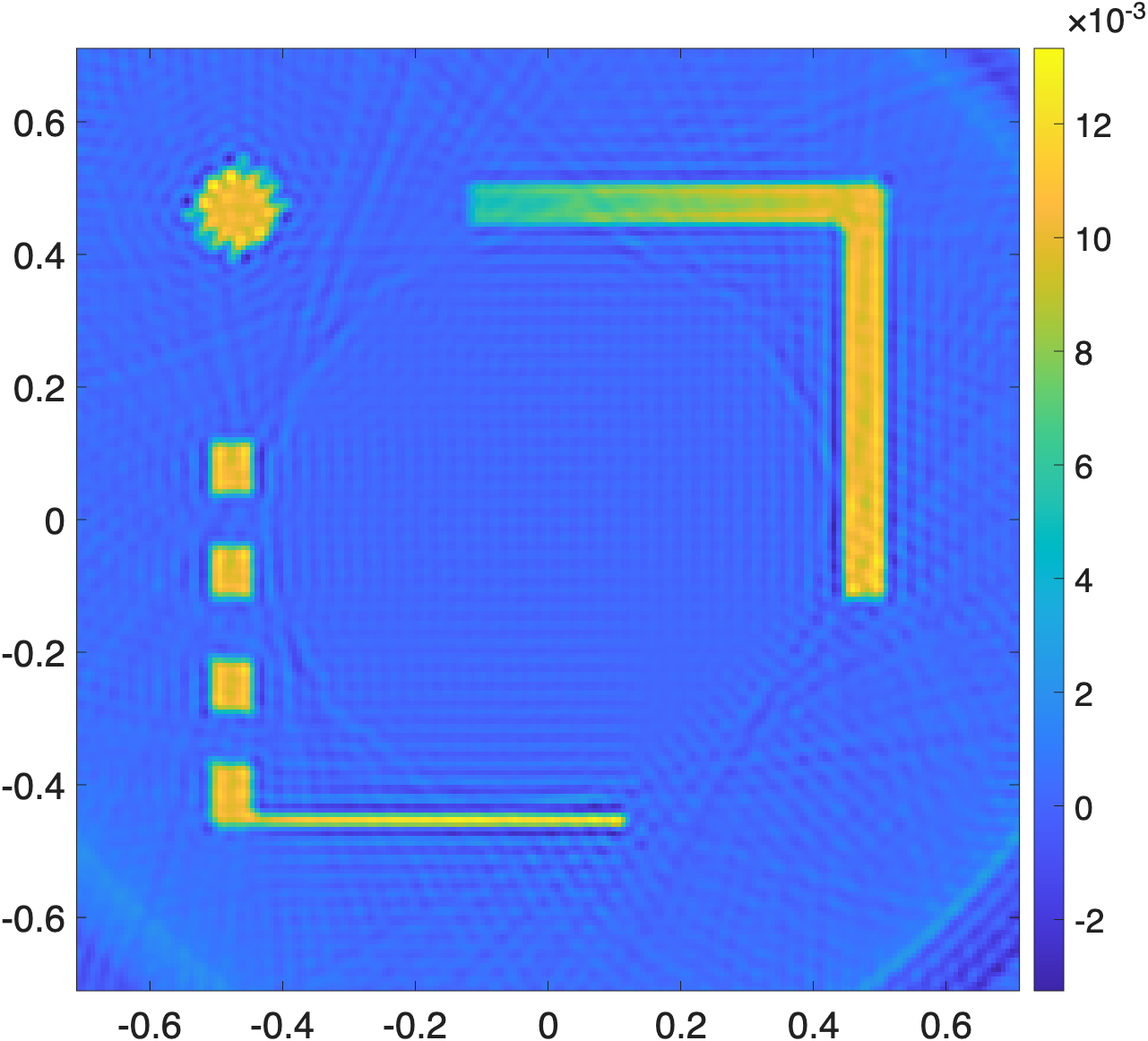}}\quad
    \caption{Imaging of $\Re({\mathbf I}^s)$ with $N_\theta=256$ and different radii $R$.}
    \label{fig: 2D-example1-nearField}
\end{figure}

In Figure \ref{fig: 2D-example1-nearField} we test the sampling results of $\mathrm{Re}({\mathbf I}^s)$ with respect to the radius $R$ of the measurement circle. Figure \ref{fig: 2D-example1-nearField} shows that the sampling quality improves with respect to the increase of $R$. These results are well consistent with the theoretical finding in Theorem \ref{thm: near-field approximation}, which requires a moderately large $R$ to suppress the discrepancy between $I^s$ and the exact contrast $q$.
\end{example}

\begin{example}\label{exple: 2D-2}
    The second example is devoted to reconstructing the contrast of a Shepp-Logan phantom, see Figure \ref{fig: 2D-example2-exact} for an image of the ground-truth contrast function. In this experiment, we fix $N_\theta=256$ and test the sampling results with respect to some key parameters.
\begin{figure}
    \centering
    \includegraphics[width=0.3\linewidth]{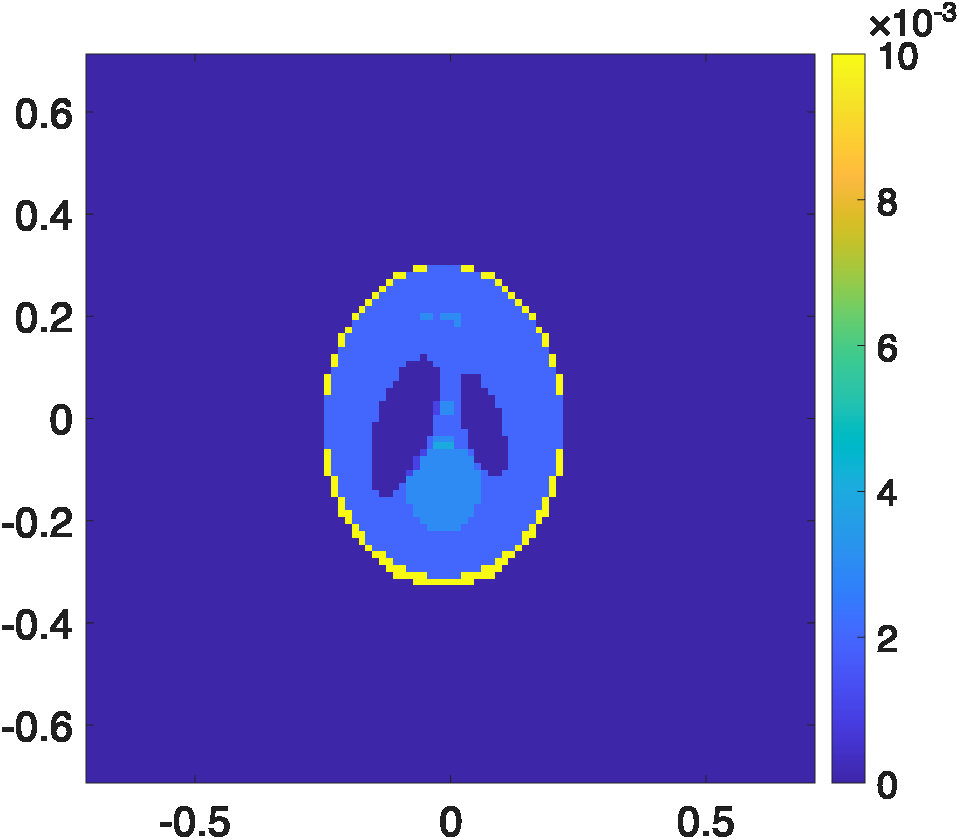}
    \caption{The ground-truth contrast of Example \ref{exple: 2D-2}.}
    \label{fig: 2D-example2-exact}
\end{figure}

\begin{figure}
    \centering
    \subfigure[$\Delta k=2, k_{\rm min}=1, k_{\rm max}=121$]{\includegraphics[width=0.3\linewidth]{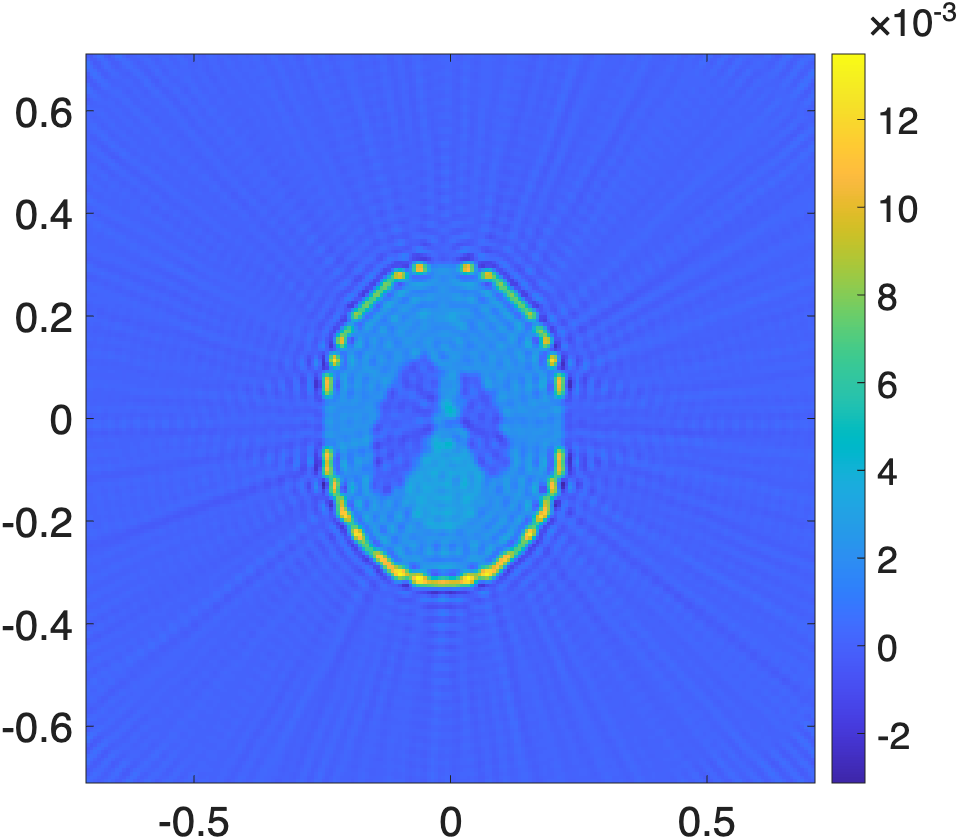}}\quad
    \subfigure[$\Delta k=2, k_{\rm min}=1, k_{\rm max}=61$]{\includegraphics[width=0.3\linewidth]{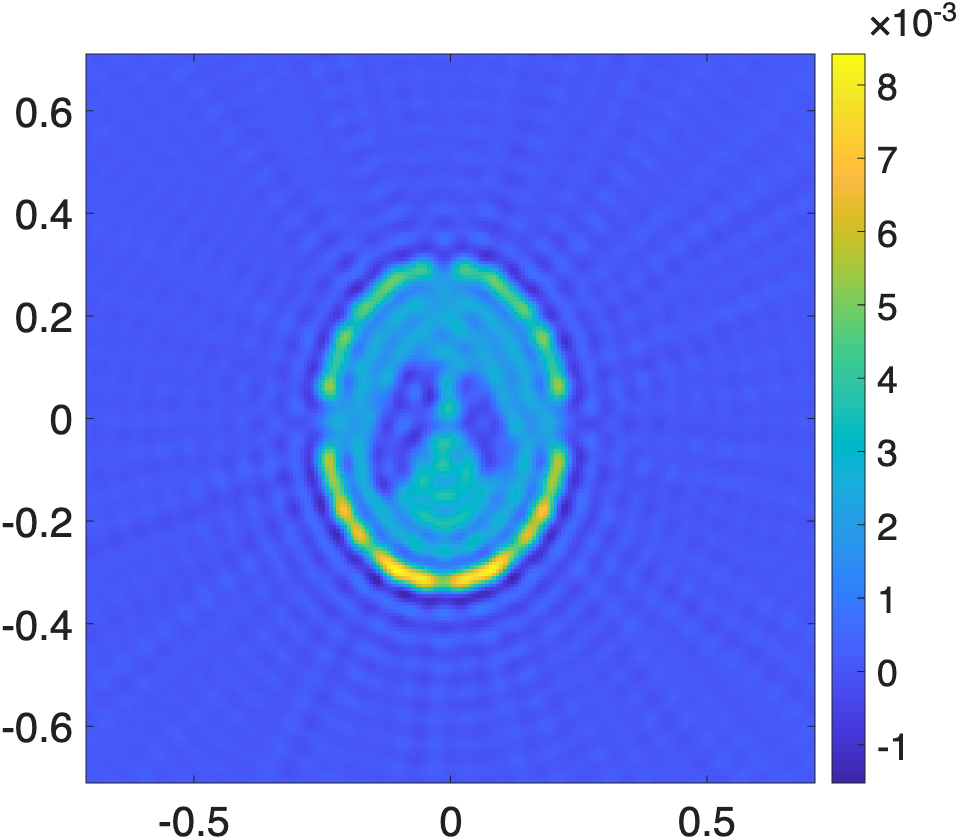}}\quad
    \subfigure[$\Delta k=2, k_{\rm min}=61, k_{\rm max}=121$]{\includegraphics[width=0.3\linewidth]{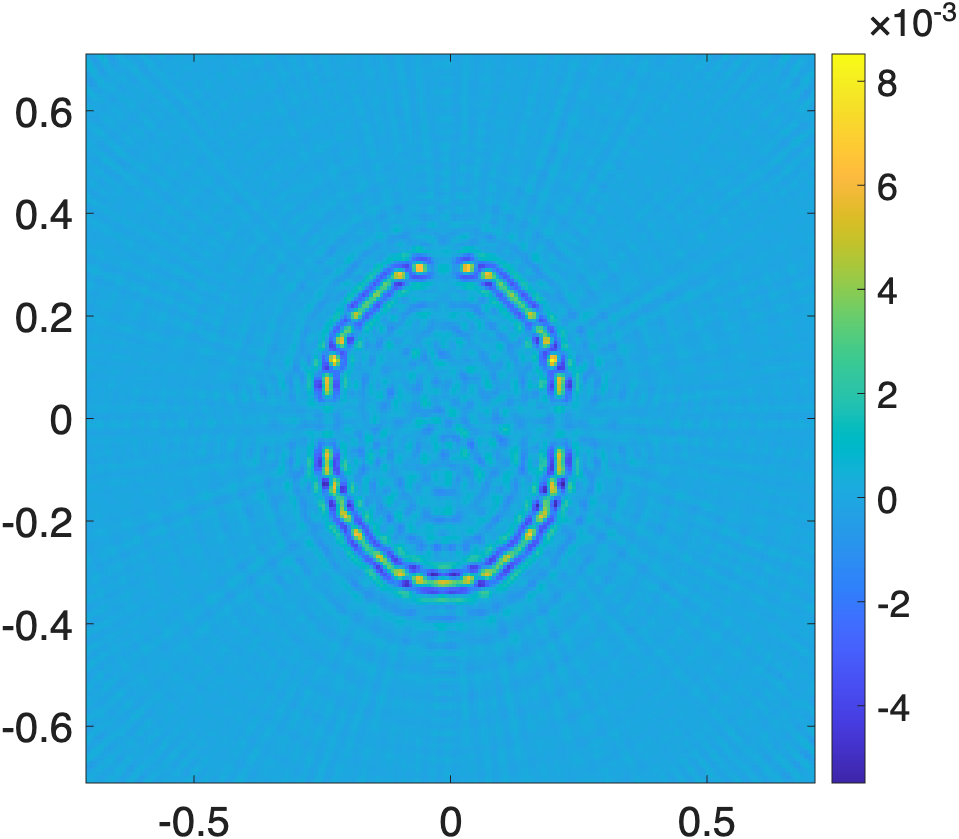}}\\
    \subfigure[$\Delta k=0.5, k_{\rm min}=1, k_{\rm max}=121$]{\includegraphics[width=0.3\linewidth]{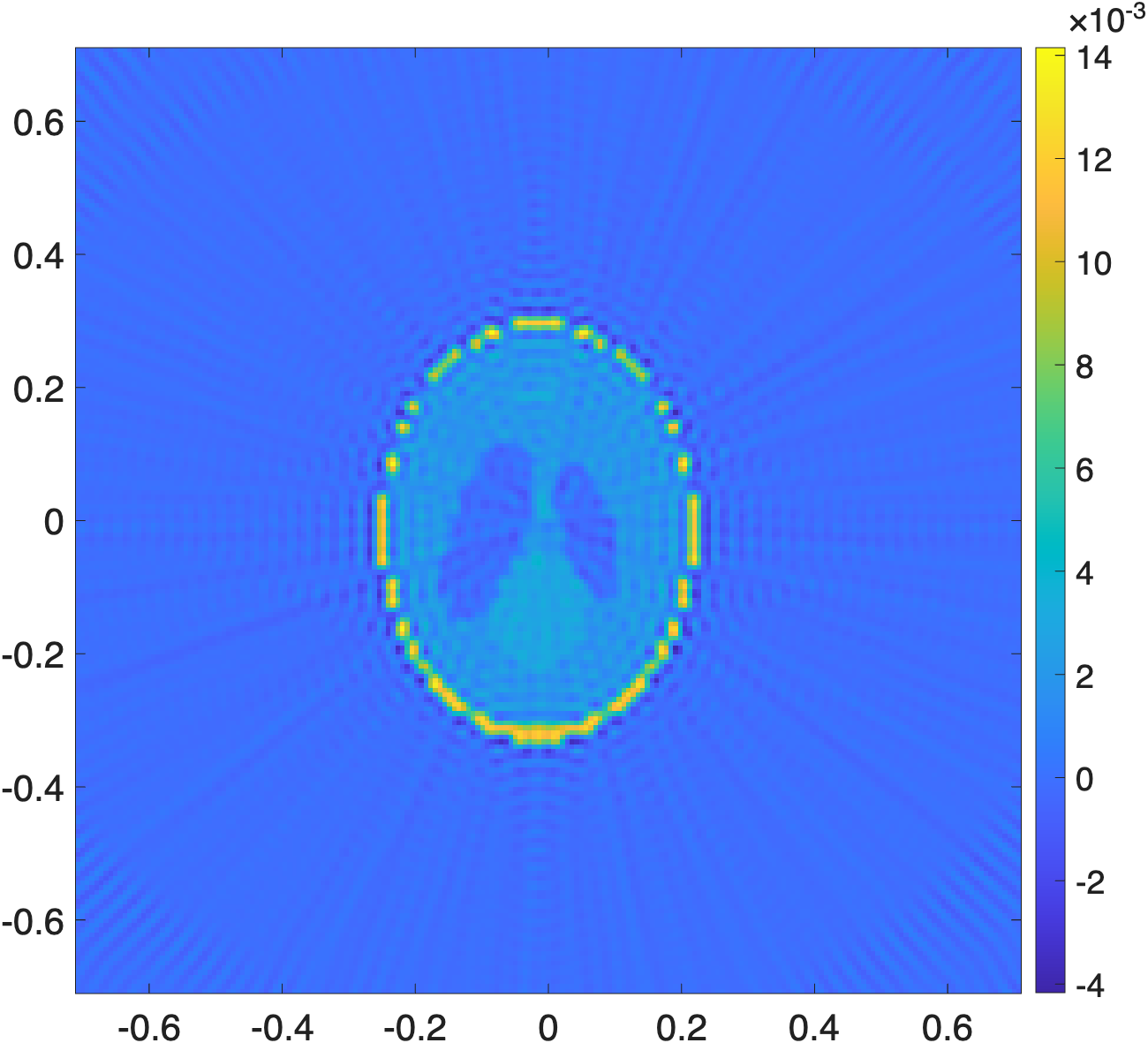}}\quad
    \subfigure[$\Delta k=0.5, k_{\rm min}=1, k_{\rm max}=61$]{\includegraphics[width=0.3\linewidth]{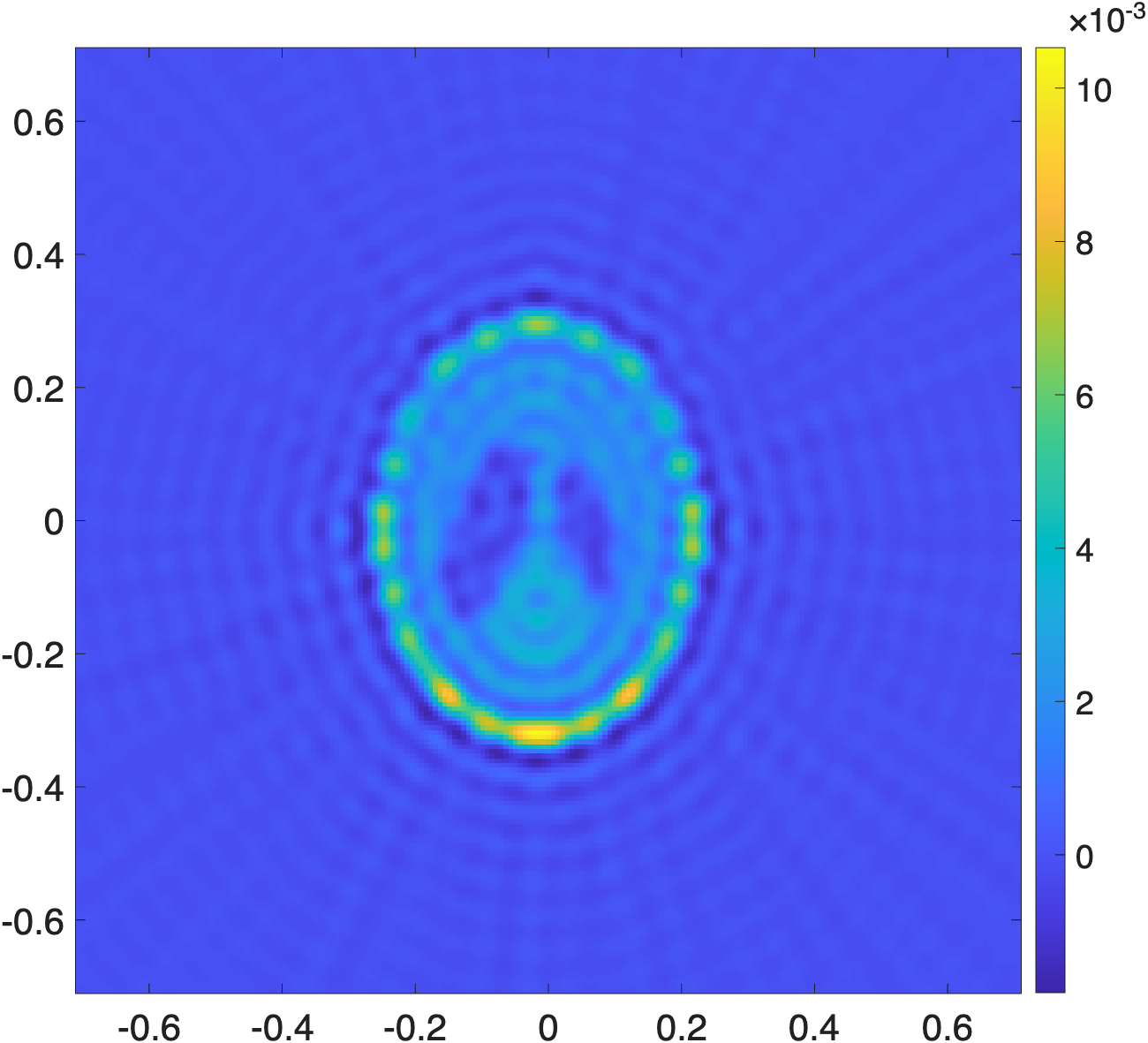}}\quad
    \subfigure[$\Delta k=0.5, k_{\rm min}=61, k_{\rm max}=121$]{\includegraphics[width=0.3\linewidth]{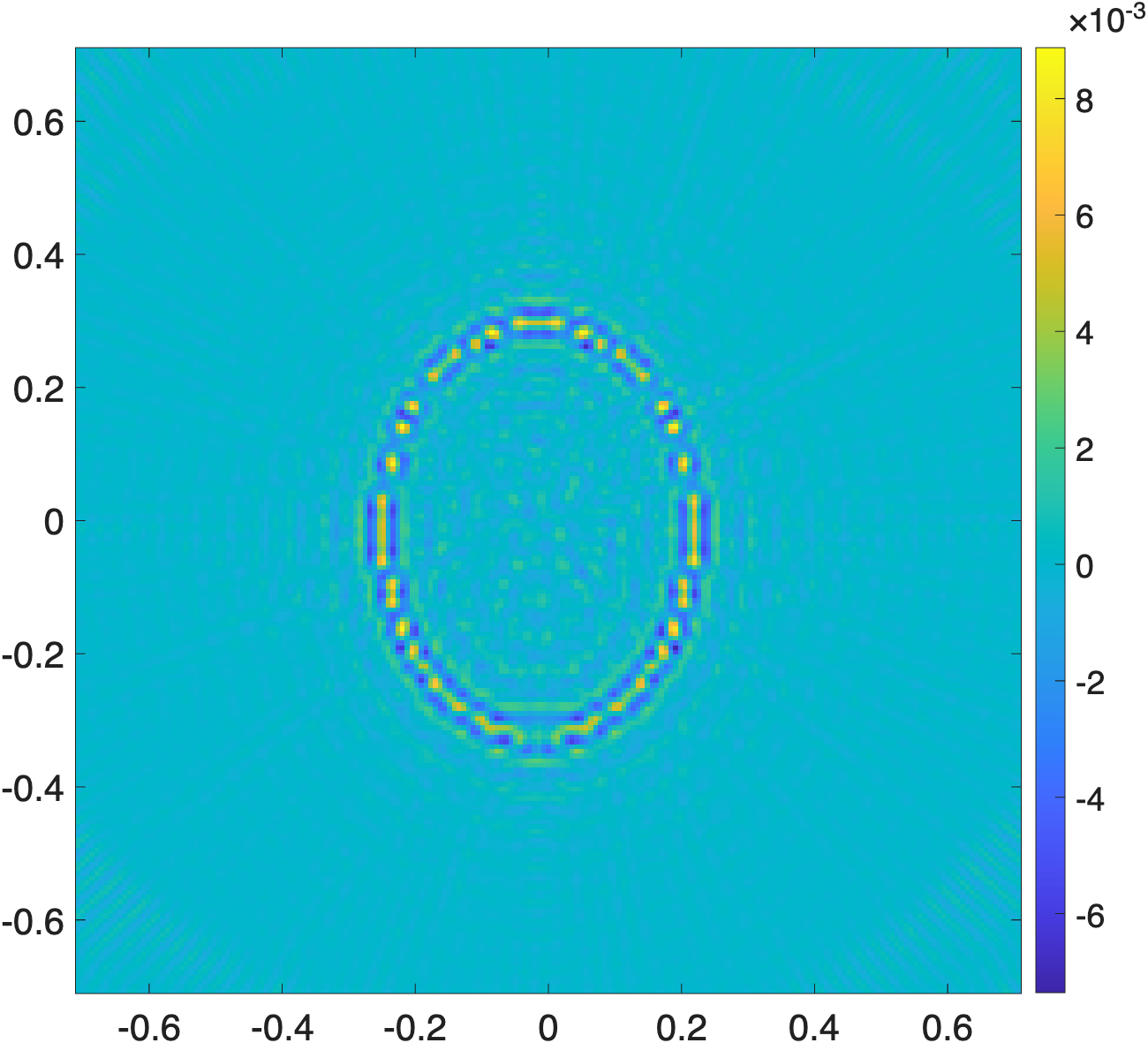}}
    \caption{Imaging of $\Re({\mathbf I}^{\infty}(z))$ with different choices of wavenumbers.}
    \label{fig: 2D-example2-farField}
\end{figure}

Figure \ref{fig: 2D-example2-farField} demonstrate how the range of frequency spectrum $(k_{\rm min}, k_{\rm max})$ and the sampling density of the wavenumbers $\Delta k$ would influence the quality of reconstructions. Specifically, if we choose a relatively large wavenumber spectrum (from $k_{\rm min}=1$ to $k_{\rm max}=121$ in Figure \ref{fig: 2D-example2-farField}(a)(d)), then a satisfactory reconstruction can be achieved. Moreover, the result with more wavenumbers (Figure \ref{fig: 2D-example2-farField}(d)) outperforms the one with fewer wavenumbers (Figure \ref{fig: 2D-example2-farField}(a)). When only low-frequency information is used ($k_{\rm max}=61$ in Figure \ref{fig: 2D-example2-farField}(b)(e)), the fine details in the reconstructions are unavailable. In contrast, if we only resort to high-frequency component of the data (from $k_{\rm min}=61$ to $k_{\rm max}=121$ in Figure \ref{fig: 2D-example2-farField}(c)(f)), the error due to the Born approximation dominates and the imaging results also significantly deteriorates. The similar results can be found in the near-field case, see Figure \ref{fig: 2D-example2-nearField}. Therefore, a sufficiently large number of discrete wavenumbers and an appropriate range of the frequency spectrum play an important role in the reconstructions.

\begin{figure}
    \centering
    \subfigure[$\Delta k=2, k_{\rm min}=1, k_{\rm max}=121$]{\includegraphics[width=0.3\linewidth]{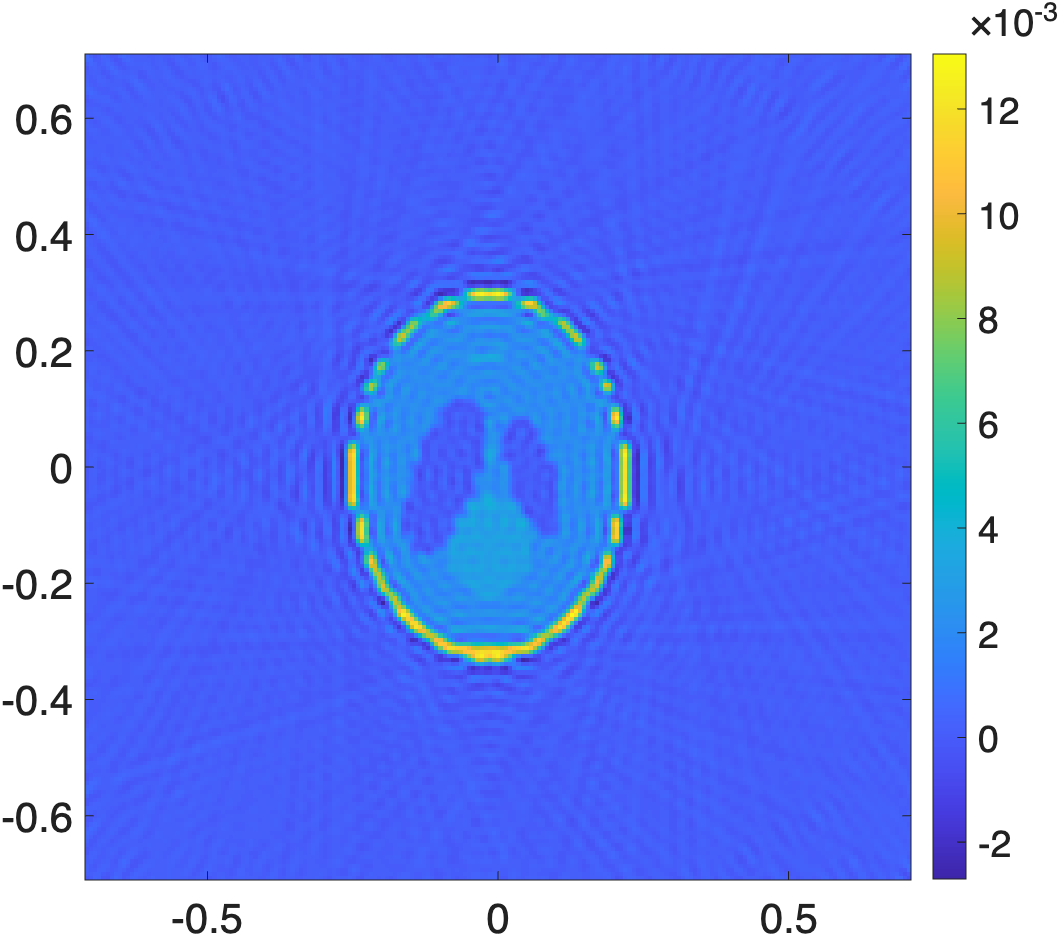}}\quad
    \subfigure[$\Delta k=2, k_{\rm min}=1, k_{\rm max}=61$]{\includegraphics[width=0.3\linewidth]{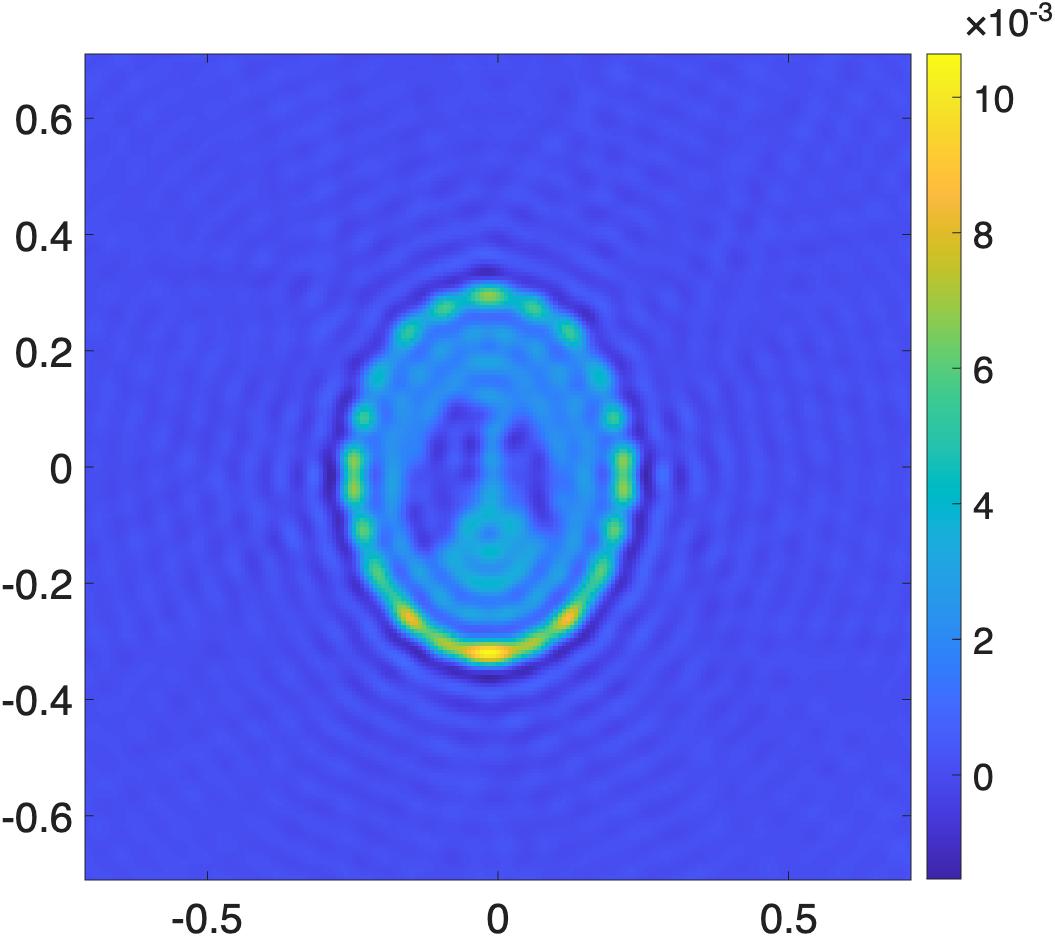}}\quad
    \subfigure[$\Delta k=2, k_{\rm min}=61, k_{\rm max}=121$]{\includegraphics[width=0.3\linewidth]{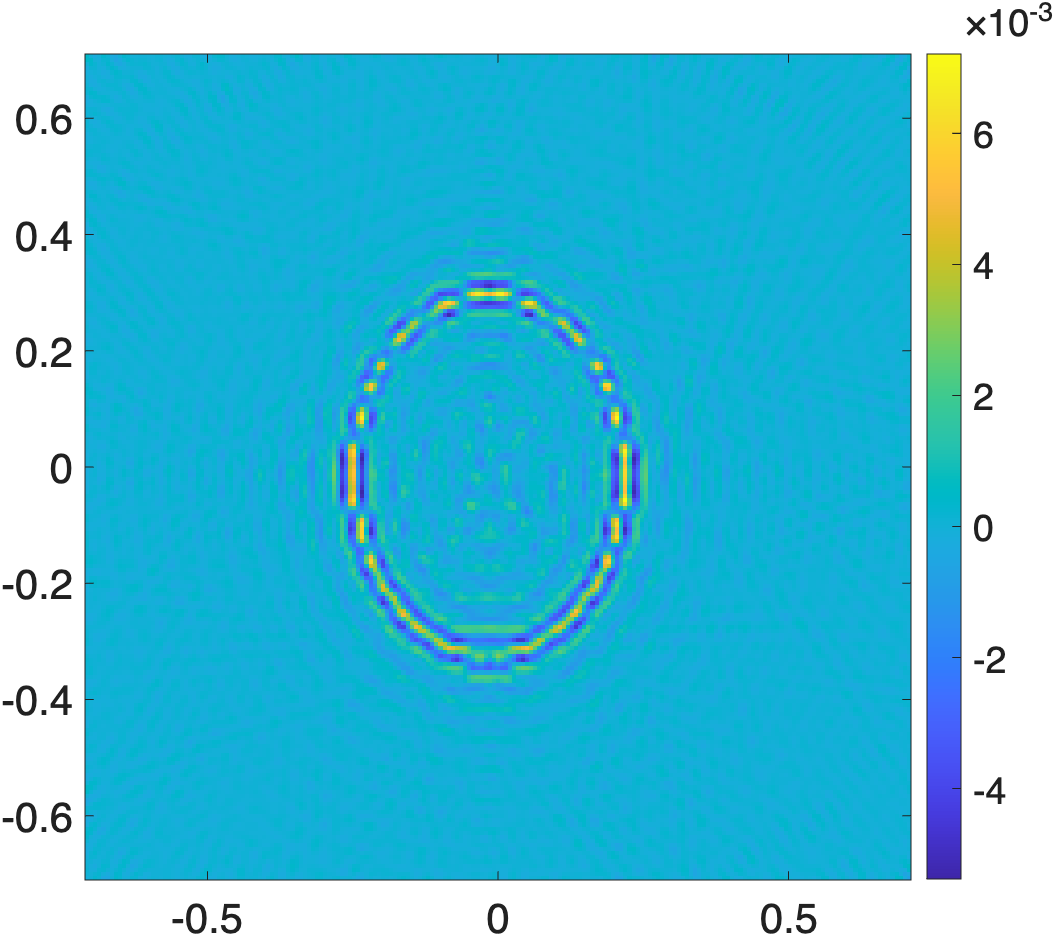}}\\
    \subfigure[$\Delta k=0.5, k_{\rm min}=1, k_{\rm max}=121$]{\includegraphics[width=0.3\linewidth]{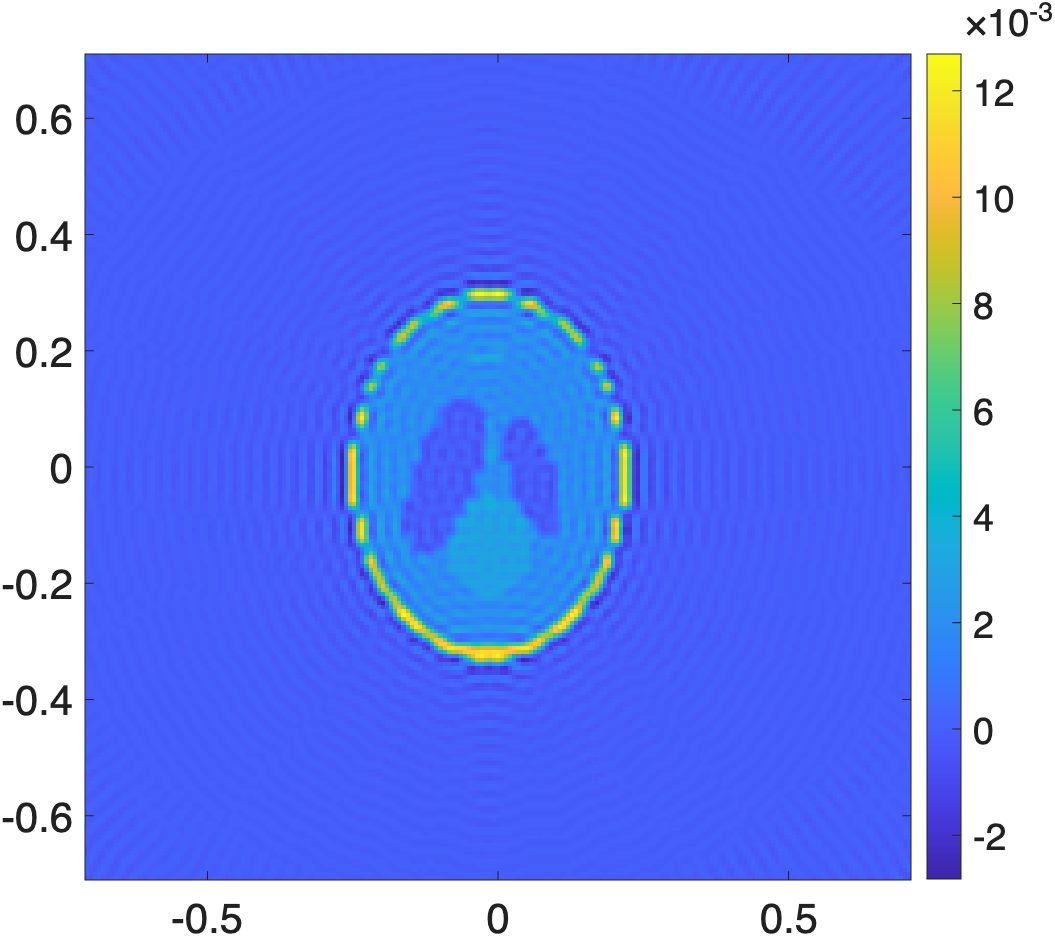}}\quad
    \subfigure[$\Delta k=0.5, k_{\rm min}=1, k_{\rm max}=61$]{\includegraphics[width=0.3\linewidth]{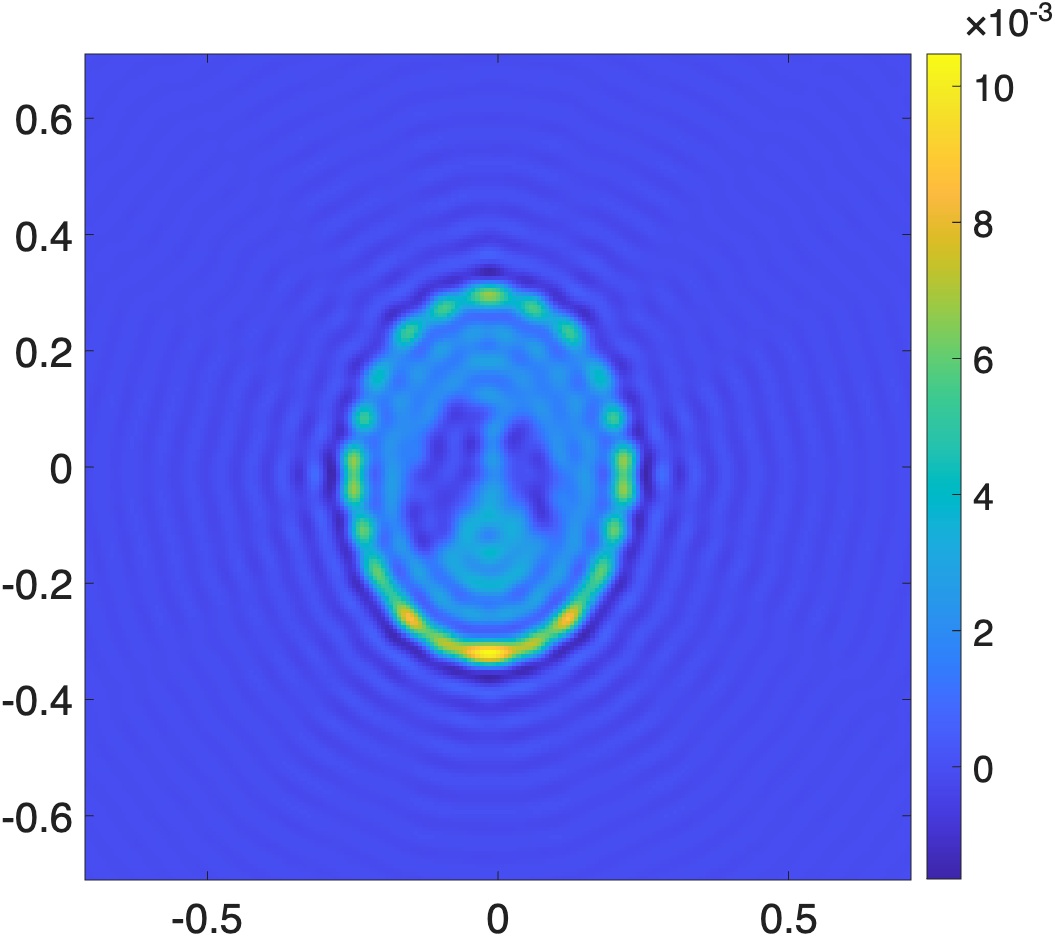}}\quad
    \subfigure[$\Delta k=0.5, k_{\rm min}=61, k_{\rm max}=121$]{\includegraphics[width=0.3\linewidth]{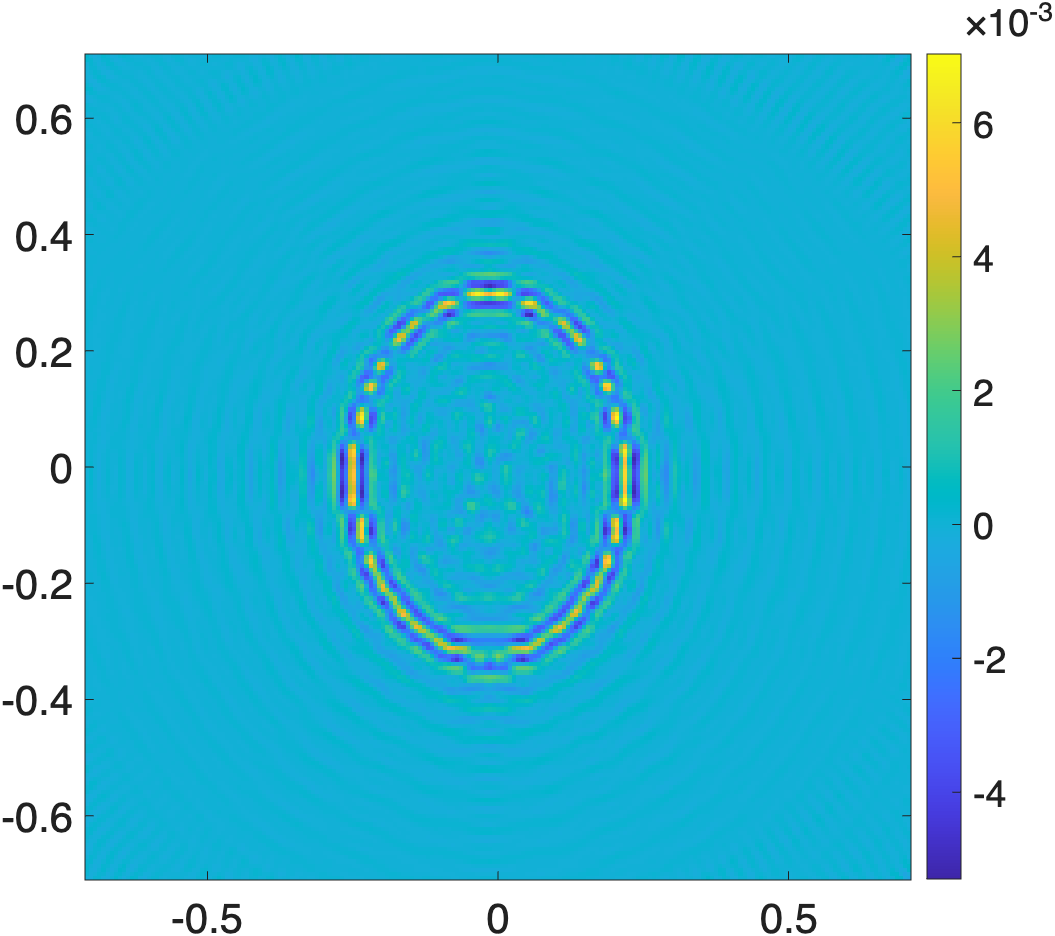}}
    \caption{Imaging of $\Re({\mathbf I}^s(z))$ with $R=5$ and different choices of wavenumbers.}
    \label{fig: 2D-example2-nearField}
\end{figure}


\end{example}

\begin{example}\label{exple: 2D-3}
The third 2D example is designed to demonstrate the algorithm's performance in recovering a complex-valued contrast function. In the previous examples, we assume it is known a priori that the ground-truth contrast is real-valued. If this a priori information $\Im(q)=0$ is available, only the real part of the indicators are required to be plotted and the imaginary part is ignored. Now we shall consider a more complicated case that the contrast function is a smooth complex function with a possibly non-positive real part. Here the ground-truth contrast is given by 
\begin{equation}\label{eq: complex_q}
q(x)=q(x_1, x_2)=\Re(q(x_1, x_2))+\mathrm{i}\Im(q(x_1, x_2)),\quad x=(x_1, x_2)\in\mathbb{R}^2,
\end{equation}
where
\begin{align*}
\Re(q(x_1, x_2))&=1.1\times 10^{-2}e^{-200[(x_1-0.01)^2+(x_2-0.12)^2]}-(x_2^2-x_1^2)e^{-90(x_1^2+x_2^2)},\\
\Im(q(x_1, x_2))&=10^{-2}\Big\{0.9e^{-100[(x_1-0.2)^2+(x_2-0.2)^2]}+1.1e^{-250[(x_1+0.15)^2+(x_2-0.15)^2]}\\
&\quad+1.3e^{-150[(x_1+0.2)^2+2(x_2+0.2)^2]}+e^{-50[(x_1-0.25)^2+x_2^2]}\Big\}.
\end{align*}

We refer to Figure \ref{fig: 2D-example3-exact} for a depiction of the contrast function $q$ defined in \eqref{eq: complex_q}. The parameters $N_\theta=256, \Delta k=2, k_{\rm min}=1, k_{\rm max}=99$ are used in this experiment.

\begin{figure}
    \centering
    \subfigure[$\Re(q)$]{\includegraphics[width=0.3\linewidth]{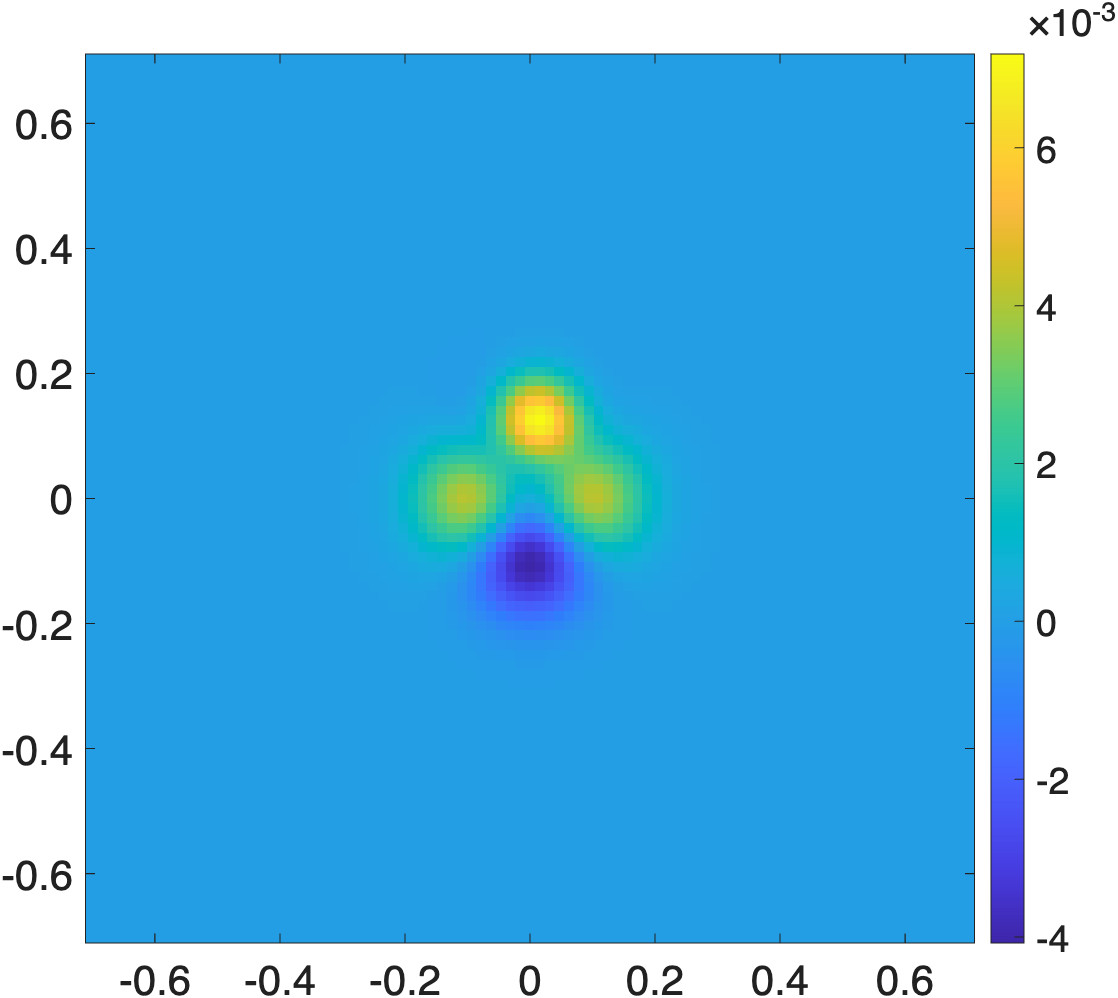}}\quad
     \subfigure[$\Im(q)$]{\includegraphics[width=0.3\linewidth]{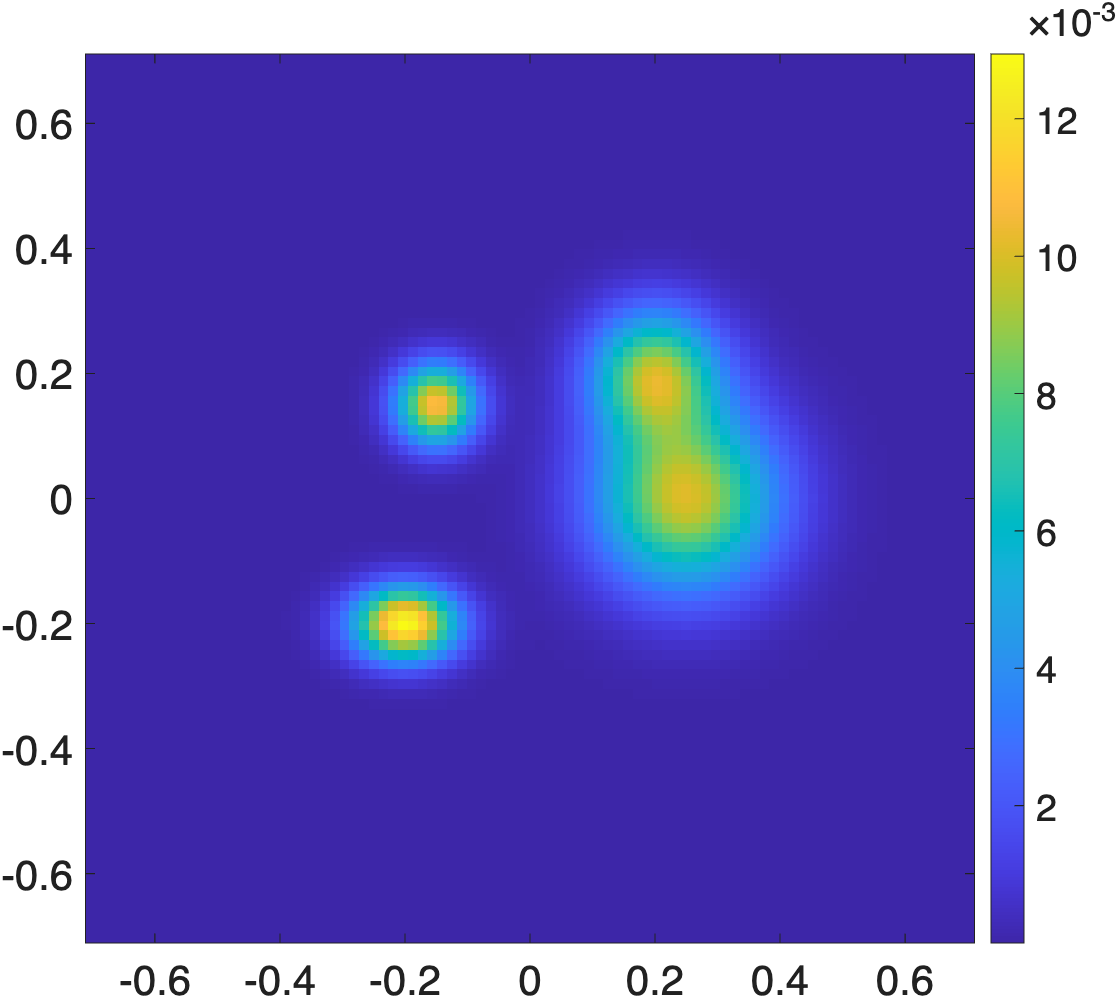}}
     \caption{The ground-truth contrast $q$ of Example \ref{exple: 2D-3}.}
    \label{fig: 2D-example3-exact}
\end{figure}

\begin{figure}
    \centering
    \subfigure[$\Re({\mathbf I}^\infty),\ \delta=1\%$]{\includegraphics[width=0.3\linewidth]{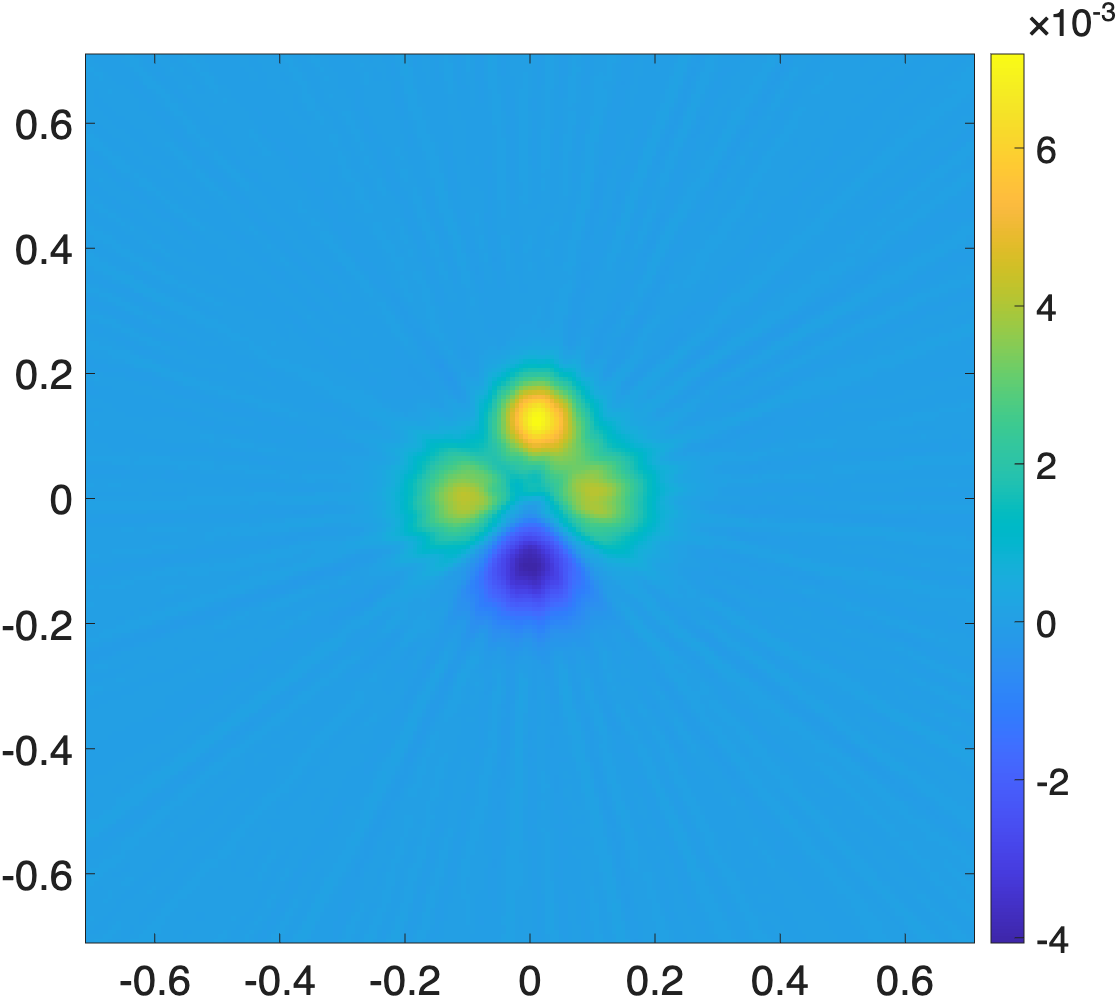}}\quad
    \subfigure[$\Re({\mathbf I}^\infty),\ \delta=5\%$]{\includegraphics[width=0.3\linewidth]{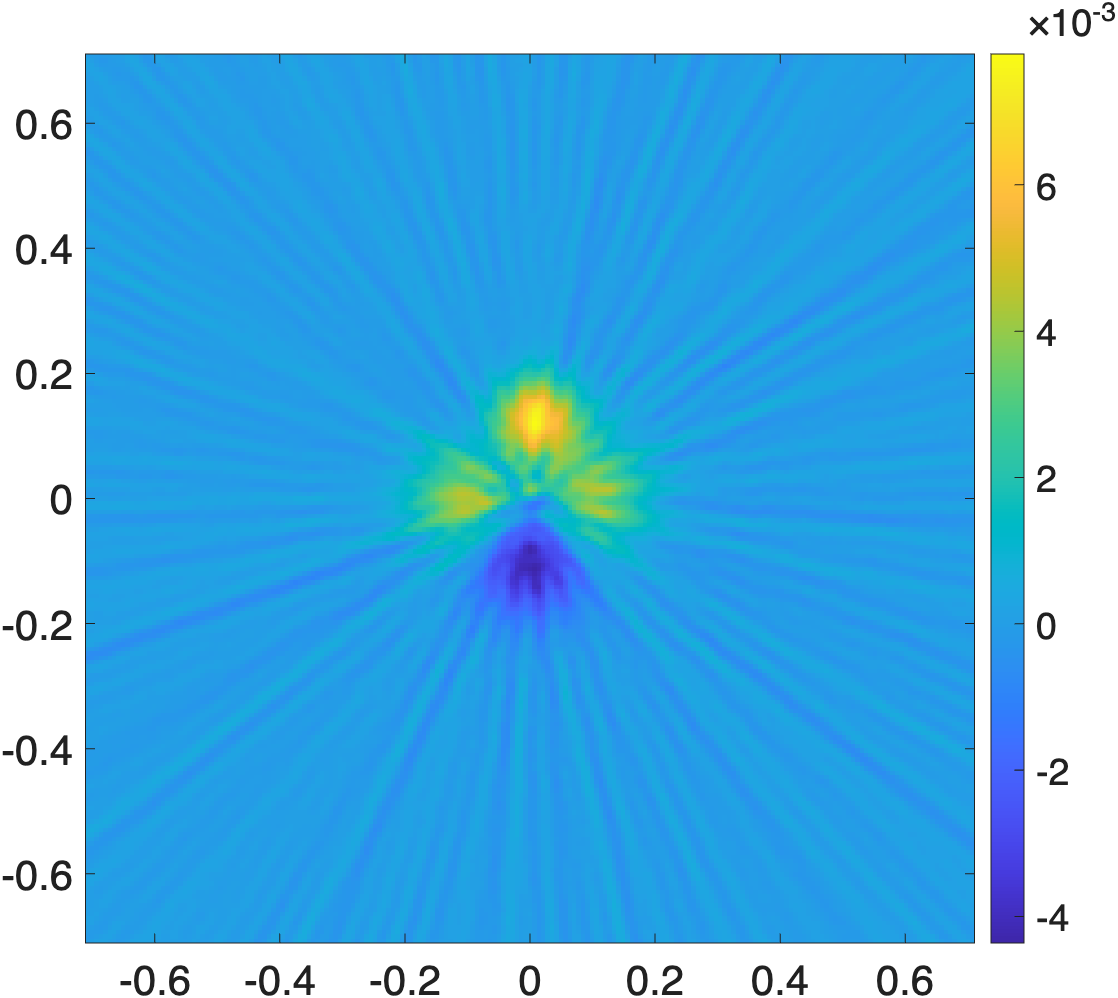}}\quad
    \subfigure[$\Re({\mathbf I}^\infty),\ \delta=10\%$]{\includegraphics[width=0.3\linewidth]{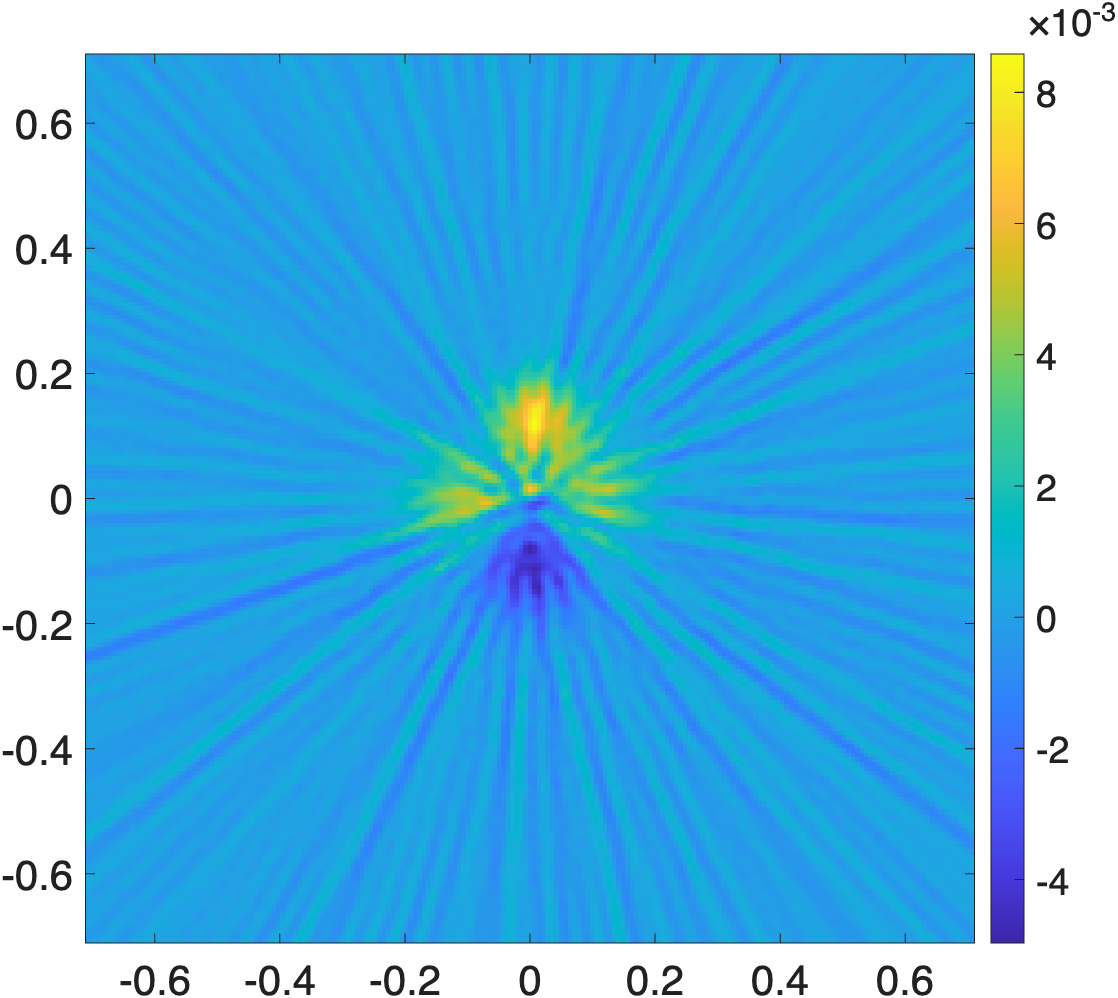}}\\
    \subfigure[$\Im({\mathbf I}^\infty),\ \delta=1\%$]{\includegraphics[width=0.3\linewidth]{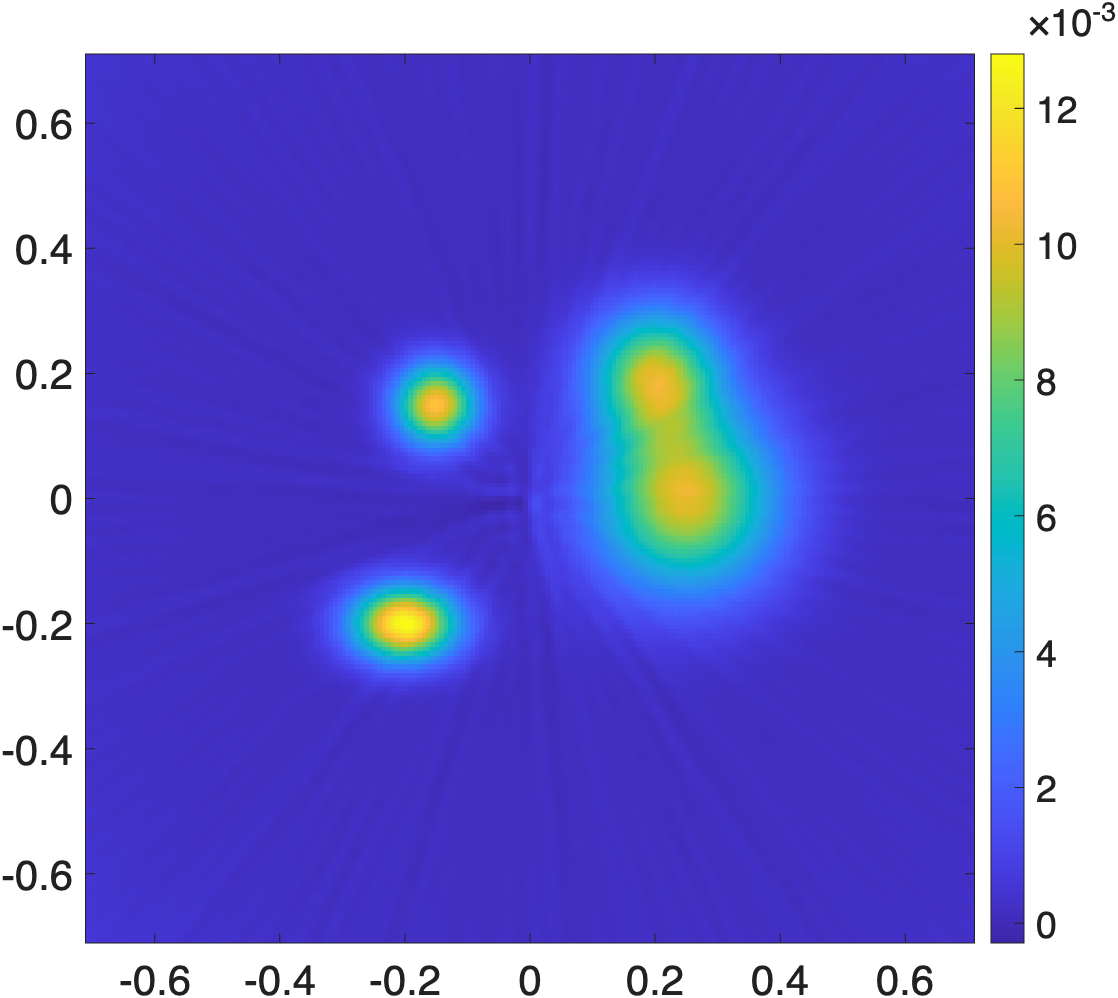}}\quad
    \subfigure[$\Im({\mathbf I}^\infty),\ \delta=5\%$]{\includegraphics[width=0.3\linewidth]{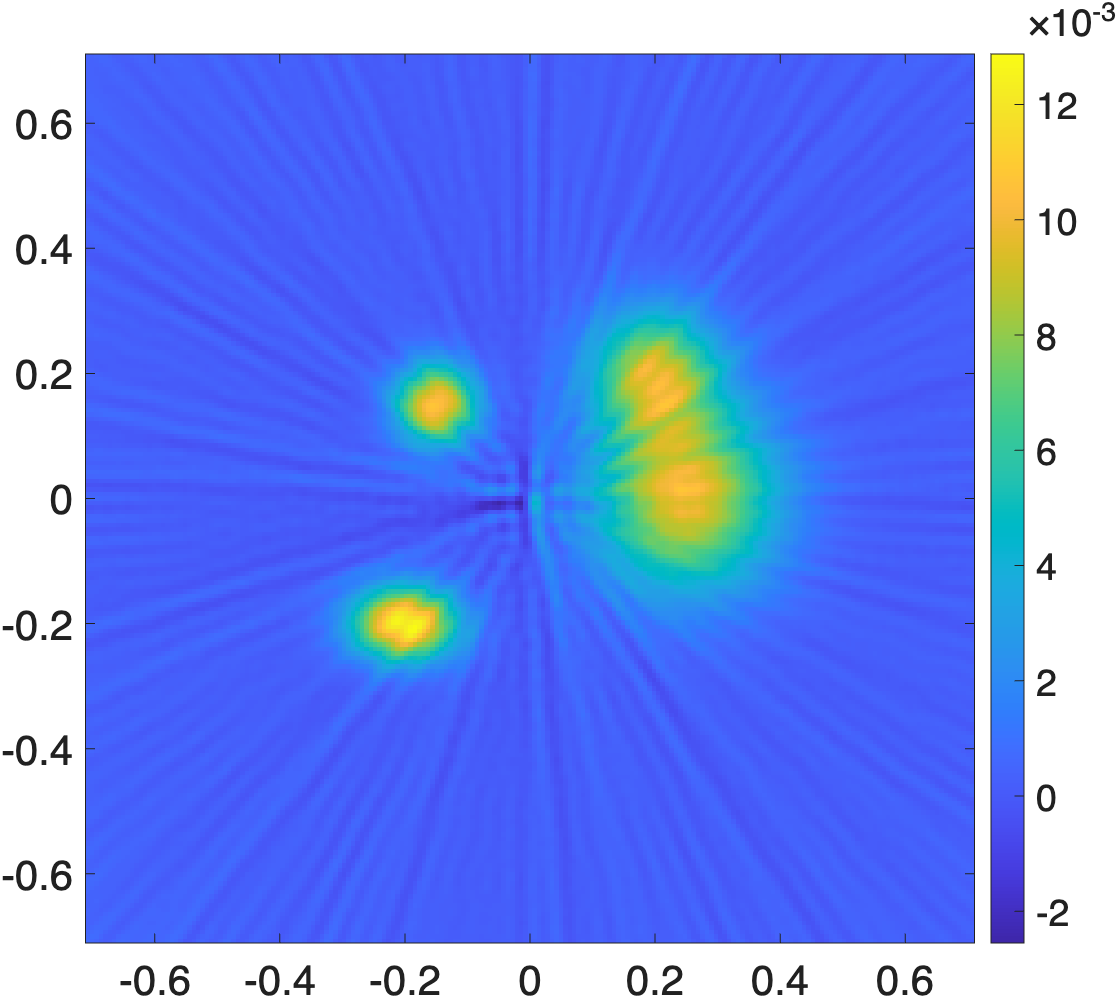}}\quad
    \subfigure[$\Im({\mathbf I}^\infty),\ \delta=10\%$]{\includegraphics[width=0.3\linewidth]{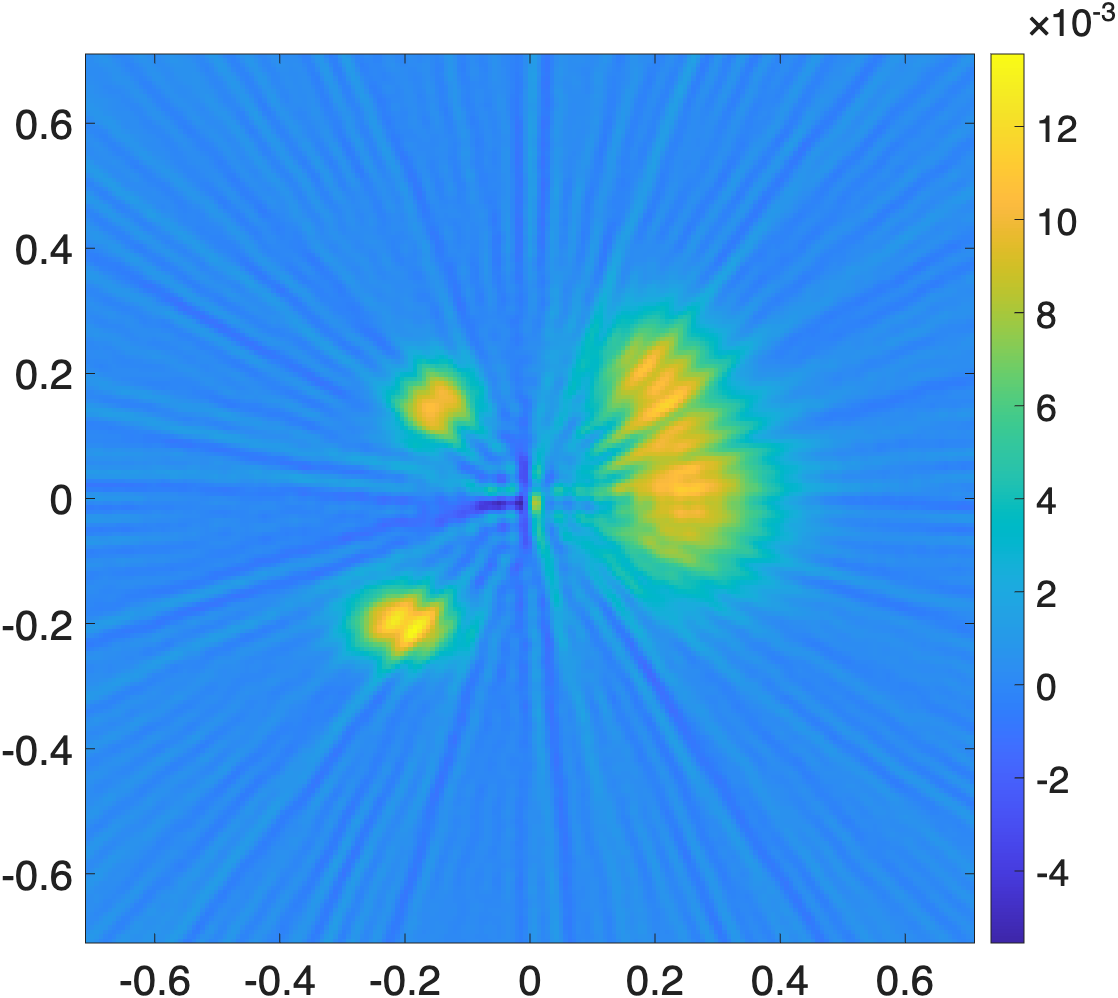}}
    \caption{Imaging the complex contrast by ${\mathbf I}^\infty$ with different noise levels.}
    \label{fig: 2D-example3-farField}
\end{figure}

\begin{figure}
    \centering
    \subfigure[$\Re({\mathbf I}^s),\ \delta=1\%$]{\includegraphics[width=0.3\linewidth]{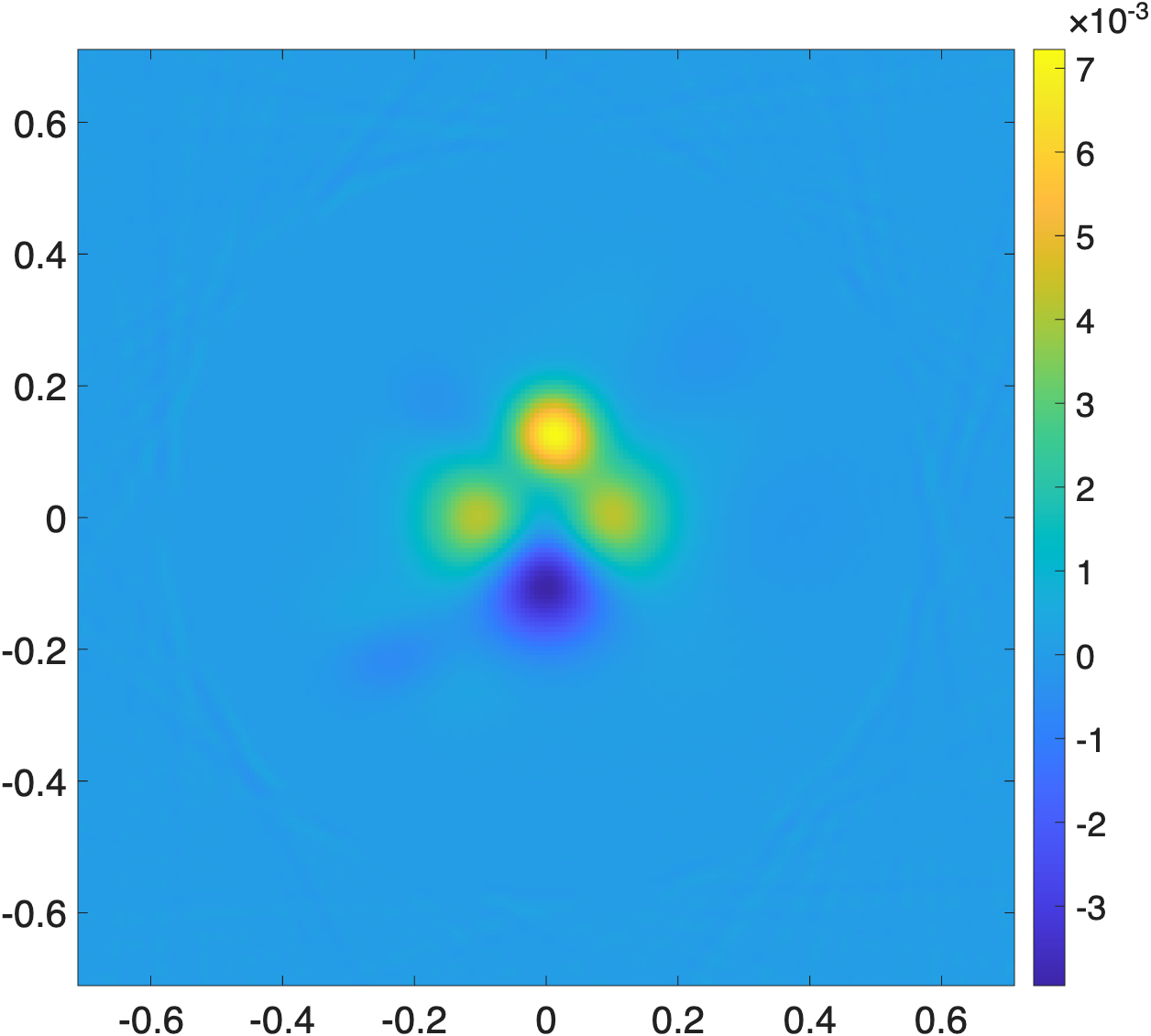}}\quad
    \subfigure[$\Re({\mathbf I}^s),\ \delta=5\%$]{\includegraphics[width=0.3\linewidth]{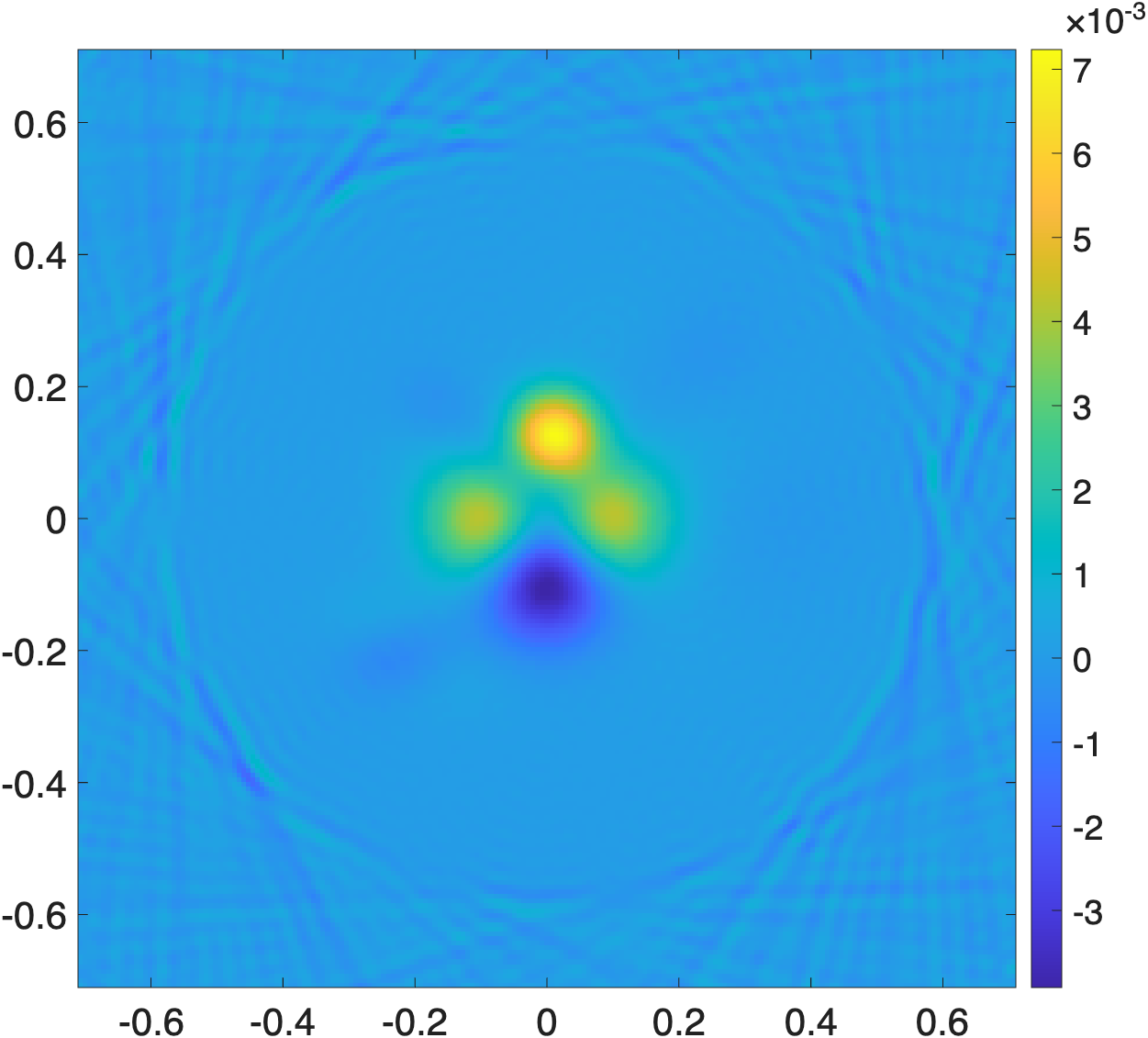}}\quad
    \subfigure[$\Re({\mathbf I}^s),\ \delta=10\%$]{\includegraphics[width=0.3\linewidth]{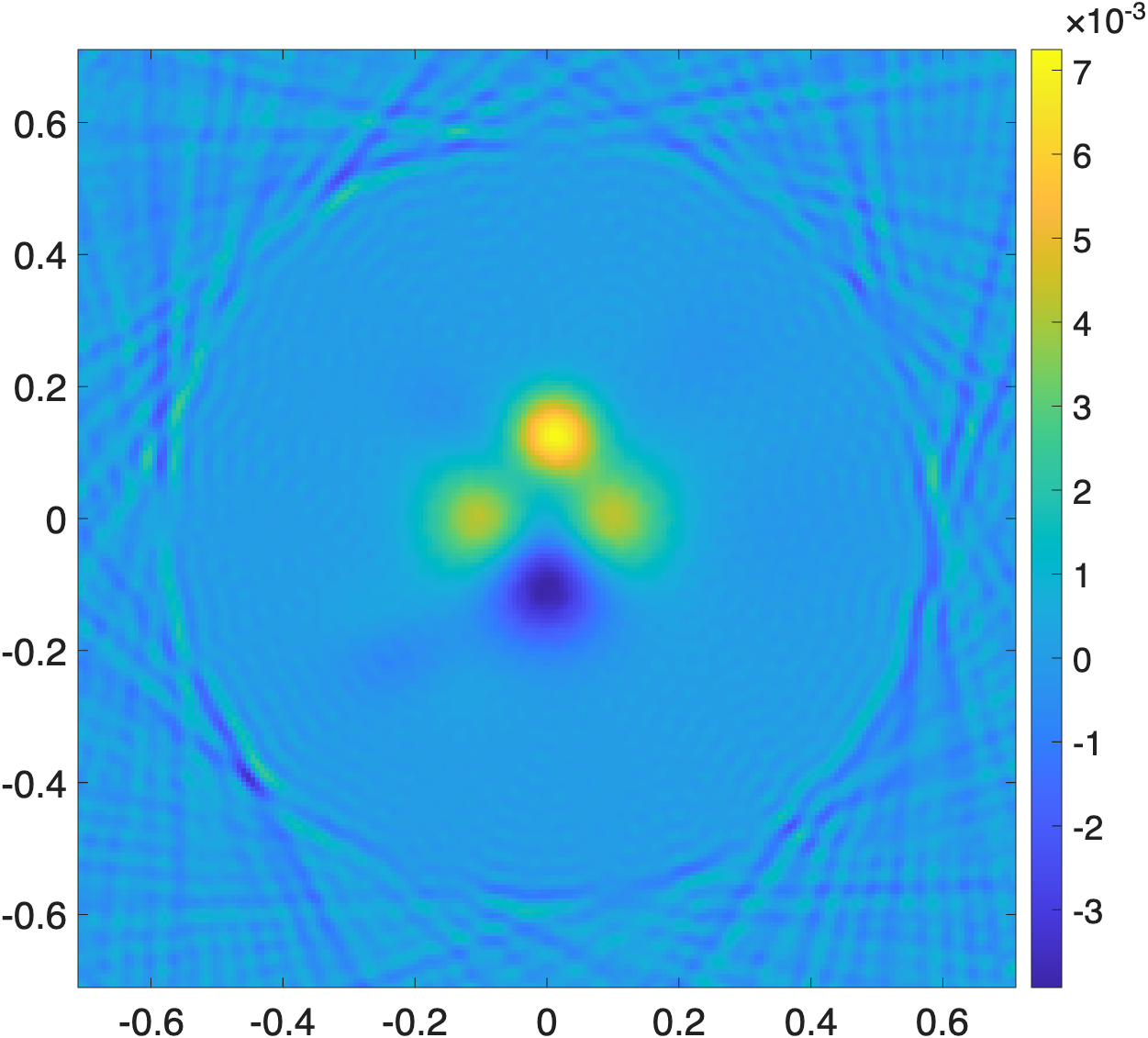}}\\
    \subfigure[$\Im({\mathbf I}^s),\ \delta=1\%$]{\includegraphics[width=0.3\linewidth]{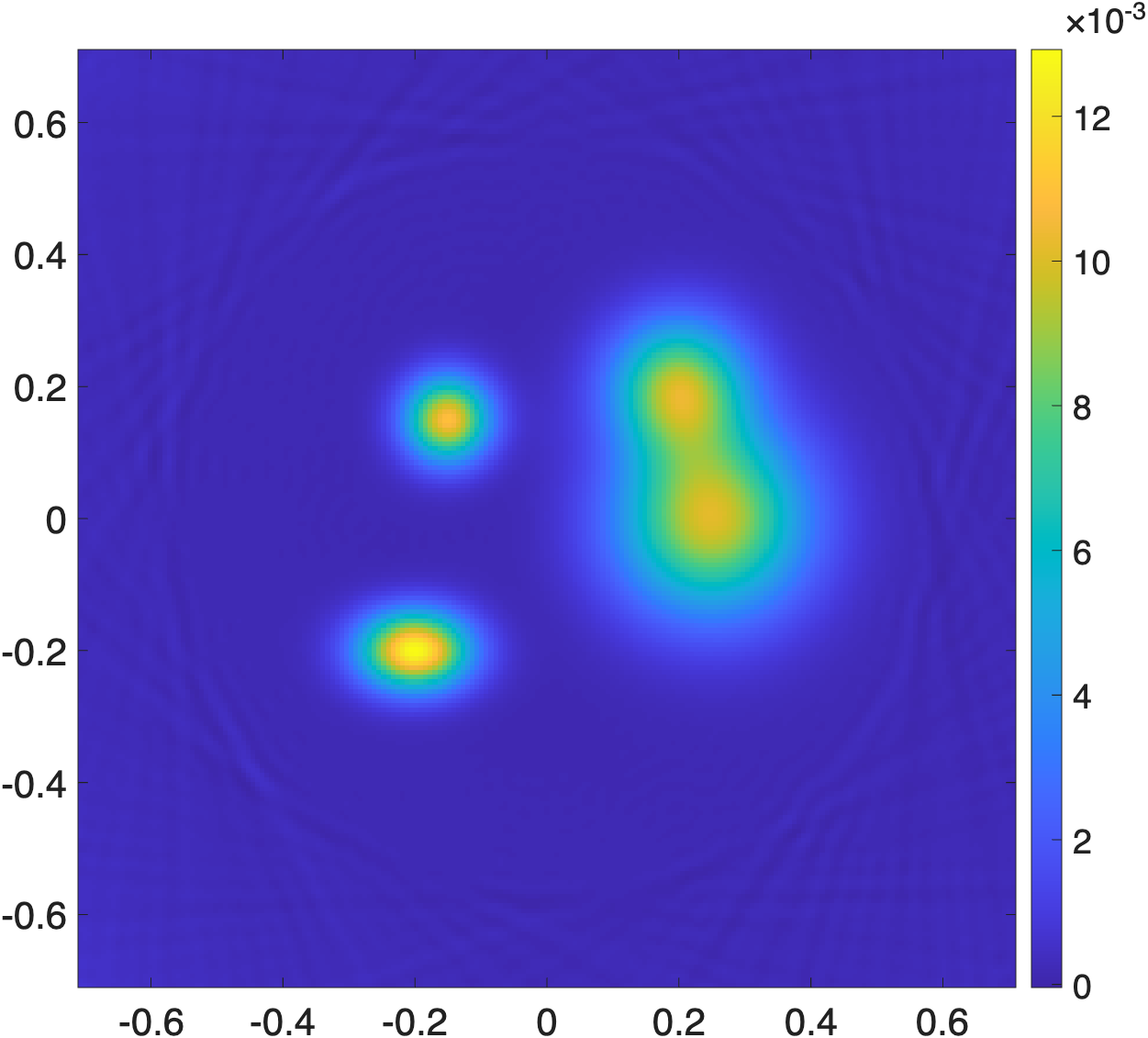}}\quad
    \subfigure[$\Im({\mathbf I}^s),\ \delta=5\%$]{\includegraphics[width=0.3\linewidth]{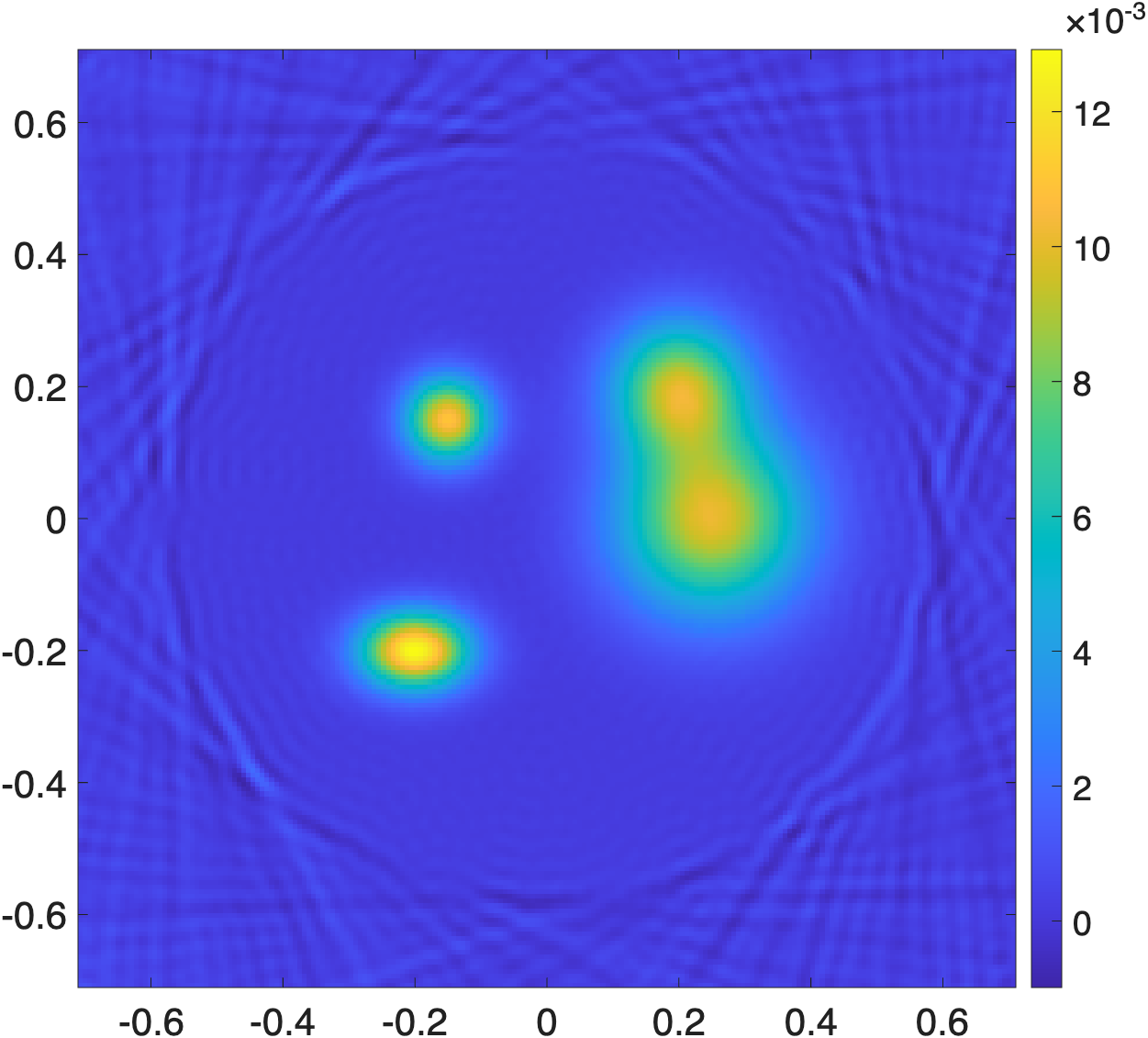}}\quad
    \subfigure[$\Im({\mathbf I}^s),\ \delta=10\%$]{\includegraphics[width=0.3\linewidth]{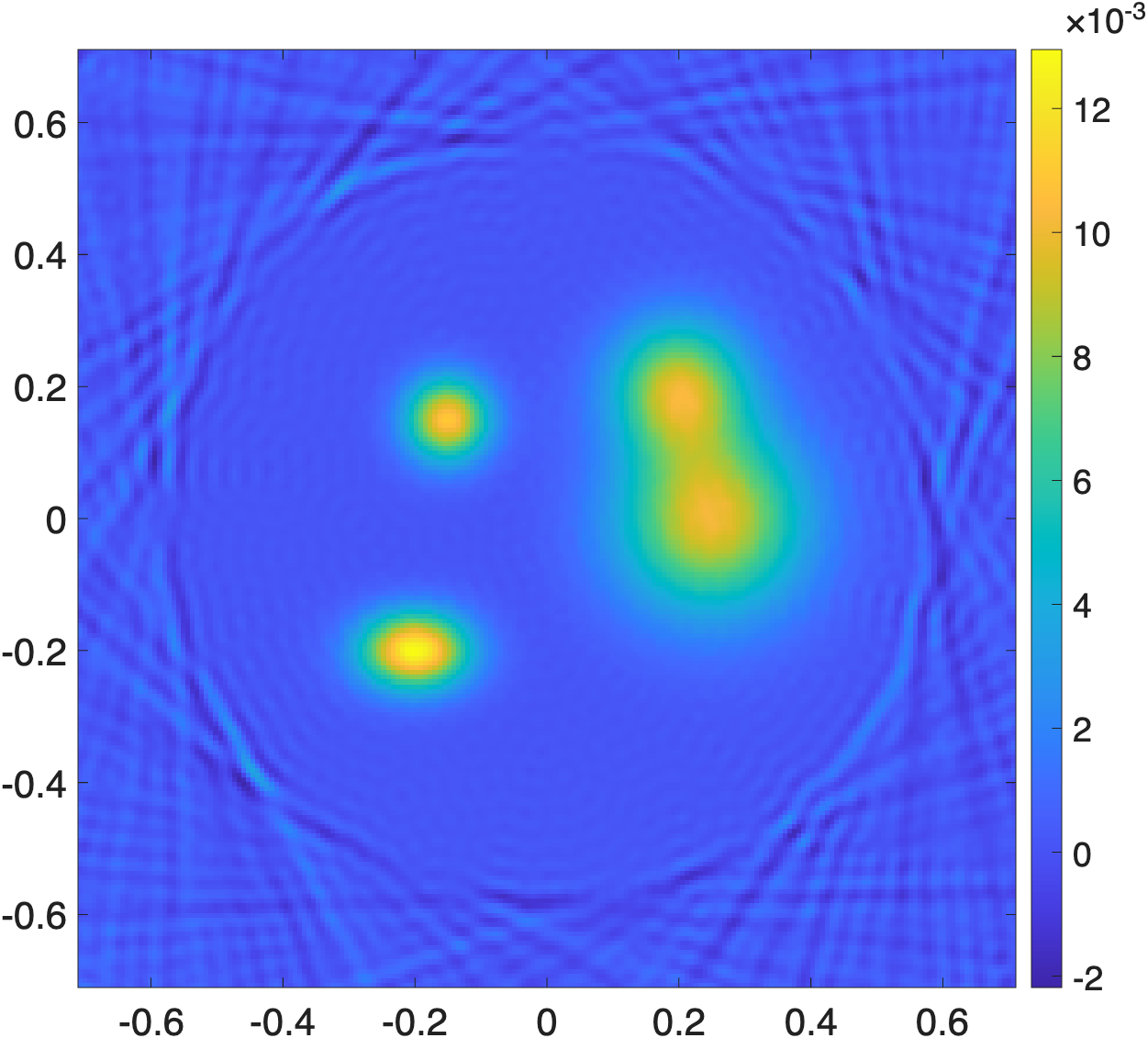}}
    \caption{Imaging the complex contrast by ${\mathbf I}^s$ with $R=10$ and different noise levels.}
    \label{fig: 2D-example3-nearField}
\end{figure}

\end{example}

\subsection{Three-dimensional examples}

In $\mathbb{R}^3$, we construct the set of observation directions using the Fibonacci lattice points:
$$
\hat{x}_\ell=(x_{\ell,1}, x_{\ell, 2}, x_{\ell, 3})^\top,\quad \ell=1,2,\cdots, L
$$
with
\begin{align*}
    &x_{\ell, 3}=1-\frac{2\ell}{L},\,\\
    &x_{\ell,1}=\sqrt{1-x_{\ell, 3}^2}\cos\left((\sqrt{5}-1)\pi \ell\right),\\
    &x_{\ell,2}=\sqrt{1-x_{\ell,3}^2}\sin\left((\sqrt{5}-1)\pi \ell\right).
\end{align*}
These lattice points are nearly uniformly distributed over the unit sphere $\mathbb{S}^2$, allowing us to approximate the spherical area element $\mathrm{d}s_{\hat{x}}$ with the uniform sampling weight ${4\pi}/{L}$.

For the near-field case in $\mathbb{R}^3$ the transmitters and/or receivers are also pseudo-equidistantly distributed on a sphere generated by the Fibonacci lattice. For an illustration of such a Fibonacci lattice, Figure \ref{fig: 3D-Fibonacci-lattice} demonstrates the 256 incident directions on $\mathbb{S}^2$ (denoted by the small red balls).

\begin{figure}
    \centering
    \includegraphics[width=0.3\linewidth]{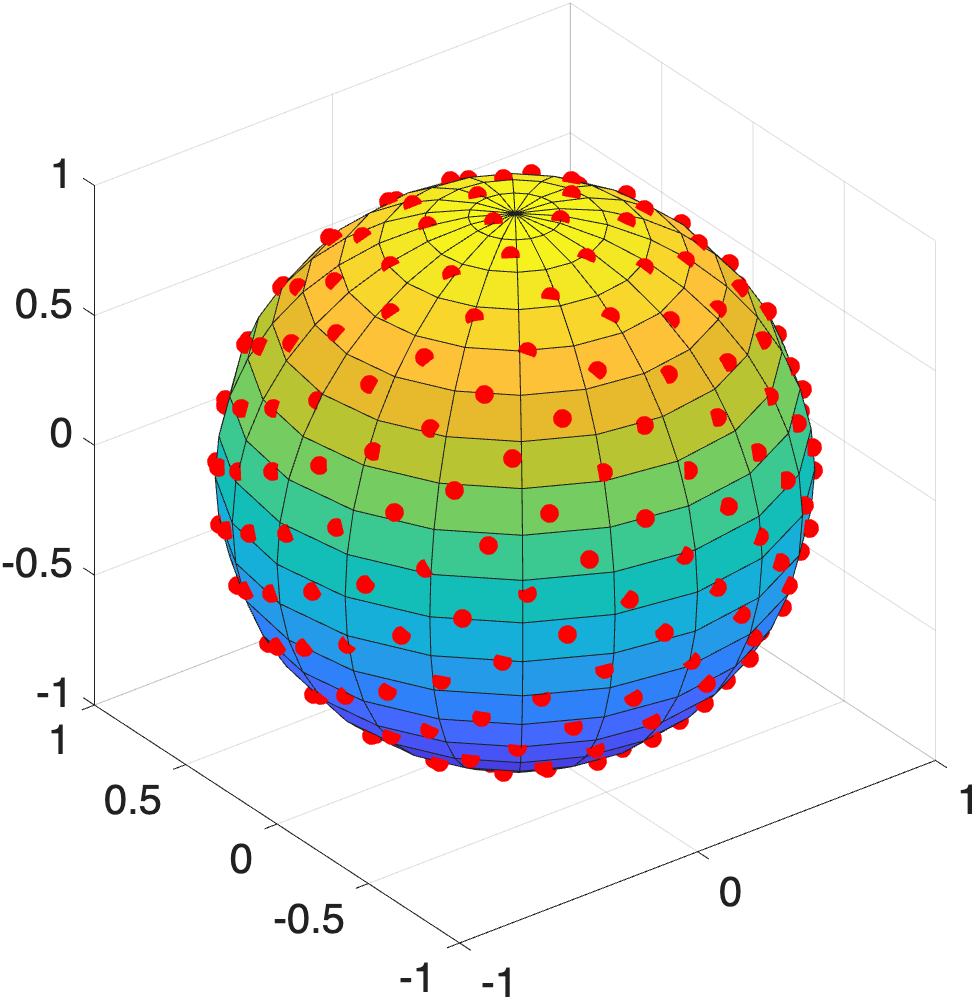}
    \caption{The 256 incident directions on the unit sphere $\mathbb{S}^2$.}
    \label{fig: 3D-Fibonacci-lattice}
\end{figure}

In terms of $\gamma_3(k)=1/(4\pi)$, the discrete forms of the 3D version indicator functions \eqref{I-infty} and \eqref{I-s} are written correspondingly as
$$
{\mathbf I}^{\infty}(z) = \frac{4\Delta\theta\Delta k}{\pi^2}\sum_{m=1}^{N_k} \sum_{j=1}^{N_\theta} u_b^{\infty, \delta}(-\theta_j, \theta_j, k_m)e^{-2{\rm i}k_m \theta_j\cdot z}
$$
and
$$
{\mathbf I}^s(z) =\frac{16R^2\Delta\theta\Delta k}{\pi}\sum\limits_{m=1}^{N_k}\sum\limits_{j=1}^{N_\theta}u_b^{s,\delta}(-x_j, x_j, k_m)e^{2{\rm i}k_m (\theta_j\cdot z-R)},
$$
where $\{\theta_j\}_{j=1}^{N_\theta}$ are the Fibonacci lattices on $\mathbb{S}^2$ such that $x_j=R\theta_j(j=1,2,\cdots,N_\theta)$. 

\begin{example}\label{ex: 3D_smooth}
    The first 3D example is devoted to the reconstruction of a smooth complex-valued contrast. The ground-truth contrast is defined as 
    $$
    q(x)=q(x_1, x_2, x_3) = C_s q^{\ast}(x_1, x_2, x_3),\quad x=(x_1, x_2, x_3)\in\mathbb{R}^3,
    $$
    where $C_s>0$ denotes the scaling factor and
    \begin{align*}
     \Re(q^{\ast}(x_1, x_2, x_3))& = 3(1-x_1)^2 e^{-500x_1^2-800(x_2-0.1)^2-600x_3^2} -10\left(\frac{x}{5}-x_1^3-x_2^5\right)e^{-400(x_1-0.1)^2-300x_2^2-500x_3^2}\\
                          &\quad -\frac{1}{3}e^{-450(x_1-0.1)^2-600x_2^2-700x_3^2},\\
    \Im(q^{\ast}(x_1, x_2, x_3))& = 3 e^{-300x_1^2-200(x_2+0.05)^2-350x_3^2} +5e^{-180(x_1-0.1)^2-350x_2^2-250x_3^2}.
    \end{align*}

    Figure \ref{fig: 3D_smooth_exact} shows the contours of the exact contrast in different slice planes. In the experiments concerning inversion, we choose the parameters $\Delta k=2, k_{\rm min}=1, k_{\rm max}=61$ and near-field measurements are collected at $N_\theta=256$ receivers with measurement noise level $\delta =1\%$ on a sphere with radius $R=5$.  
    
\begin{figure}
    \centering
    \subfigure[$\Re(q)$ in $x_1=-0.2, 0, 0.2$]{\includegraphics[width=0.3\linewidth]{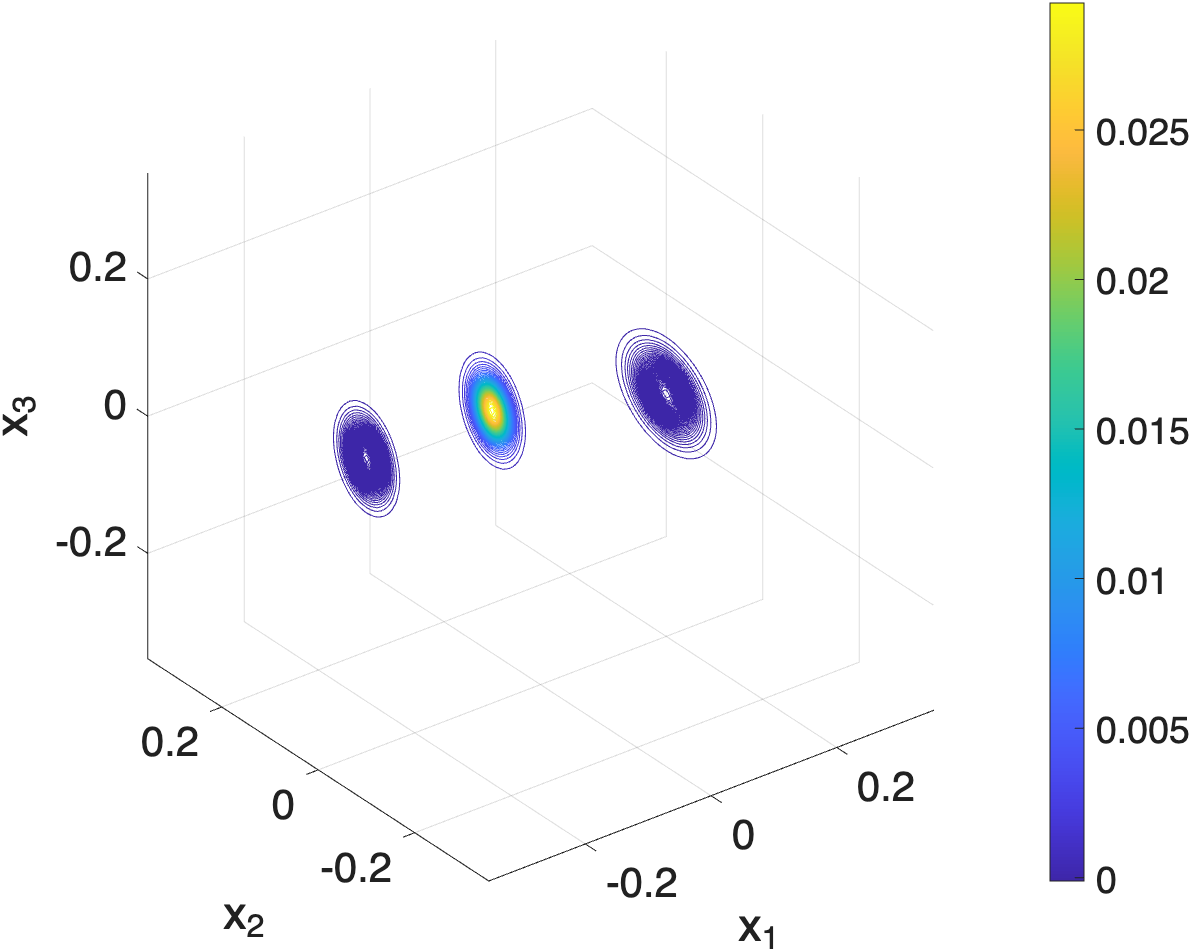}}\quad
    \subfigure[$\Re(q)$ in $x_2=-0.2, 0, 0.2$]{\includegraphics[width=0.3\linewidth]{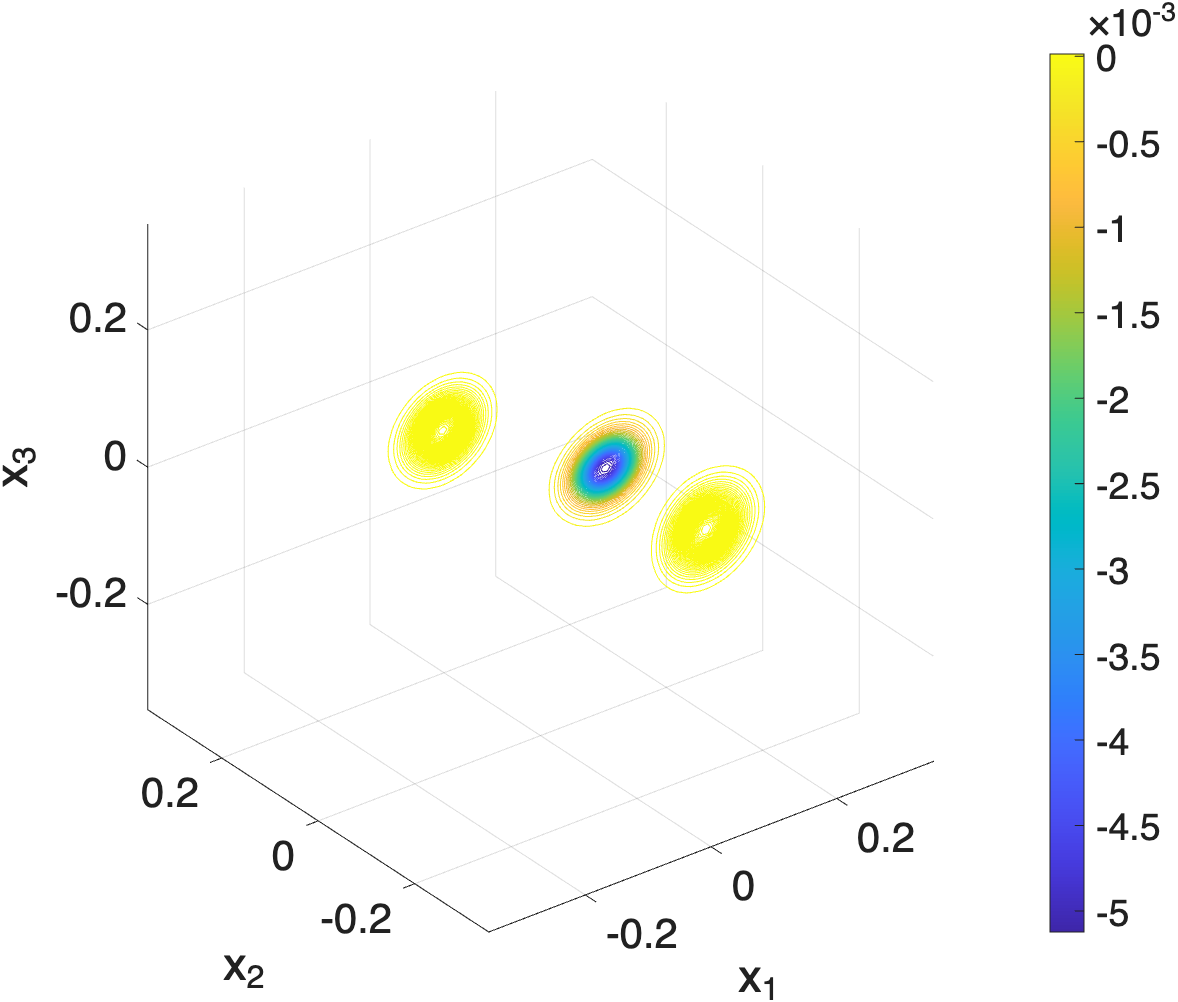}}\quad
    \subfigure[$\Re(q)$ in $x_3=-0.2, 0, 0.2$]{\includegraphics[width=0.3\linewidth]{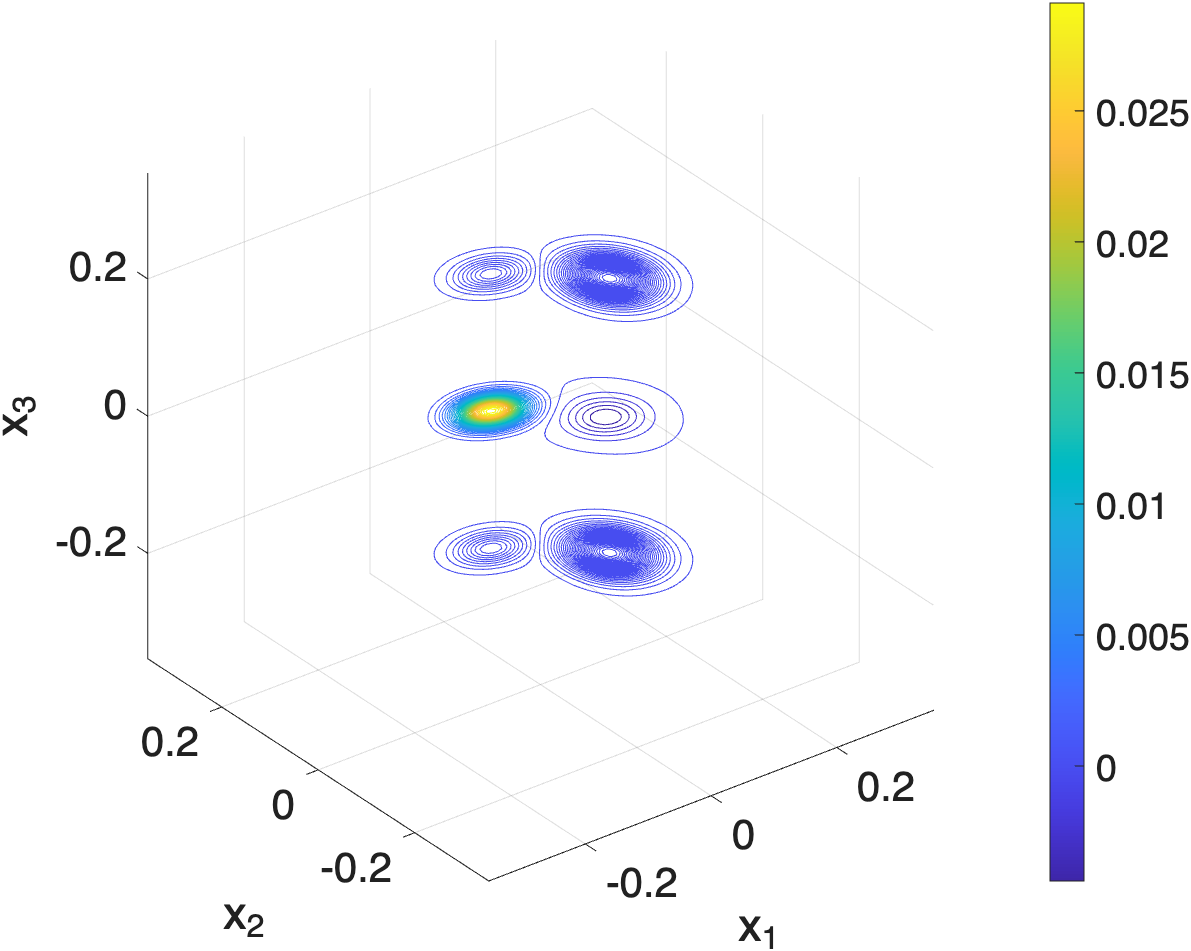}}\\
    \subfigure[$\Im(q)$ in $x_1=-0.2, 0, 0.2$]{\includegraphics[width=0.3\linewidth]{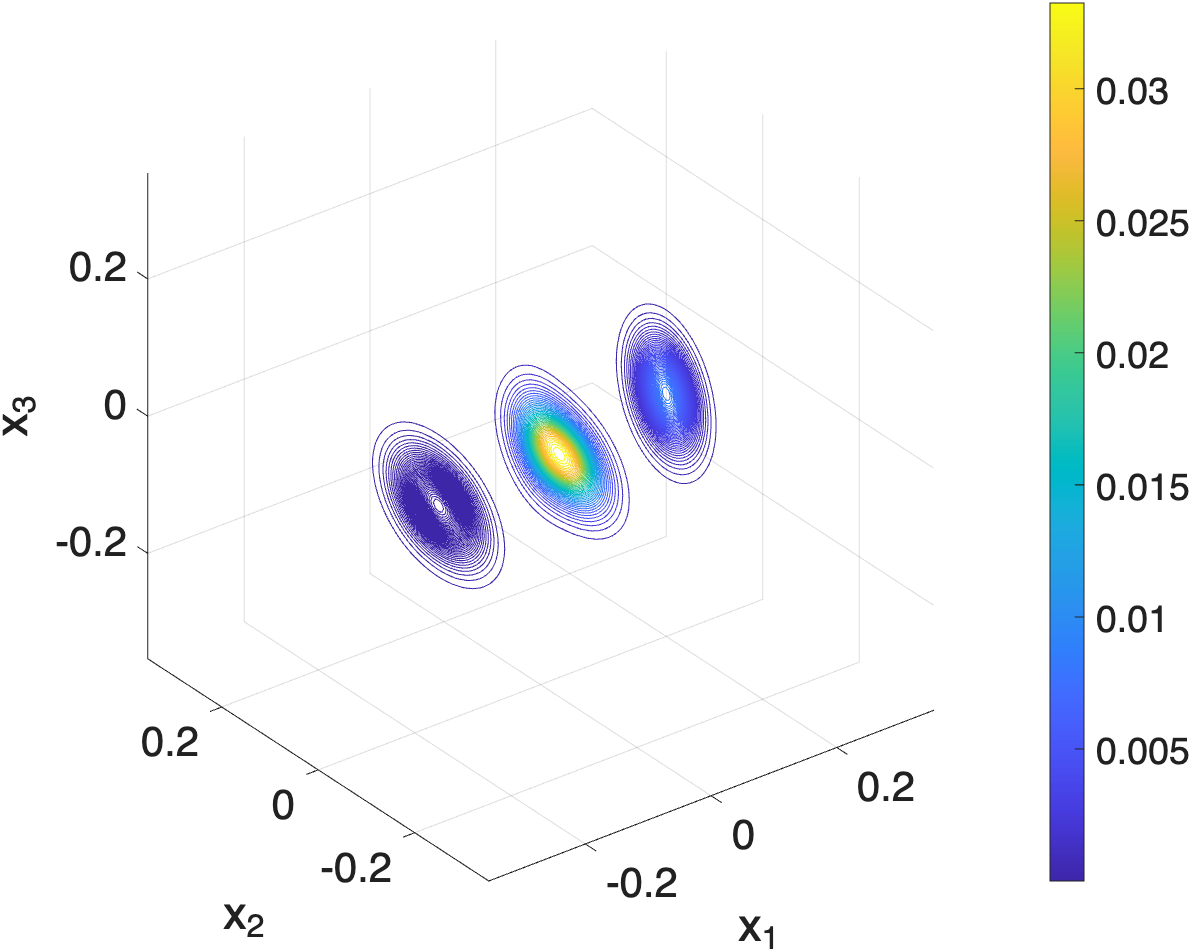}}\quad
    \subfigure[$\Im(q)$ in $x_2=-0.2, 0, 0.2$]{\includegraphics[width=0.3\linewidth]{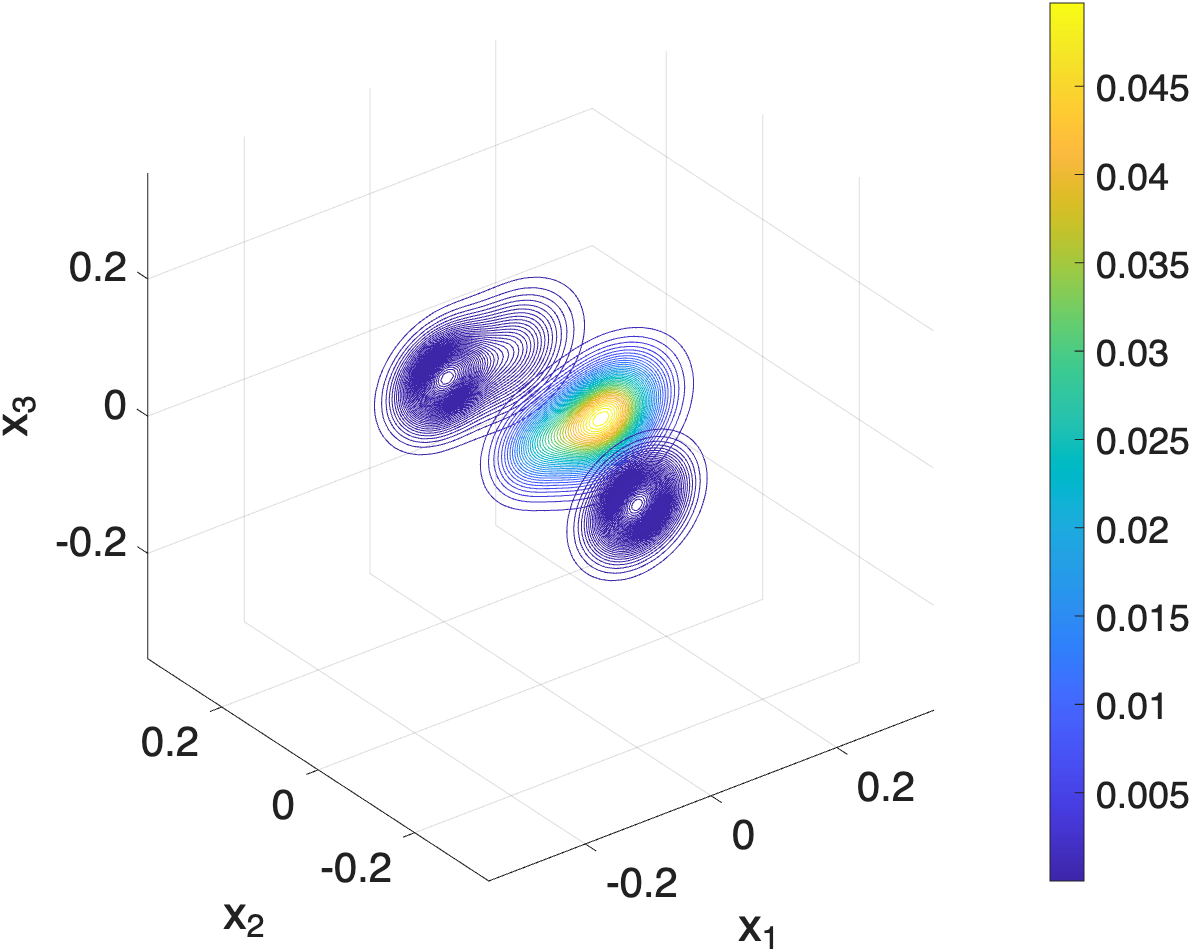}}\quad
    \subfigure[$\Im(q)$ in $x_3=-0.2, 0, 0.2$]{\includegraphics[width=0.3\linewidth]{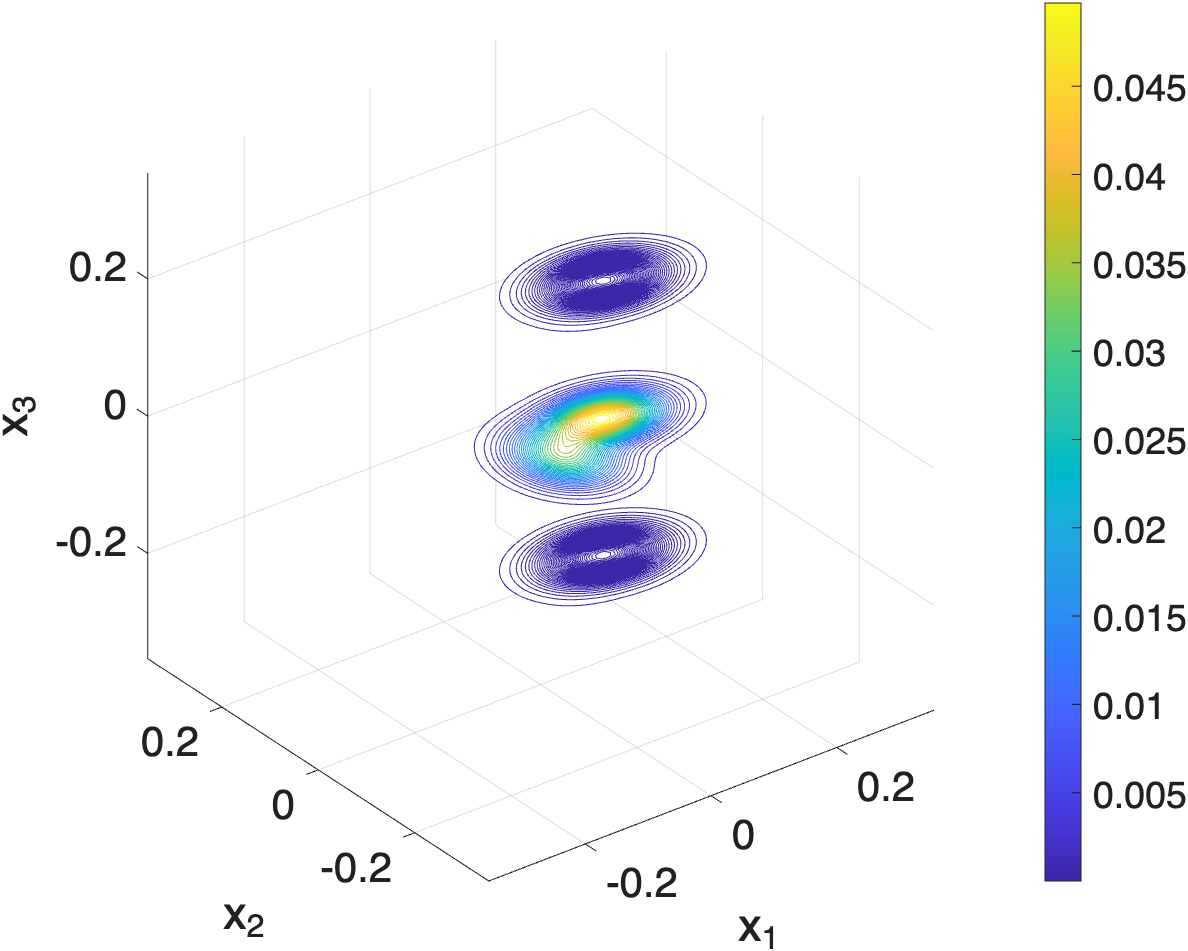}}
    \caption{Contour plots of the ground-truth contrast $q$ in slice planes ($C_s=10^{-2}$).}
    \label{fig: 3D_smooth_exact}
\end{figure}

We first select $C_s=10^{-2}$ and evaluate the indicator ${\mathbf I}^s$ over the uniformly distributed $101^3$ sampling grid in the region $[-0.35, 0.35]^3$. Figure \ref{fig: 3D_smooth_reconstruction_slice} depicts the corresponding reconstructions in the form of contour slices of the indicator. By comparing the true contrast in Figure \ref{fig: 3D_smooth_exact} and the recovered results in Figure \ref{fig: 3D_smooth_reconstruction_slice}, one can readily find that the proposed indicator well captures the target contrast.

\begin{figure}
    \centering
    \subfigure[$\Re({\mathbf I}^s)$ in $x_1=0$]{\includegraphics[width=0.3\linewidth]{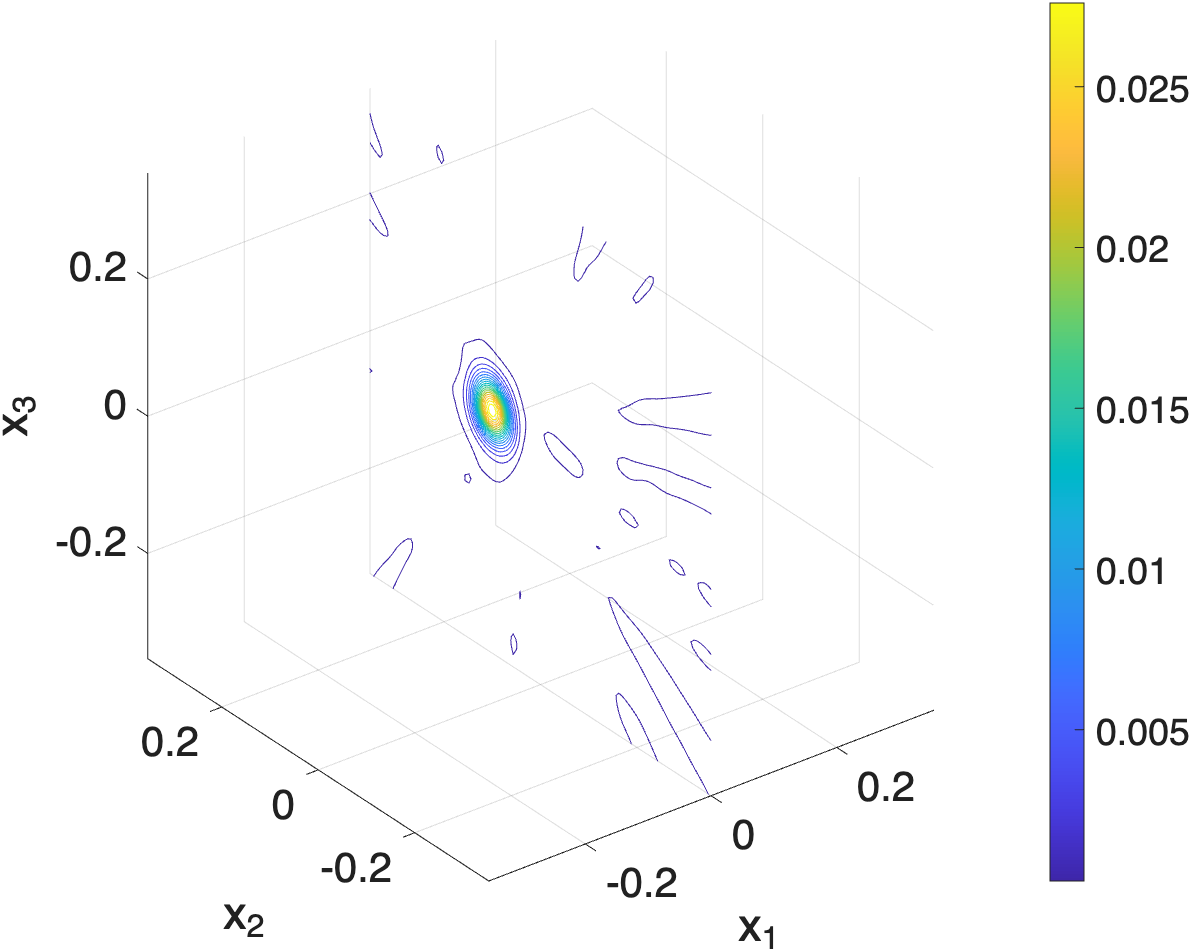}}\quad
    \subfigure[$\Re({\mathbf I}^s)$ in $x_2=0$]{\includegraphics[width=0.3\linewidth]{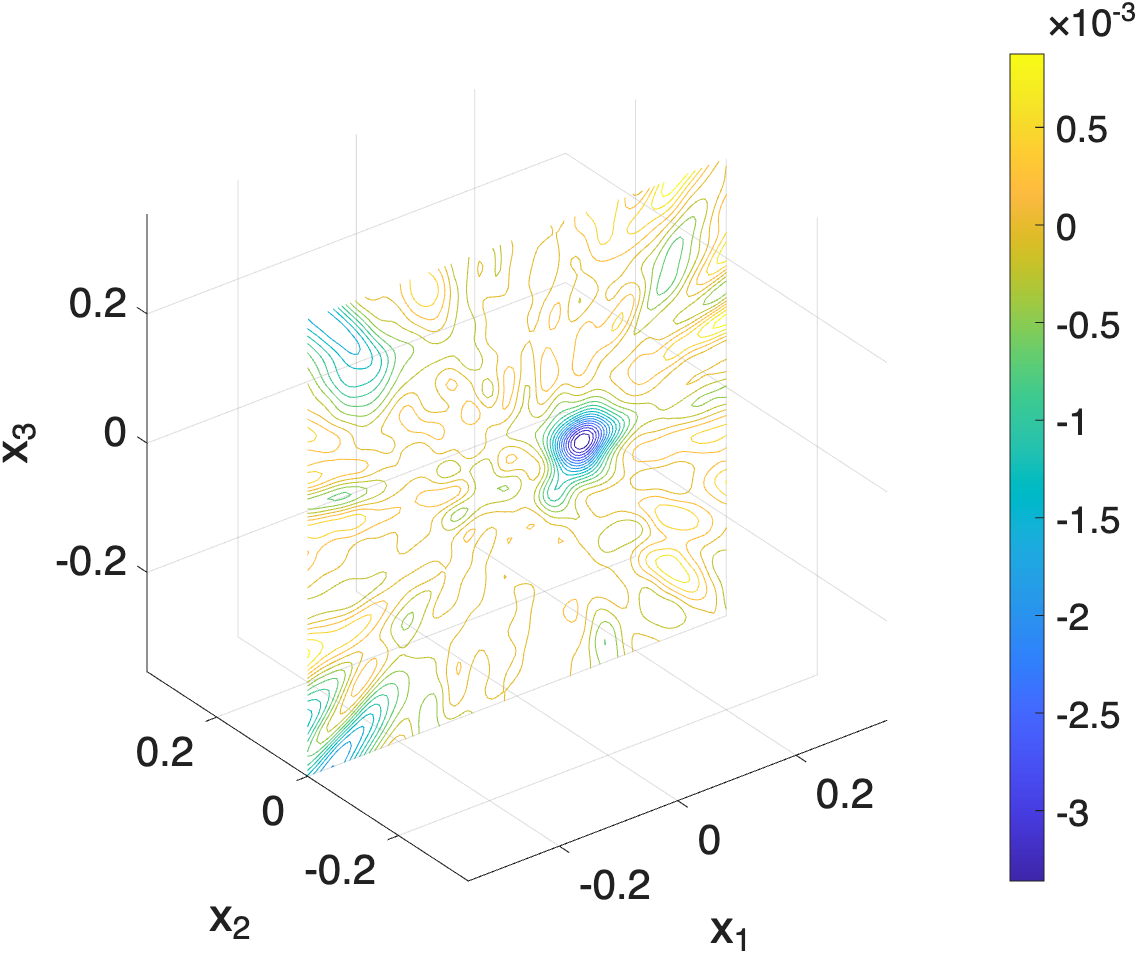}}\quad
    \subfigure[$\Re({\mathbf I}^s)$ in $x_3=0$]{\includegraphics[width=0.3\linewidth]{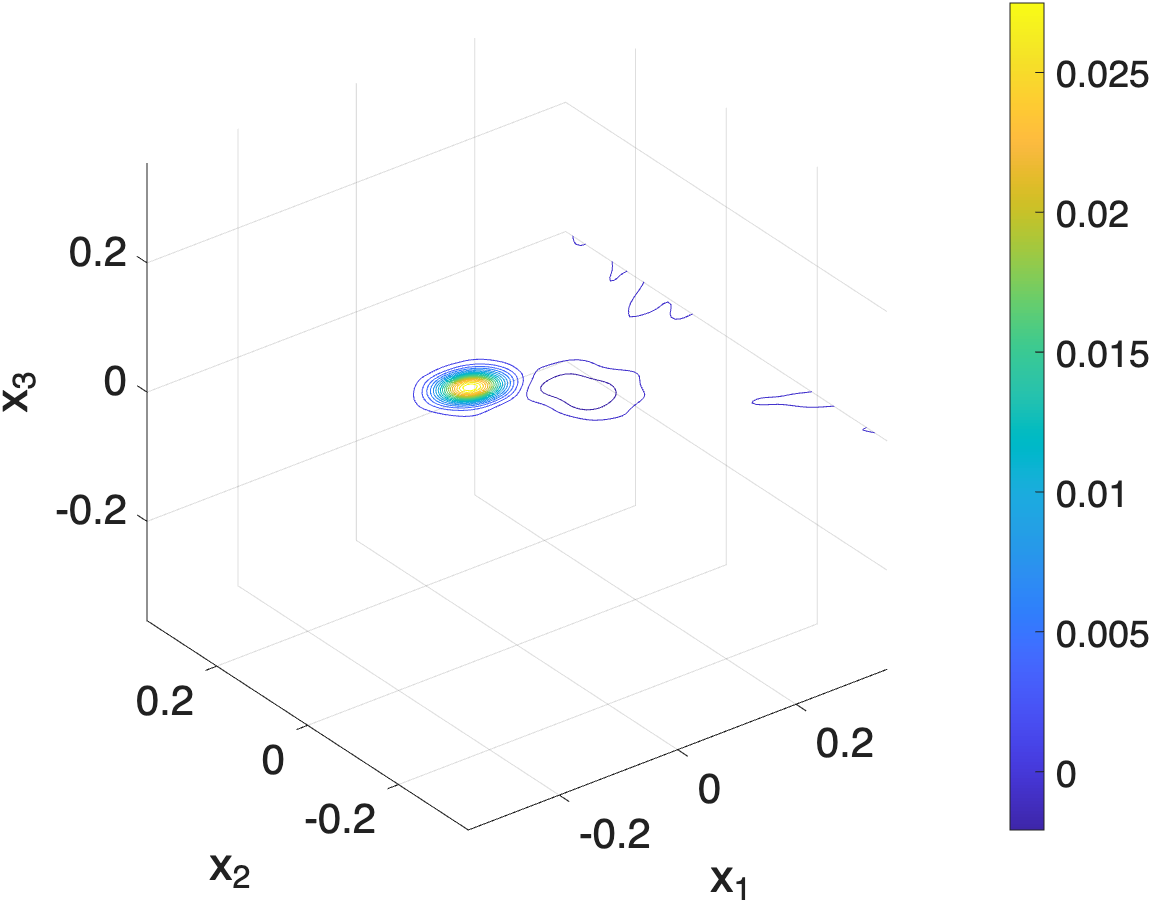}}\\
    \subfigure[$\Im({\mathbf I}^s)$ in $x_1=0$]{\includegraphics[width=0.3\linewidth]{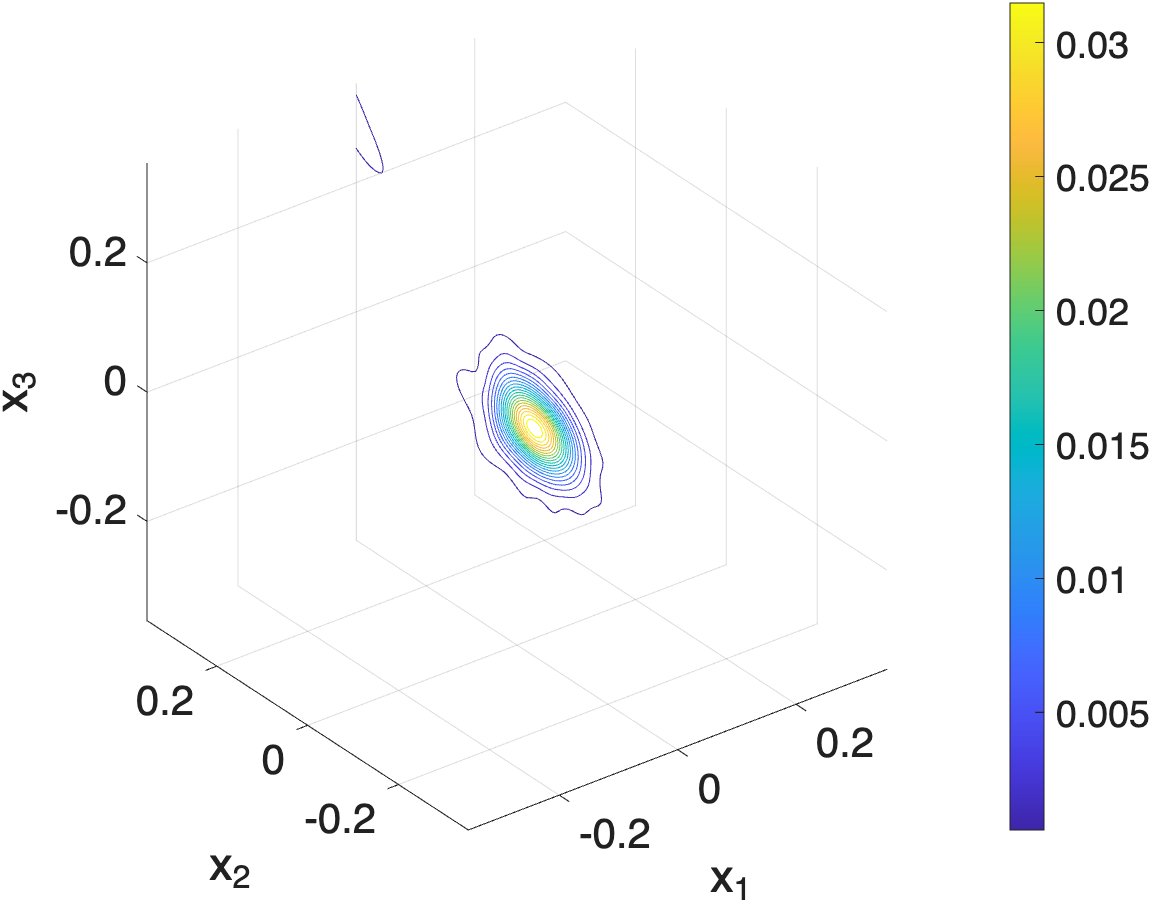}}\quad
    \subfigure[$\Im({\mathbf I}^s)$ in $x_2=0$]{\includegraphics[width=0.3\linewidth]{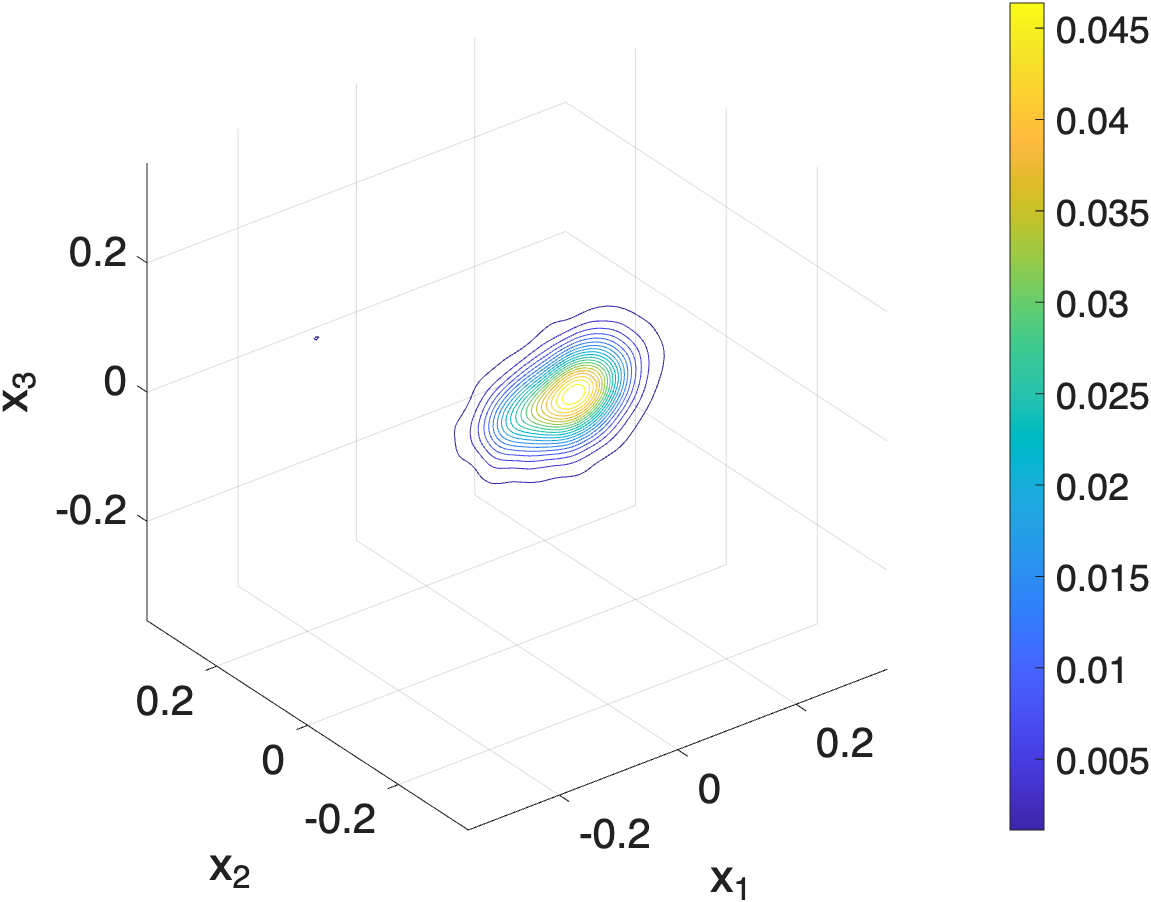}}\quad
    \subfigure[$\Im({\mathbf I}^s)$ in $x_3=0$]{\includegraphics[width=0.3\linewidth]{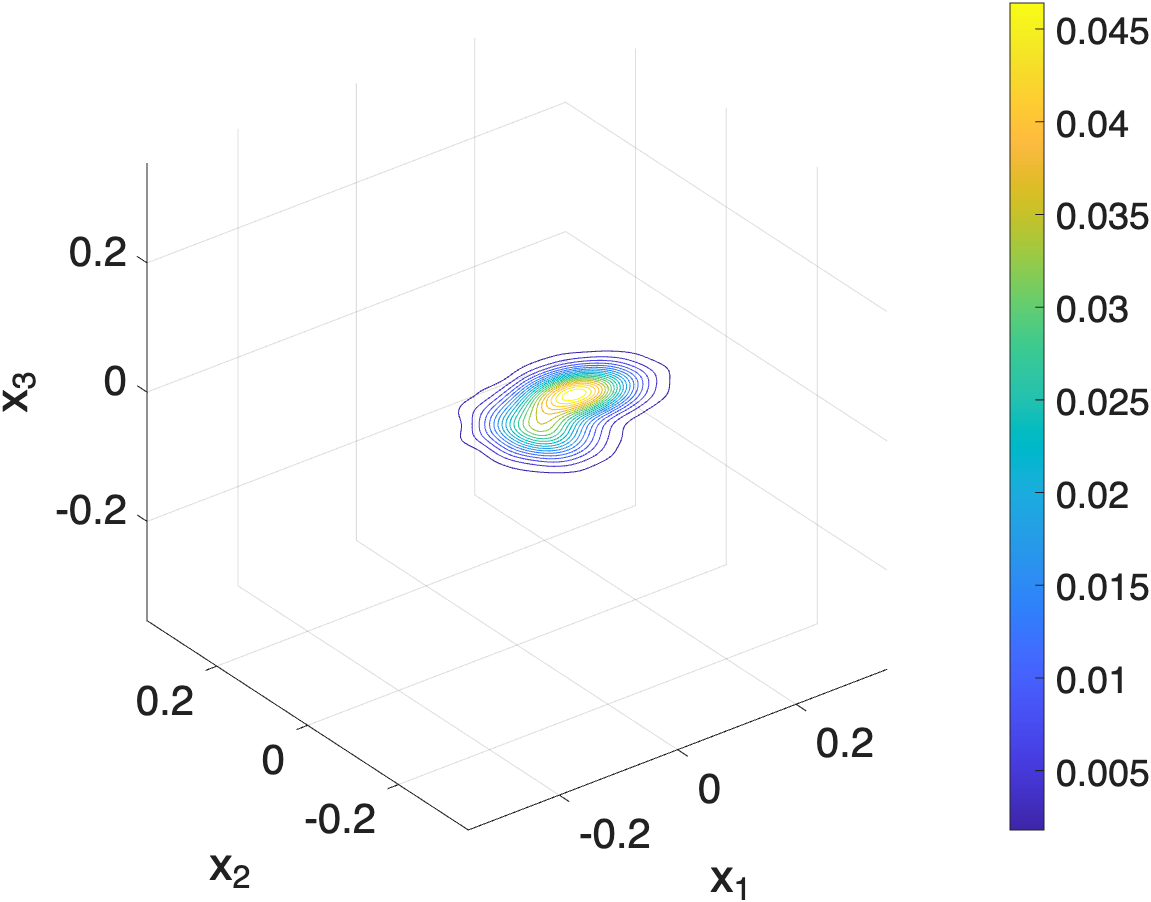}}
    \caption{Contour plots of ${\mathbf I}^s$ in slice planes ($C_s=10^{-2}$).}
    \label{fig: 3D_smooth_reconstruction_slice}
\end{figure}

From another perspective of viewing the accuracy of reconstruction, we further compare $q$ with the indicator ${\mathbf I}^s$ via the iso-surface plots, as demonstrated in Figure \ref{fig: 3D_smooth_reconstruction_isosurface}. These results strengthen the validation of the high-quality reconstructions produced by the QDSM.

\begin{figure}
    \centering
    \subfigure[$\Re(q)$]{\includegraphics[width=0.3\linewidth]{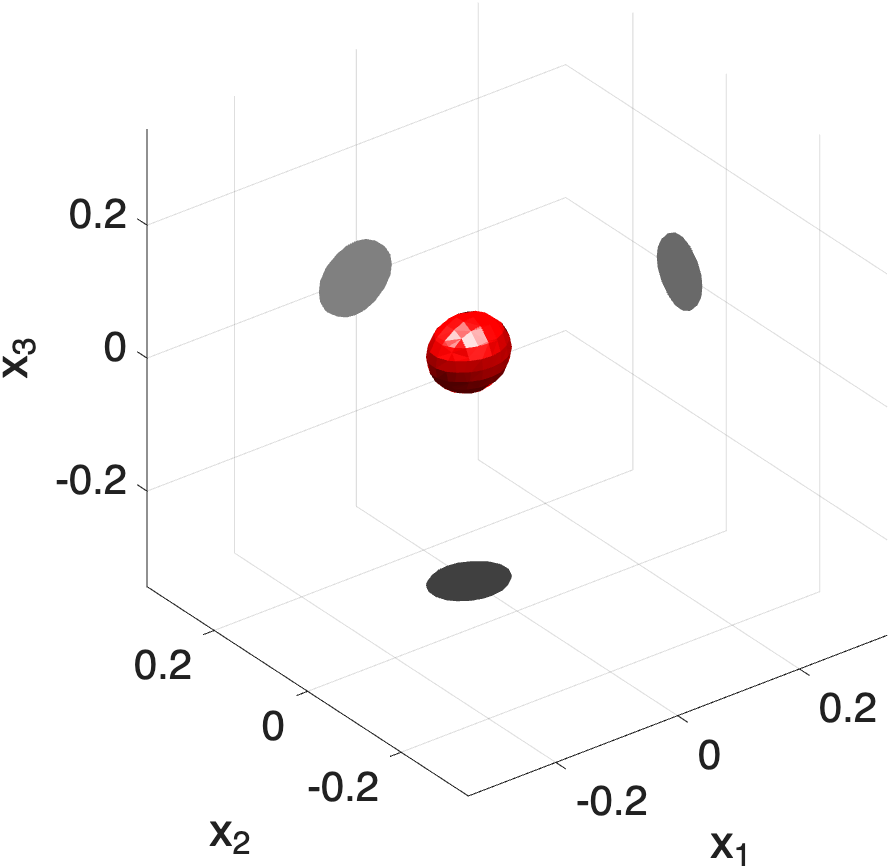}}\quad
    \subfigure[$\Im(q)$]{\includegraphics[width=0.3\linewidth]{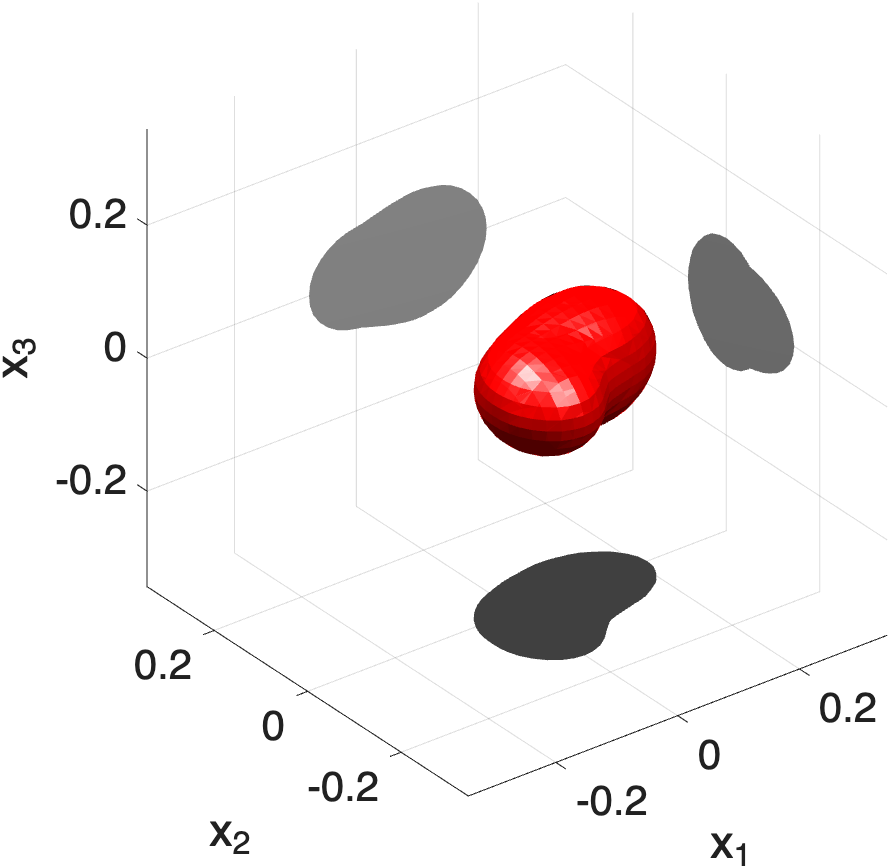}}\quad
    \subfigure[$|q|$]{\includegraphics[width=0.3\linewidth]{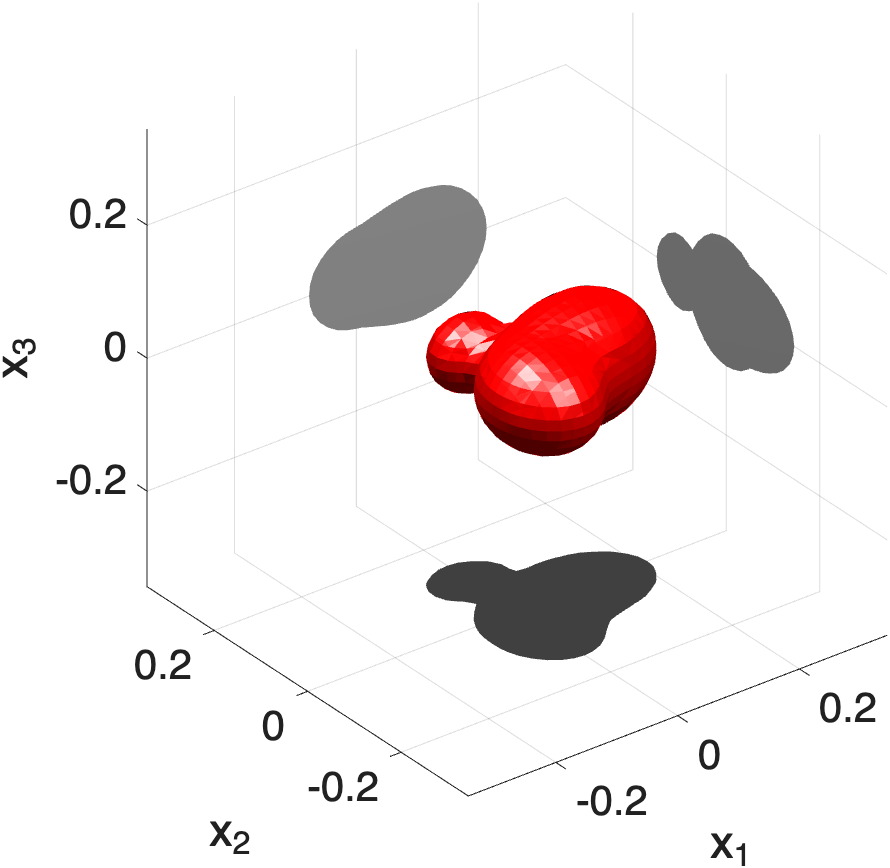}}\\
    \subfigure[$\Re({\mathbf I}^s)$]{\includegraphics[width=0.3\linewidth]{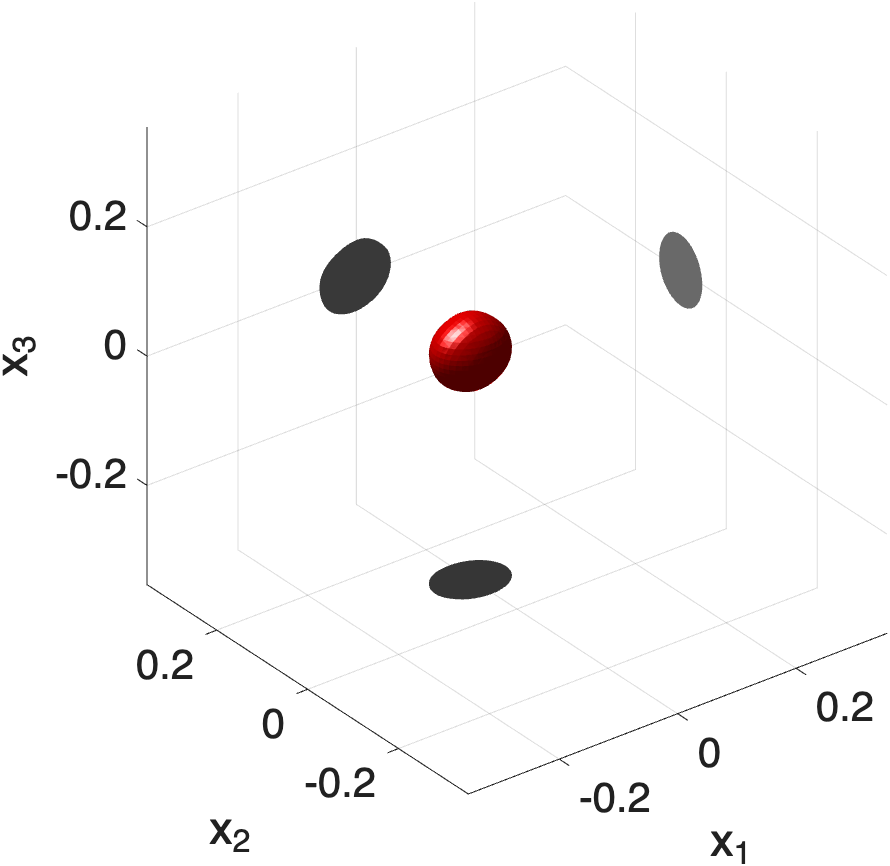}}\quad
    \subfigure[$\Im({\mathbf I}^s)$]{\includegraphics[width=0.3\linewidth]{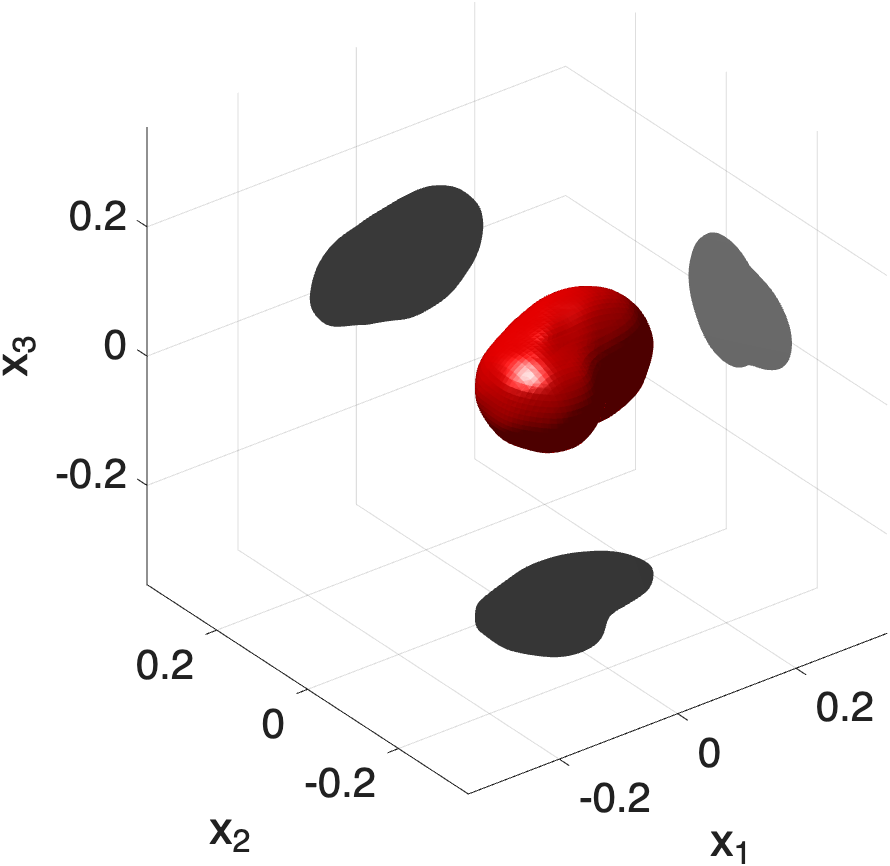}}\quad
    \subfigure[$|{\mathbf I}^s|$]{\includegraphics[width=0.3\linewidth]{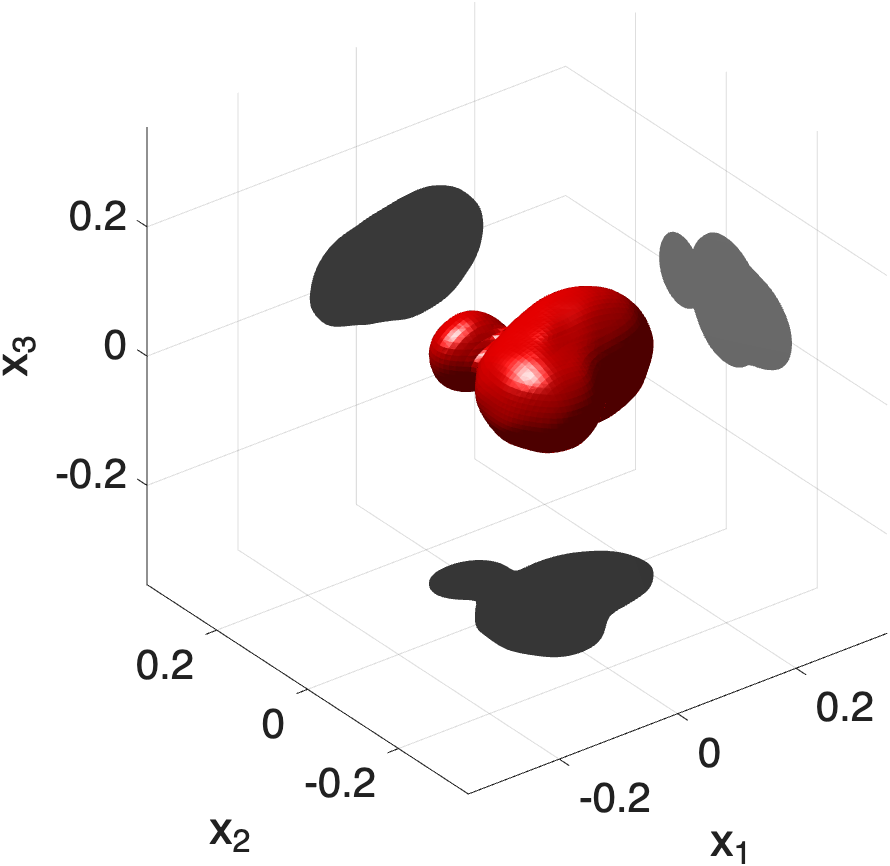}}
    \caption{Iso-surface plots with iso-surface value $C_{\rm iso}=5\times 10^{-3}$ ($C_s=10^{-2}$).}
    \label{fig: 3D_smooth_reconstruction_isosurface}
\end{figure}

Next we aim to investigate how the magnitude of the contrast function influences its identification. To this end, a new simulation is designed by simply increasing the value of the scaling factor from $C_s=10^{-2}$ to $C_s=10^{-1}$ and retaining the other experimental parameters. After recomputing the indicator function in this case, we obtain the reconstructions in Figure \ref{fig: 3D_smooth_reconstruction_slice_1e-1} and Figure \ref{fig: 3D_smooth_reconstruction_isosurface_5e-2}. Comparing the results due to $C_s=10^{-2}$ and $C_s=10^{-1}$, we find that the quality of reconstruction is, roughly speaking, inversely proportional to the magnitude of the contrast function. We also want to point out that, in our experiments this trend of correspondence can be also observed by further increasing the scaling factor to $C_s=1$ or decreasing it to $C_s=10^{-3}$.

\begin{figure}
    \centering
    \subfigure[$\Re({\mathbf I}^s)$ in $x_1=0$]{\includegraphics[width=0.3\linewidth]{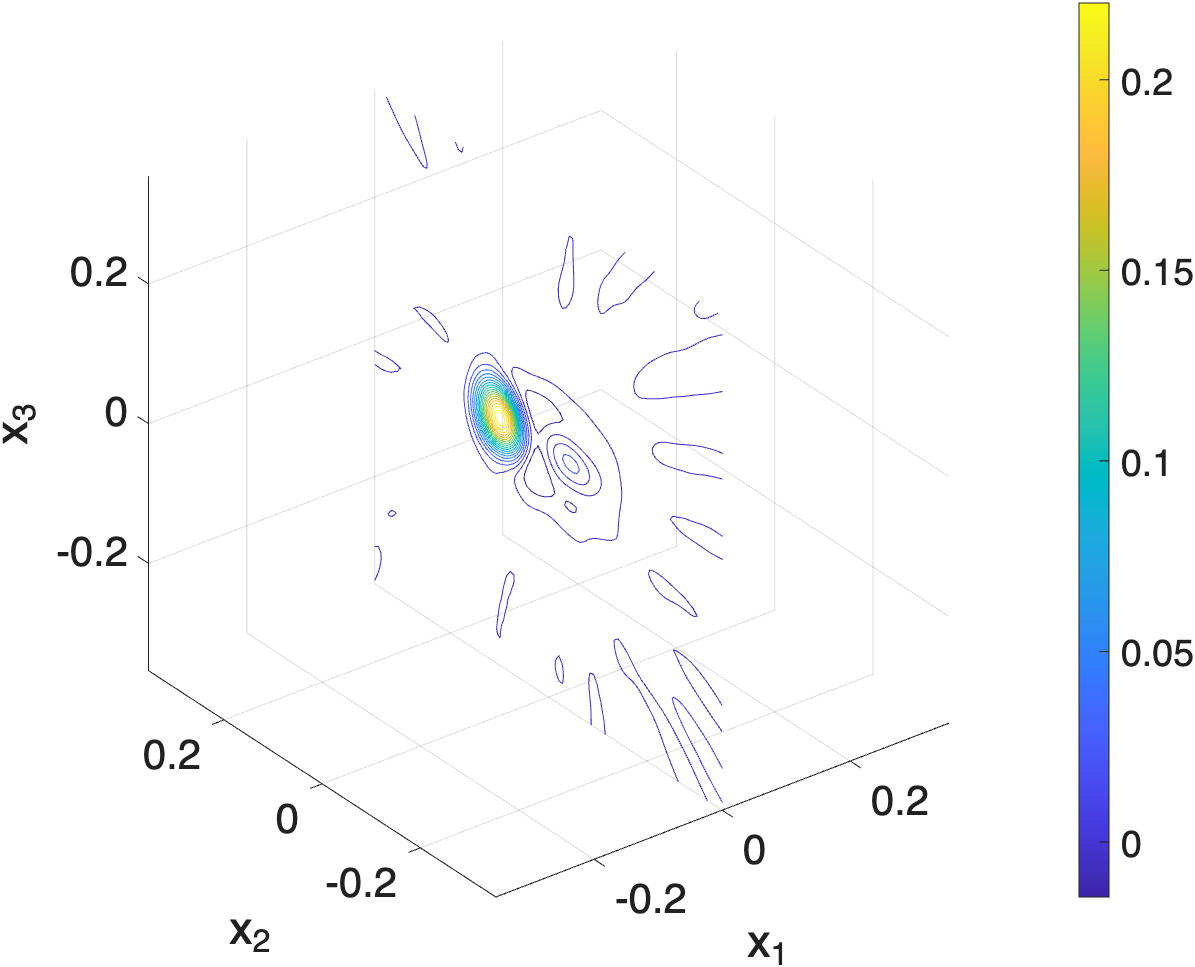}}\quad
    \subfigure[$\Re({\mathbf I}^s)$ in $x_2=0$]{\includegraphics[width=0.3\linewidth]{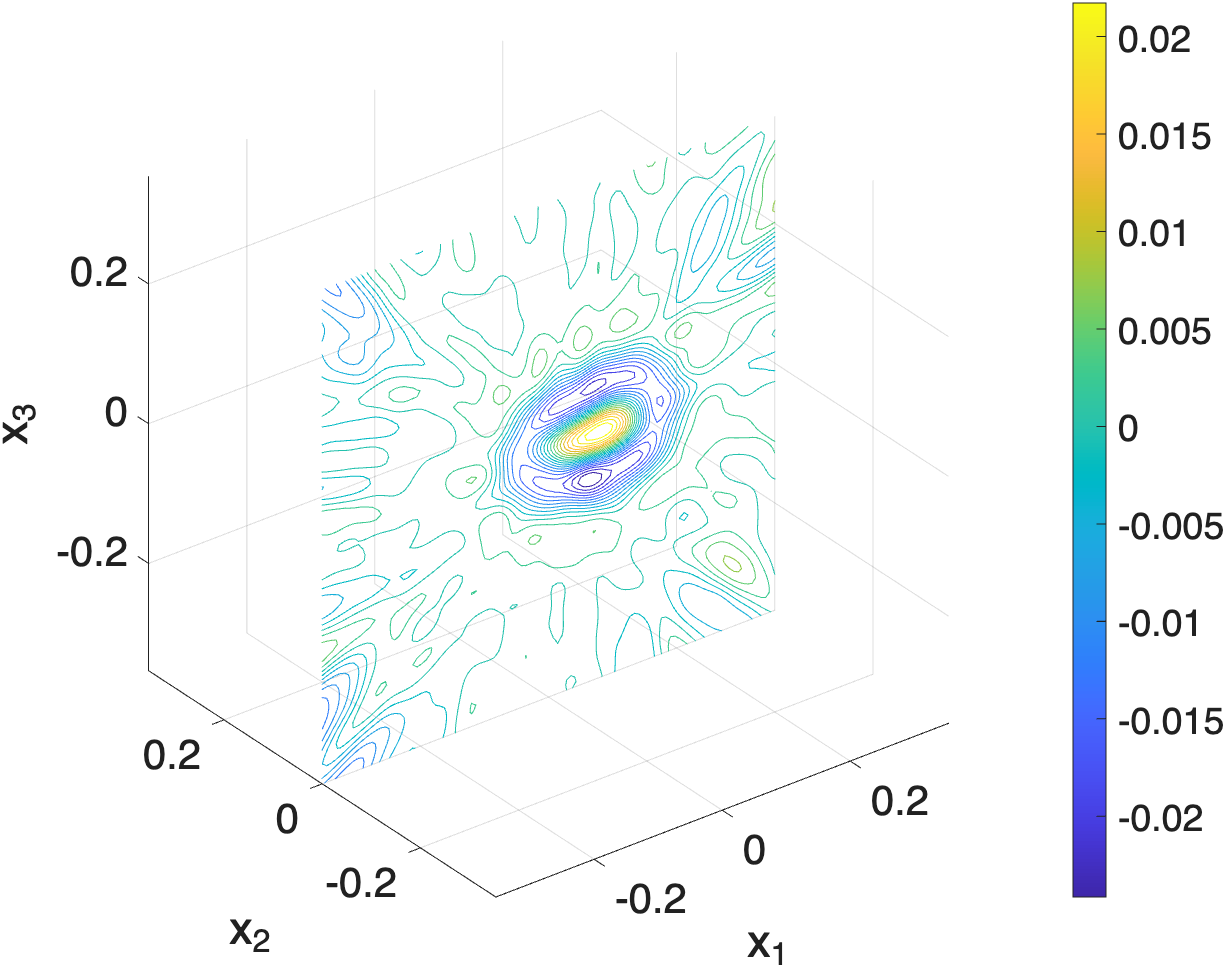}}\quad
    \subfigure[$\Re({\mathbf I}^s)$ in $x_3=0$]{\includegraphics[width=0.3\linewidth]{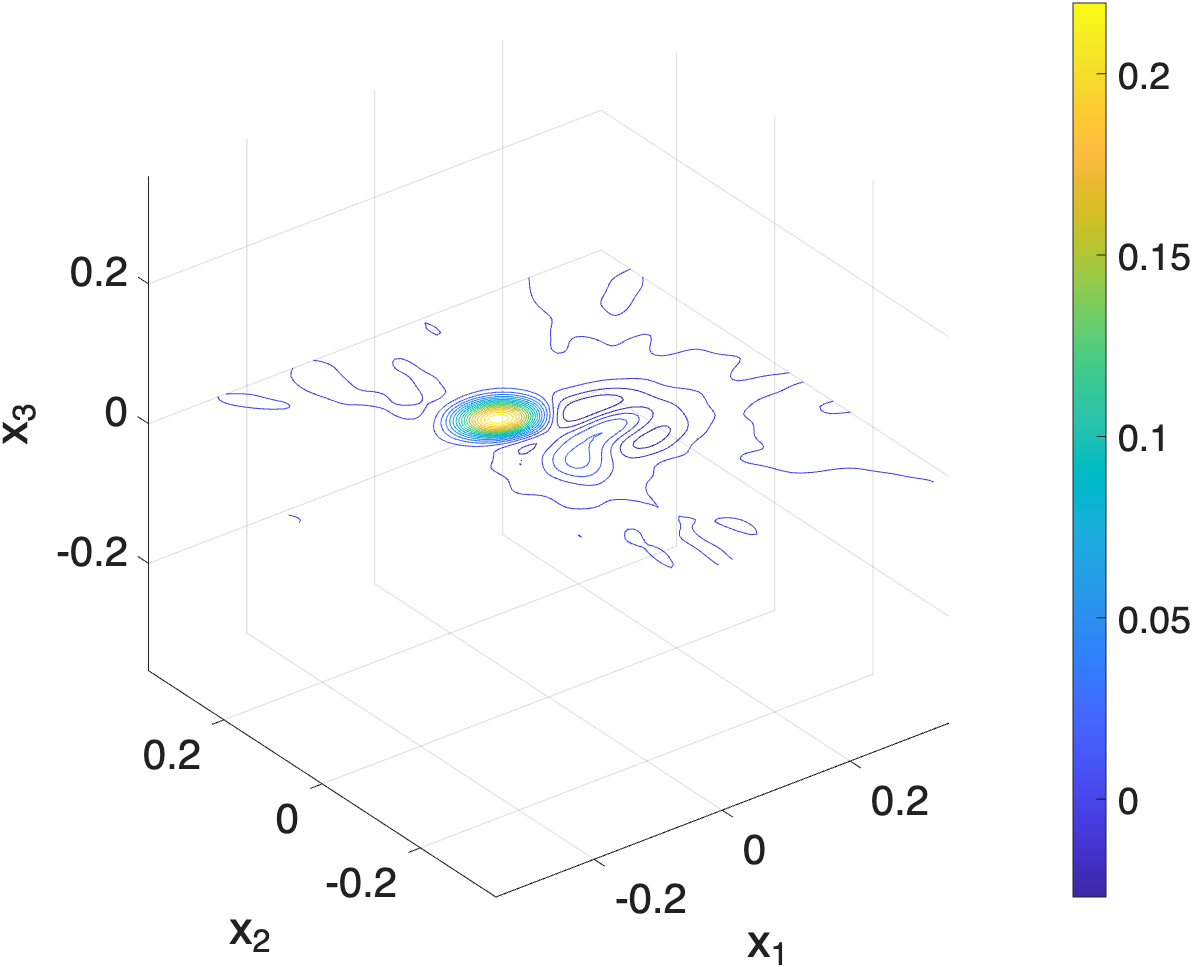}}\\
    \subfigure[$\Im({\mathbf I}^s)$ in $x_1=0$]{\includegraphics[width=0.3\linewidth]{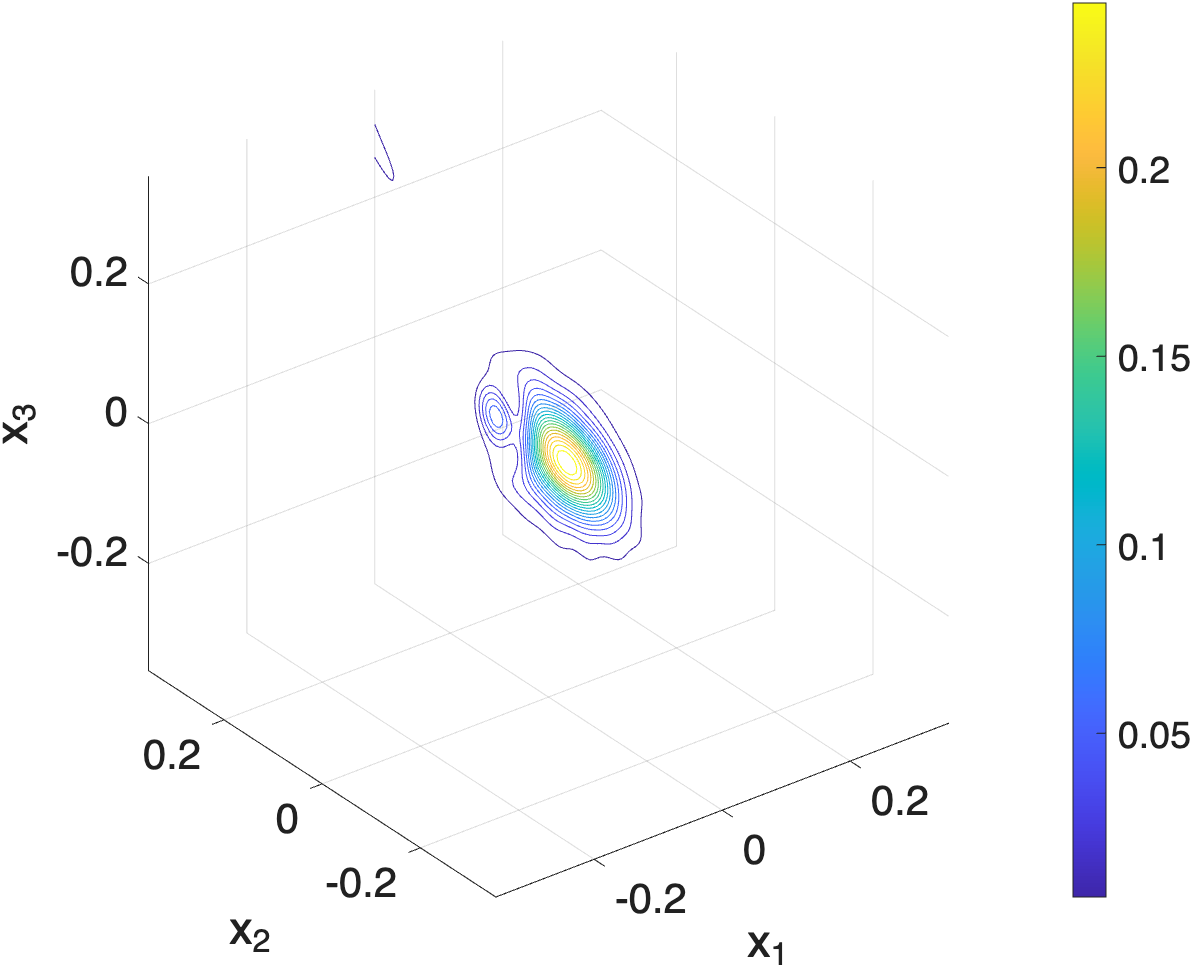}}\quad
    \subfigure[$\Im({\mathbf I}^s)$ in $x_2=0$]{\includegraphics[width=0.3\linewidth]{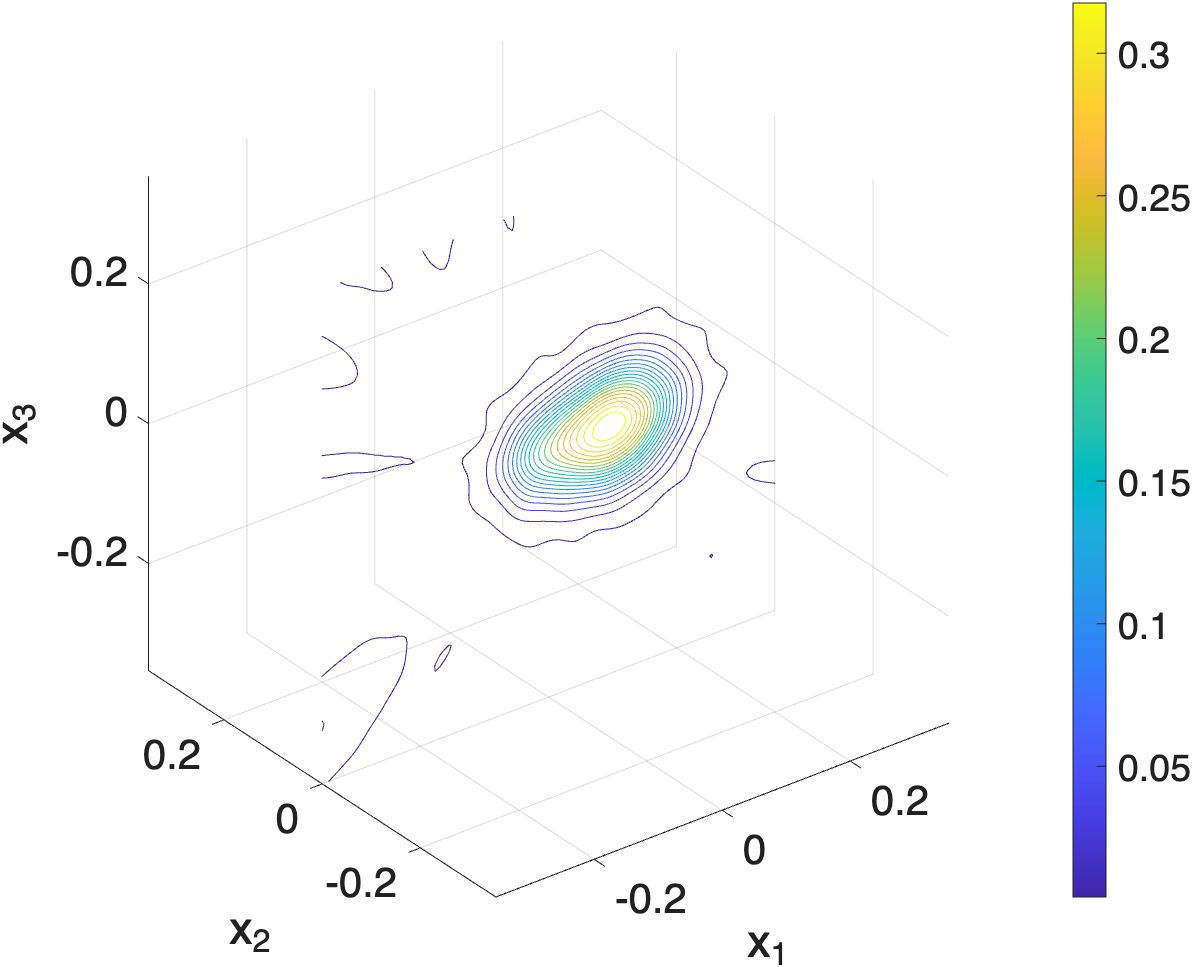}}\quad
    \subfigure[$\Im({\mathbf I}^s)$ in $x_3=0$]{\includegraphics[width=0.3\linewidth]{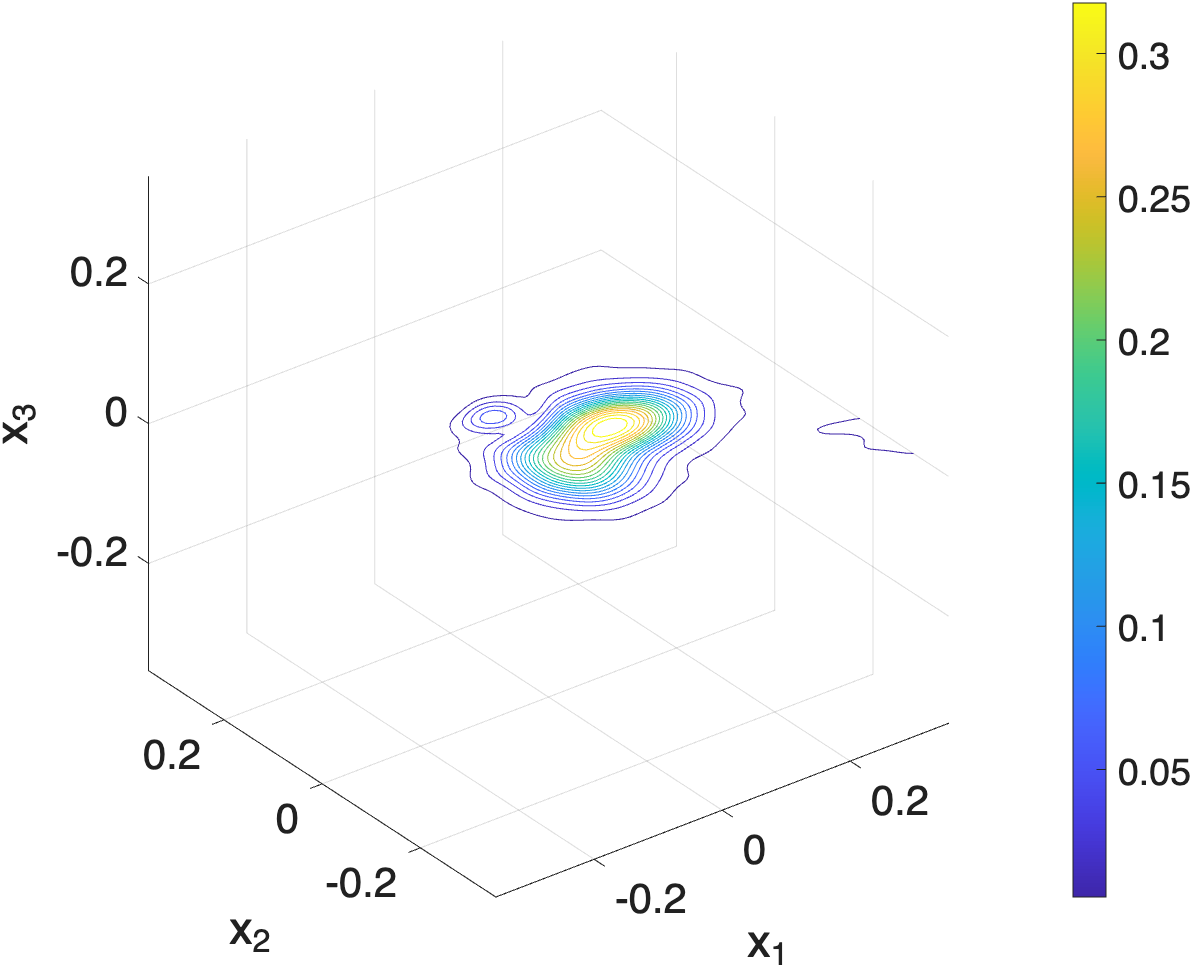}}
    \caption{Contour plots of ${\mathbf I}^s$ in slice planes ($C_s=10^{-1}$).}
    \label{fig: 3D_smooth_reconstruction_slice_1e-1}
\end{figure}

\begin{figure}
    \centering
    \subfigure[$\Re({\mathbf I}^s)$]{\includegraphics[width=0.3\linewidth]{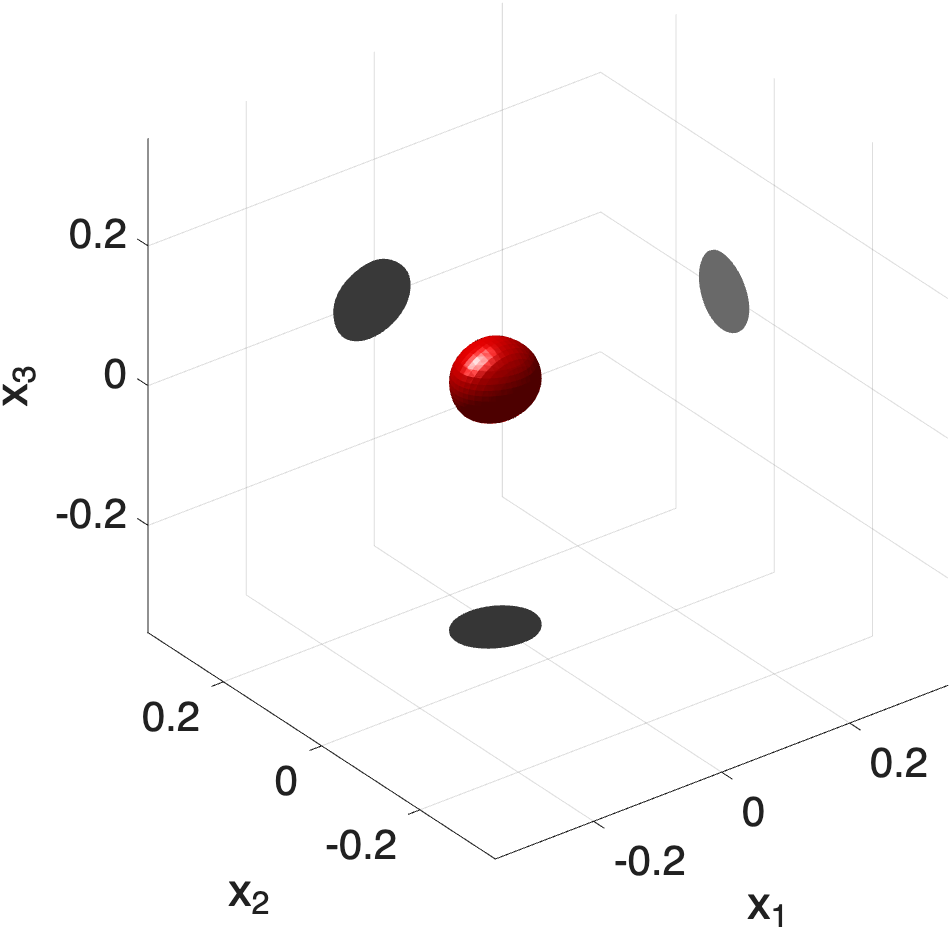}}\quad
    \subfigure[$\Im({\mathbf I}^s)$]{\includegraphics[width=0.3\linewidth]{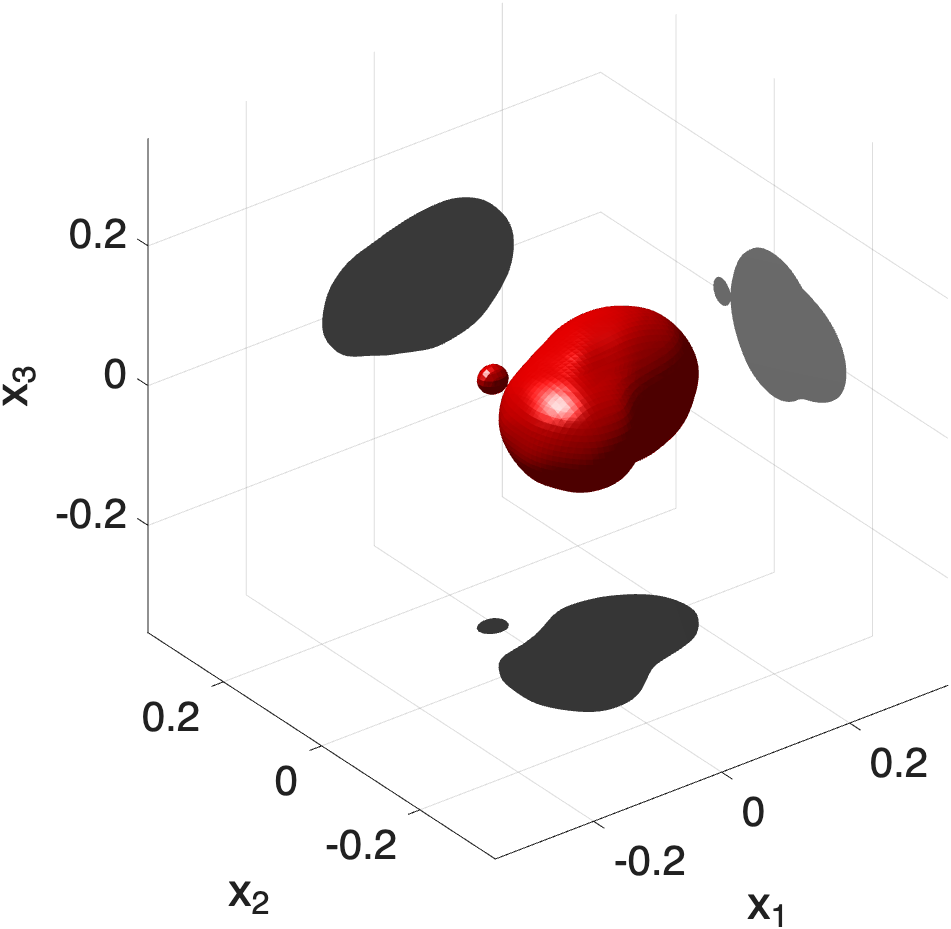}}\quad
    \subfigure[$|{\mathbf I}^s|$]{\includegraphics[width=0.3\linewidth]{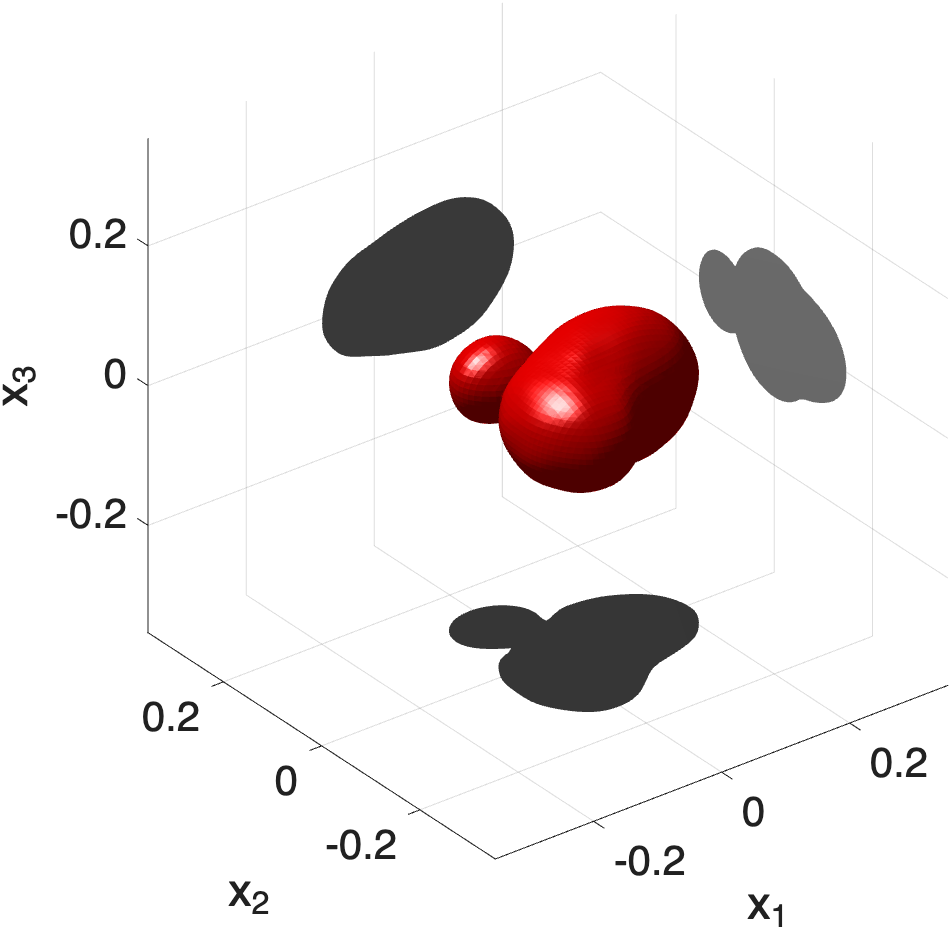}}
    \caption{Iso-surface plots with iso-surface value $C_{\rm iso}=5\times 10^{-2}$ ($C_s=10^{-1}$).}
    \label{fig: 3D_smooth_reconstruction_isosurface_5e-2}
\end{figure}
\end{example}

\begin{example}\label{ex: 3D-cross}
\begin{figure}
    \centering
    \includegraphics[width=0.3\linewidth]{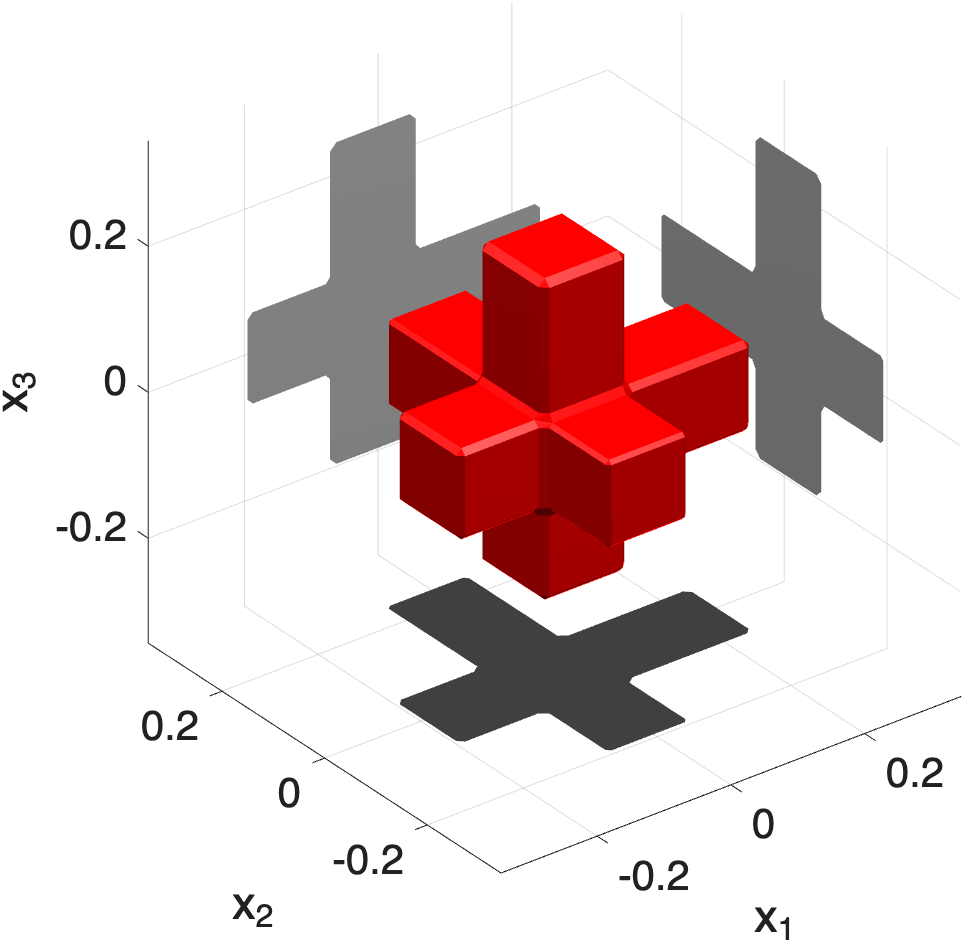}
    \caption{The ground-truth cross-shaped contrast.}
    \label{fig: 3D_cross_exact}
\end{figure}
This final example delves into the recovery of a cross-shaped contrast. The geometrical structure of the cross is composed of three perpendicular bars, $\Omega_1,\Omega_2,\Omega_3$, each with length 7/16 and width 1/8. We refer to Figure \ref{fig: 3D_cross_exact} for a view of the exact cross-shaped scatterer from several observational directions. Specifically, the first horizontal bar $\Omega_1$, the second horizontal bar $\Omega_2$, and the vertical bar $\Omega_3$, are defined as follows. 
\begin{align*}
\Omega_1 & :=\left\{x=(x_1, x_2, x_3)\in\mathbb{R}^3: -\frac 3 {16}\leq x_1\leq \frac14,\ -\frac 1 {16}\leq x_2\leq \frac 1 {16},\ -\frac 1 {16}\leq x_3\leq \frac 1 {16}\right\},\\
\Omega_2 & :=\left\{x=(x_1, x_2, x_3)\in\mathbb{R}^3: -\frac 1 {16}\leq x_1\leq \frac 1 {16},\ -\frac 3 {16}\leq x_2\leq \frac14,\ -\frac 1 {16}\leq x_3\leq \frac 1 {16}\right\},\\
\Omega_3 & :=\left\{x=(x_1, x_2, x_3)\in\mathbb{R}^3: -\frac 1 {16}\leq x_1\leq \frac 1 {16},\ -\frac 1 {16}\leq x_2\leq \frac 1 {16},\ -\frac 3 {16}\leq x_3\leq \frac14\right\}.
\end{align*}
Obviously, their intersectional center block is the cube
$$
\Omega_1\cap\Omega_2=\left\{x=(x_1, x_2, x_3)\in\mathbb{R}^3:  -\frac 1 {16}\leq x_j\leq \frac 1 {16},\ j=1,2,3\right\}.
$$
The ground-truth contrast of the cross medium is given by the block-wise constant function
$$
q:=8\times 10^{-3}\chi_{\Omega_1\setminus\Omega_2}+6\times 10^{-3}\chi_{\Omega_2\setminus\Omega_1}+1\times 10^{-2}\chi_{\Omega_3},
$$
where
$$
\chi_{\Omega}(x)=
\begin{cases}
    1, & x\in\Omega,\\
    0, & x\notin\Omega,
\end{cases}
$$
denotes the characteristic function of a generic domain $\Omega$.

For the reconstruction, we use the far-field data observed at $N_\theta=256$ directions and $\Delta k=2$, $k_{\rm min}=1$, $k_{\rm max}=81$, with $\delta=1\%$ noise level. The sampling region is chosen as $[-0.35, 0.35]^3$ with a $101^3$ uniformly spaced sampling grid. 
Figure \ref{fig: 3D_Cross_farField_slice} shows the reconstructions of the cross-shaped contrast in the form of contour slices. It can be observed from Figure \ref{fig: 3D_Cross_farField_slice} that both the support of contrast profile and its contrast value could be recovered.

\begin{figure}
    \centering
    \subfigure[$x_1=0$]{\includegraphics[width=0.3\linewidth]{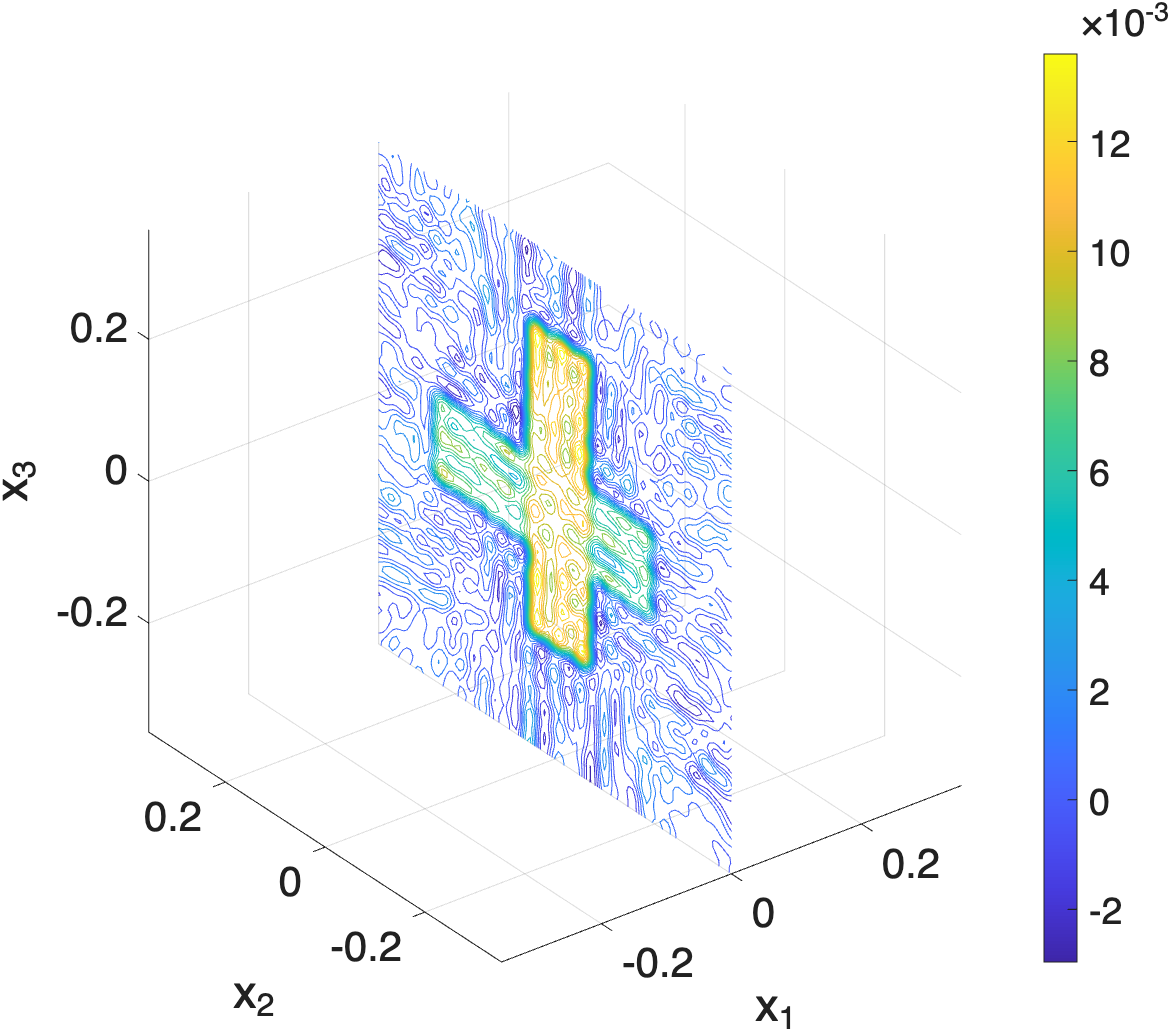}}\quad
    \subfigure[$x_2=0$]{\includegraphics[width=0.3\linewidth]{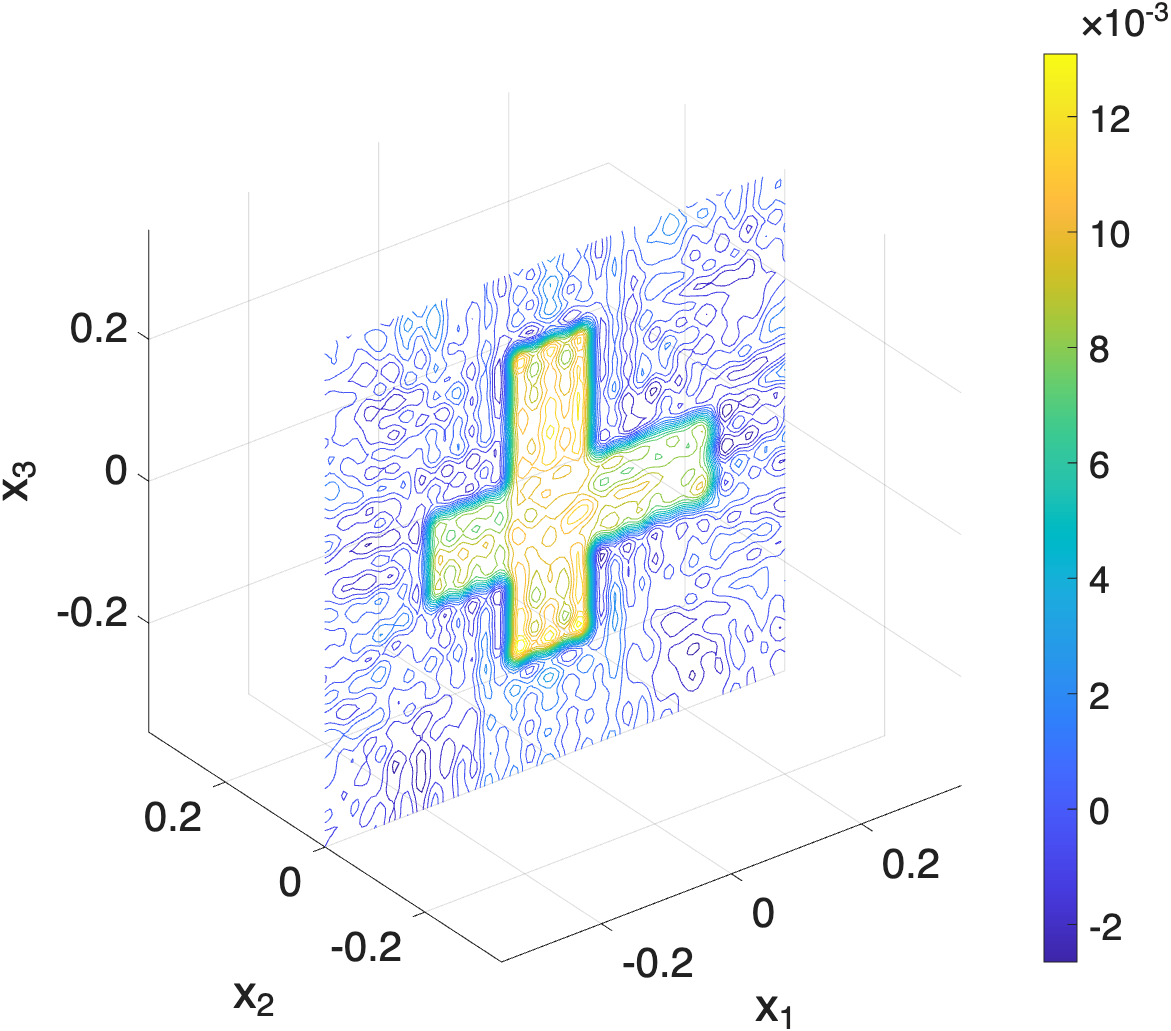}}\quad
    \subfigure[$x_3=0$]{\includegraphics[width=0.3\linewidth]{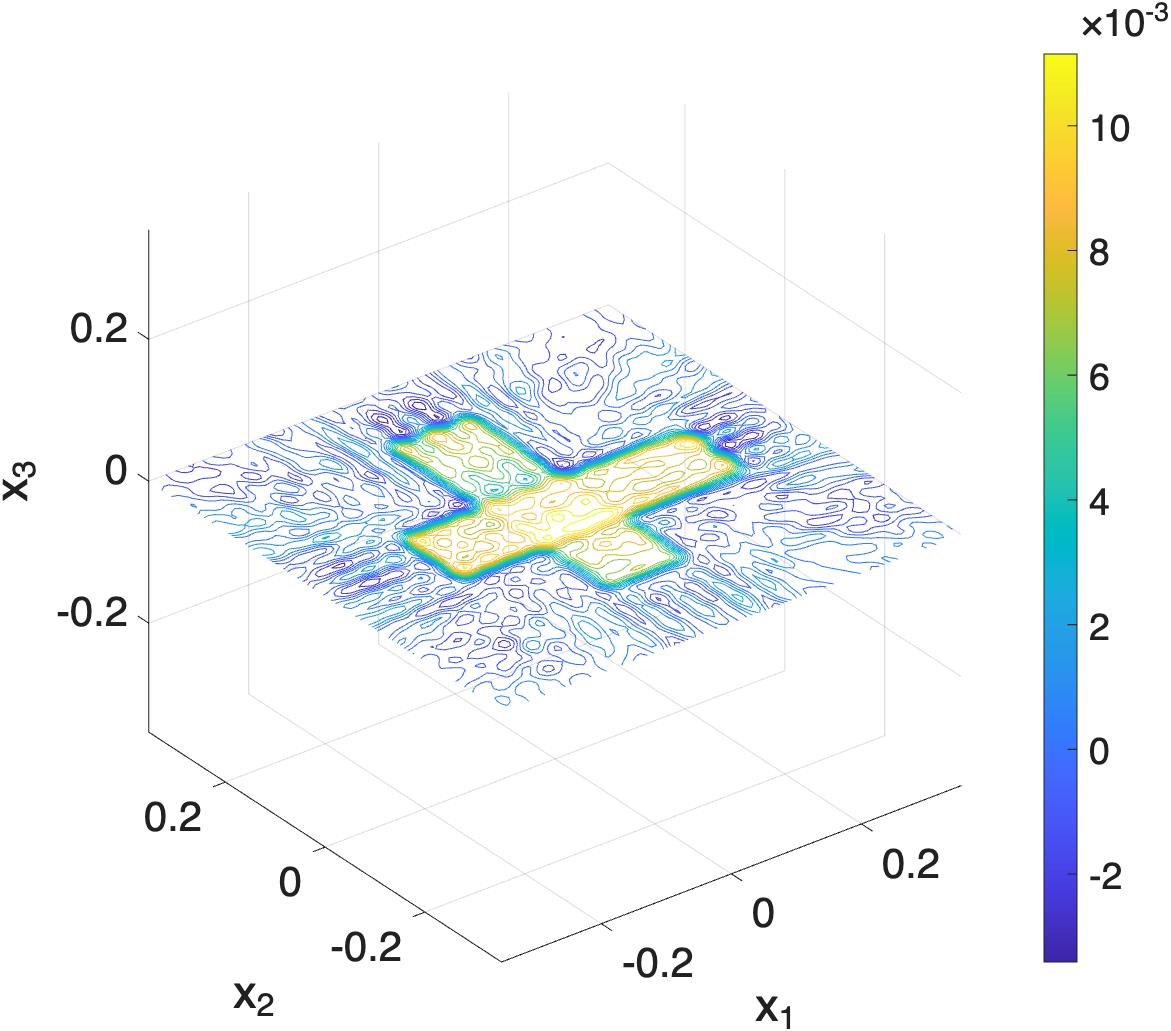}}
    \caption{Contour plots of $\Re({\mathbf I}^{\infty})$ at different slice positions.}
    \label{fig: 3D_Cross_farField_slice}
\end{figure}

To further shed light on the inversion results, the iso-surface version of reconstructions is exhibited in Figure \ref{fig: 3D_Cross_farField_isosurface}. Here, we plot the iso-surfaces of the imaging results by extracting iso-surface value data (with equal value $C_{\rm iso}$) from the volume indicator $\Re({\mathbf I}^{\infty})$. Note that when the iso-surface value equals the exact contrast of the vertical bar (that is, $C_{\rm iso}=q|_{\Omega_3}=10^{-2}$), only the vertical bar $\Omega_3$ can be well identified whereas the horizontal bars are almost invisible, see Figure  \ref{fig: 3D_Cross_farField_isosurface}(a). If we modify the iso-surface value $C_{\rm iso}$ to $q|_{\Omega_1\setminus\Omega_2}$ and $q|_{\Omega_2\setminus\Omega_1}$ respectively, then the two horizontal bars consequently and successively emerge, see Figure \ref{fig: 3D_Cross_farField_isosurface}(b)(c).

\begin{figure}
    \centering
    \subfigure[$C_{\rm iso}=10^{-2}$]{\includegraphics[width=0.3\linewidth]{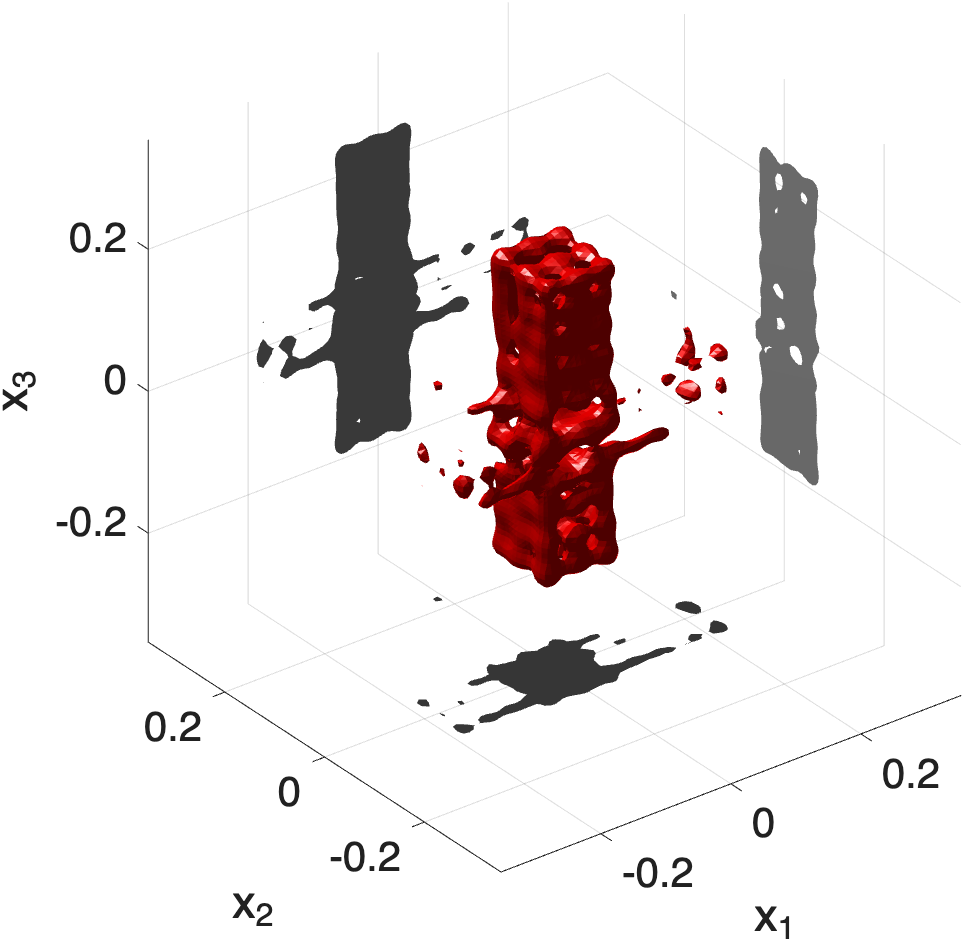}}\quad
    \subfigure[$C_\mathrm{iso}=8\times 10^{-3}$]{\includegraphics[width=0.3\linewidth]{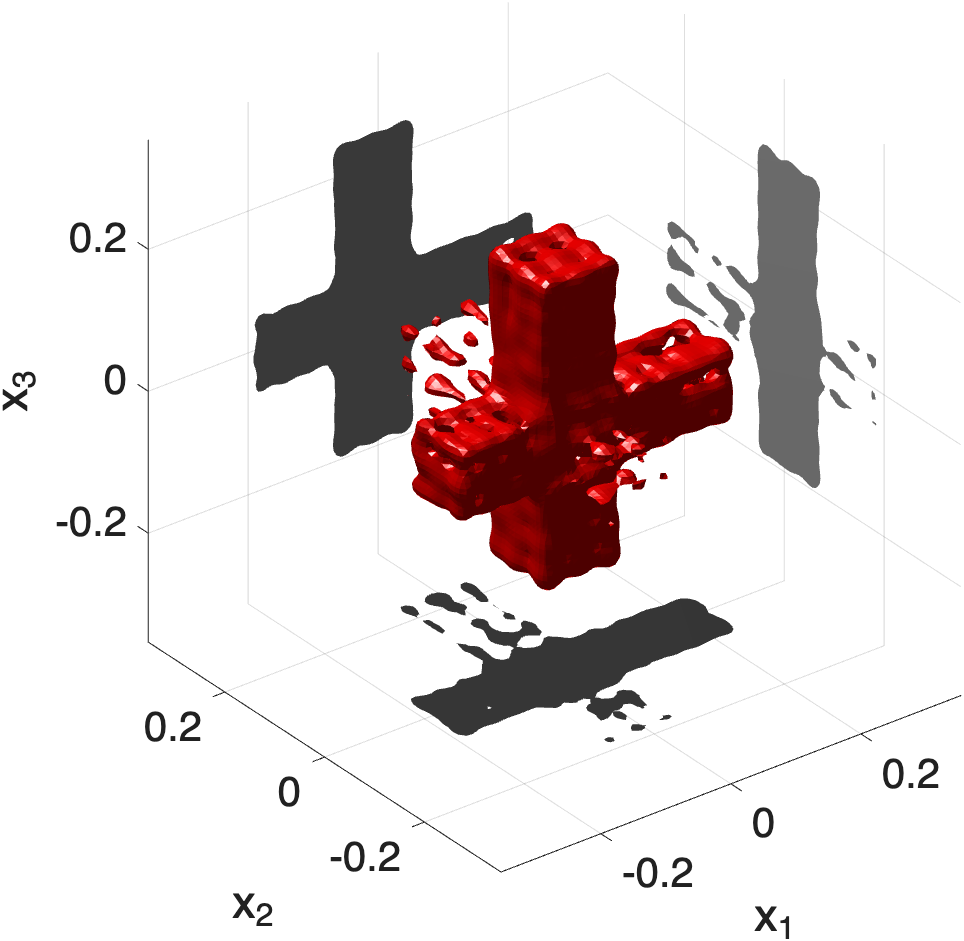}}\quad
    \subfigure[$C_\mathrm{iso}=6\times 10^{-3}$]{\includegraphics[width=0.3\linewidth]{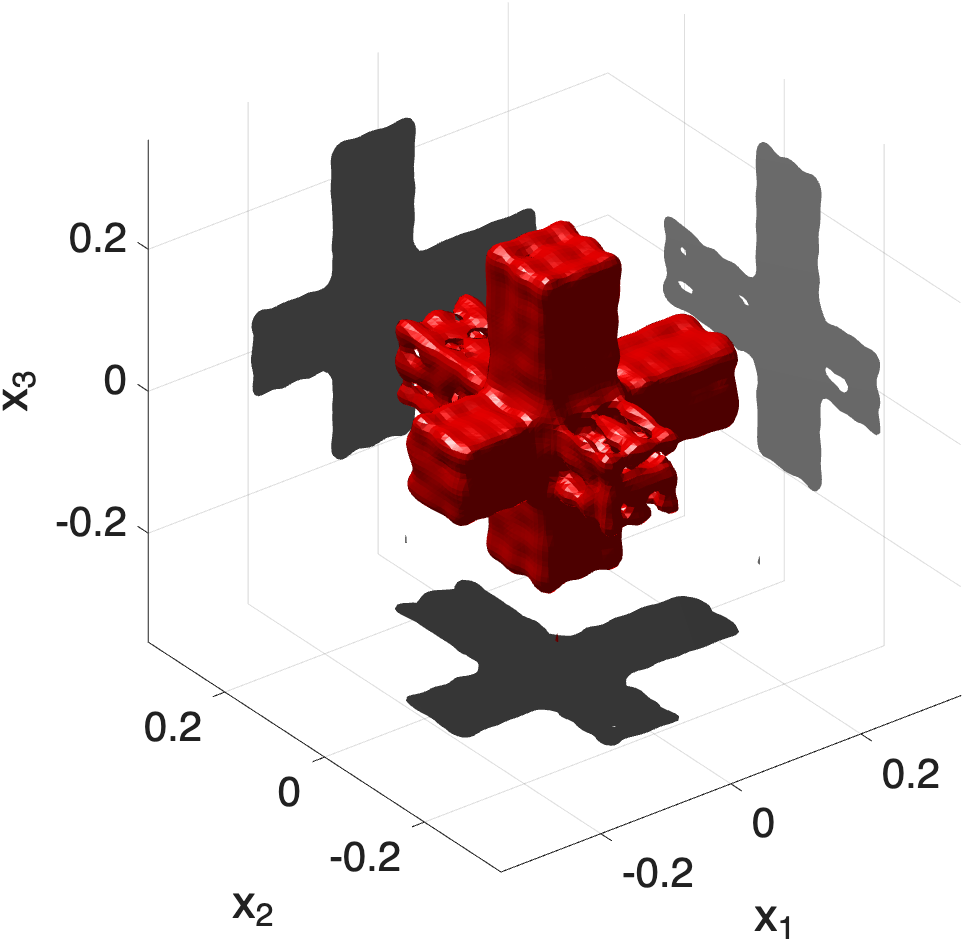}}
    \caption{Iso-surface plots of $\Re({\mathbf I}^{\infty})$ with iso-surface value $C_{\rm iso}$.}
    \label{fig: 3D_Cross_farField_isosurface}
\end{figure}

Finally we remove the center block $\Omega_1\cap\Omega_2$ from the aforementioned solid cross and test the algorithm's capability of recovering the hollow-center cross. Now the ground-truth contrast of the cross medium is given by
$$
q:=8\times 10^{-3}\chi_{\Omega_1\setminus\Omega_2}+6\times 10^{-3}\chi_{\Omega_2\setminus\Omega_1}+1\times 10^{-2}\chi_{\Omega_3\setminus(\Omega_1\cap\Omega_2)},
$$
Note that the hollow-center cross consists of six adjacent but disconnected blocks with an exterior surface almost identical to the previous solid cross, hence the geometrical structure of this hollow-center cross is more complicated than the original filled-center one. The corresponding reconstructions are shown in Figure \ref{fig: 3D_cross_hollow_slice} and Figure \ref{fig: 3D_cross_hollow_isosurface}.

\begin{figure}
    \centering
    \subfigure[$x_1=0$]{\includegraphics[width=0.3\linewidth]{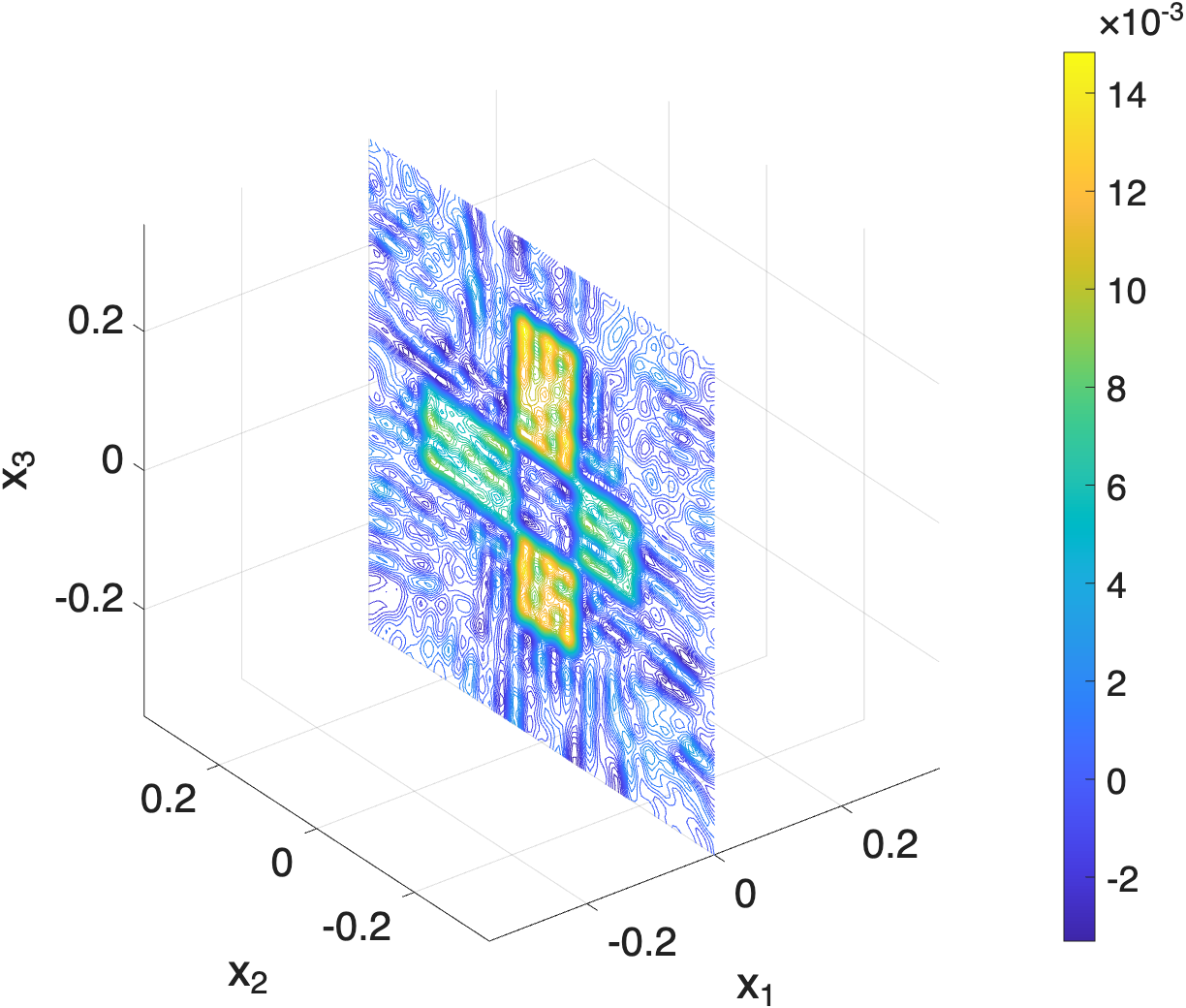}}\quad
    \subfigure[$x_2=0$]{\includegraphics[width=0.3\linewidth]{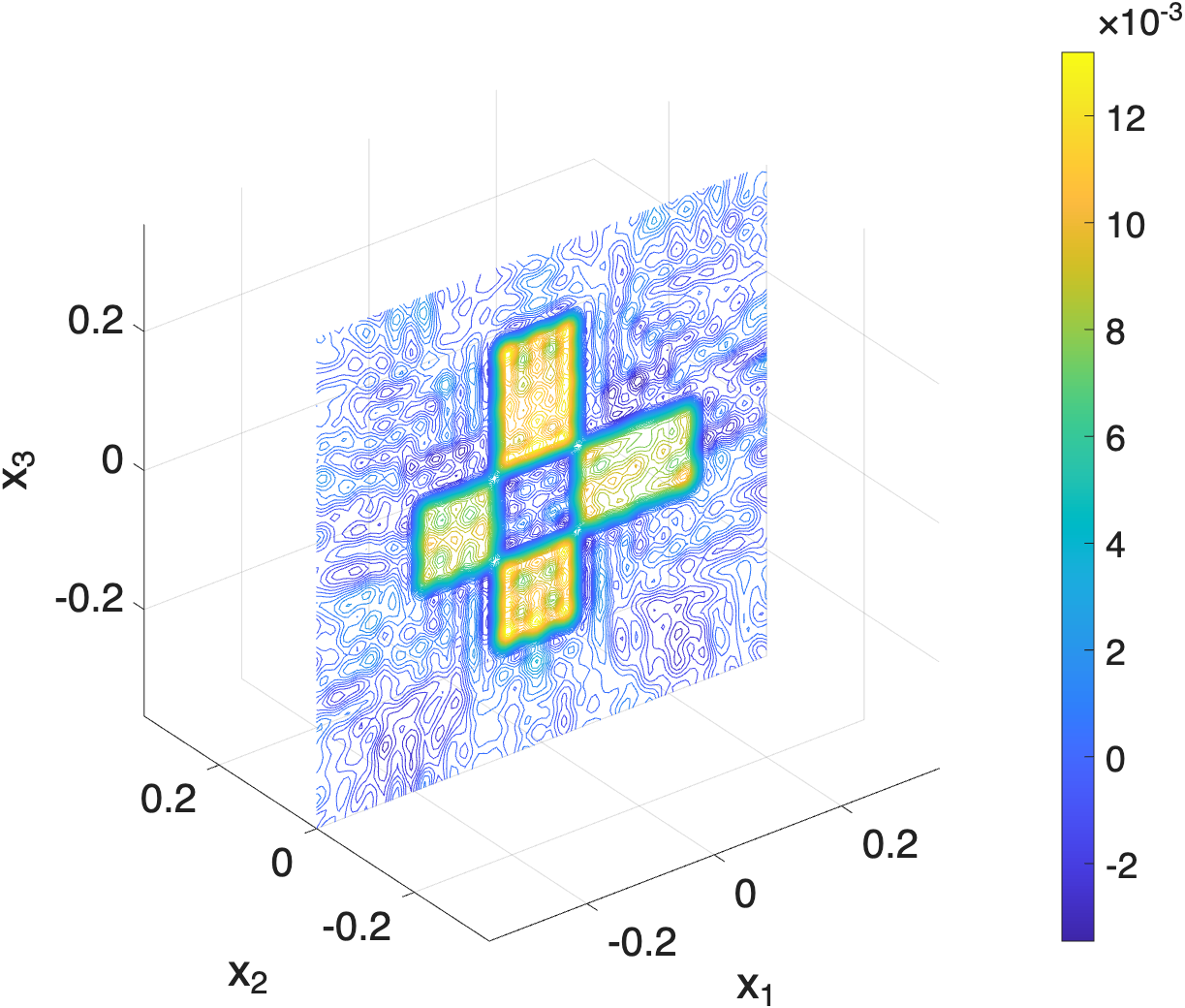}}\quad
    \subfigure[$x_3=0$]{\includegraphics[width=0.3\linewidth]{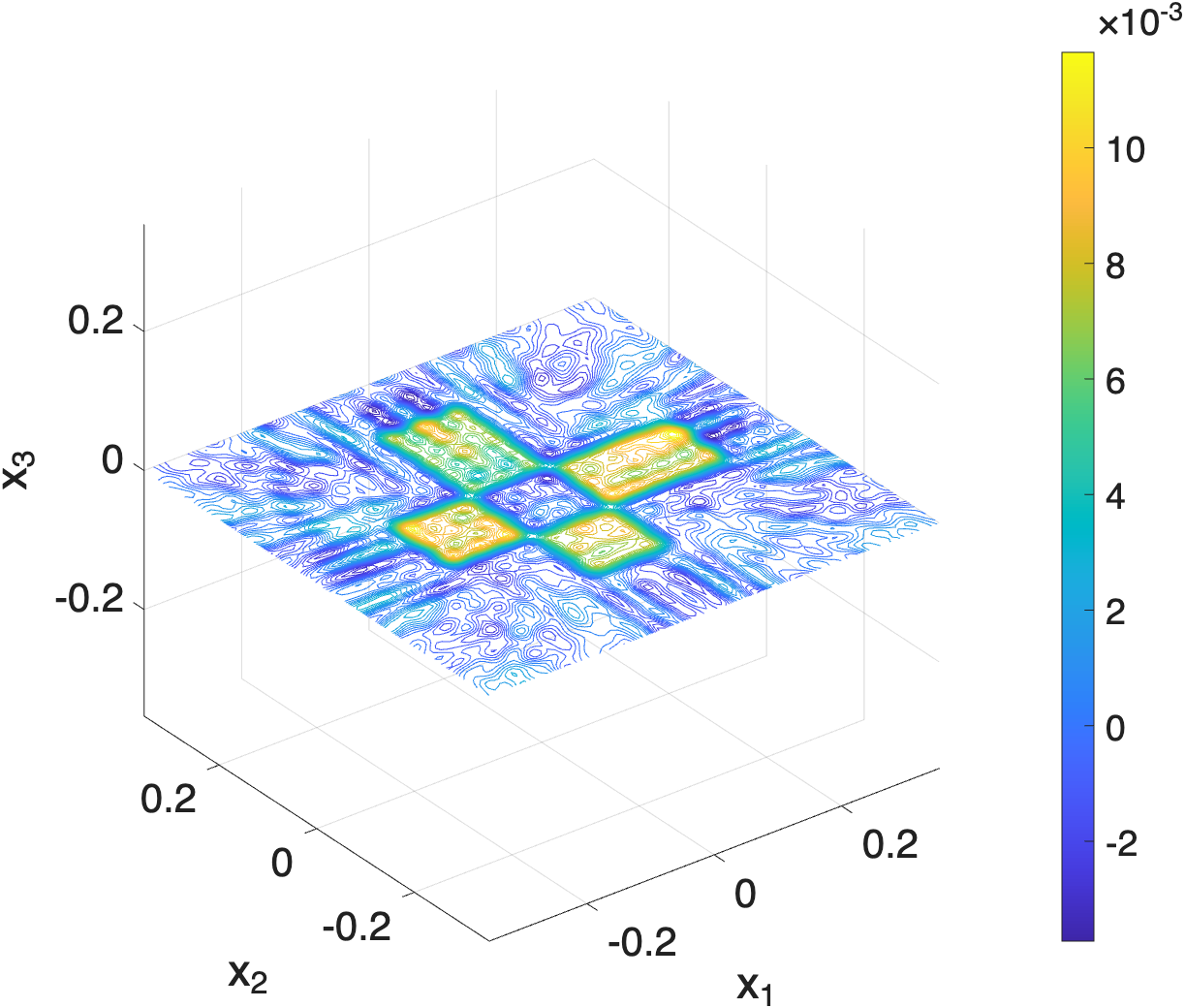}}
    \caption{Contour plots of $\Re({\mathbf I}^{\infty})$ for imaging the hollow-center cross at different slice positions.}
    \label{fig: 3D_cross_hollow_slice}
\end{figure}
\begin{figure}
    \centering
    \subfigure[$C_{\rm iso}=10^{-2}$]{\includegraphics[width=0.3\linewidth]{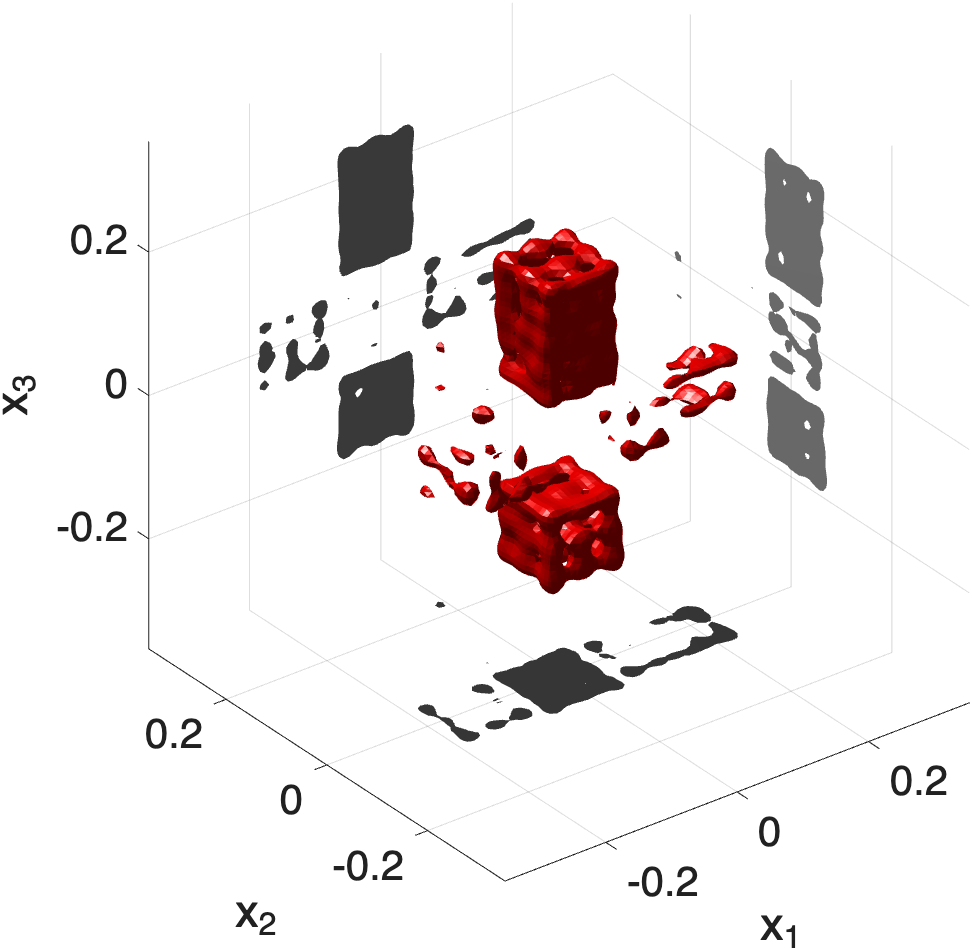}}\quad
    \subfigure[$C_\mathrm{iso}=8\times 10^{-3}$]{\includegraphics[width=0.3\linewidth]{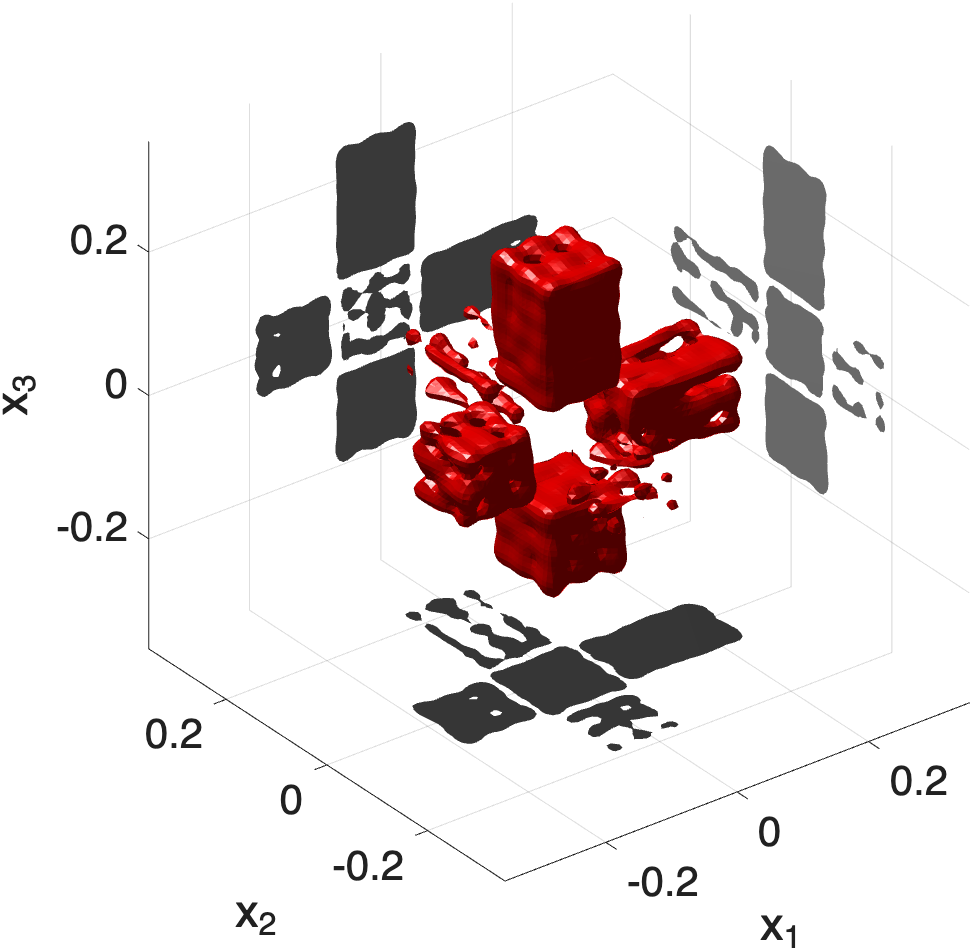}}\quad
    \subfigure[$C_\mathrm{iso}=6\times 10^{-3}$]{\includegraphics[width=0.3\linewidth]{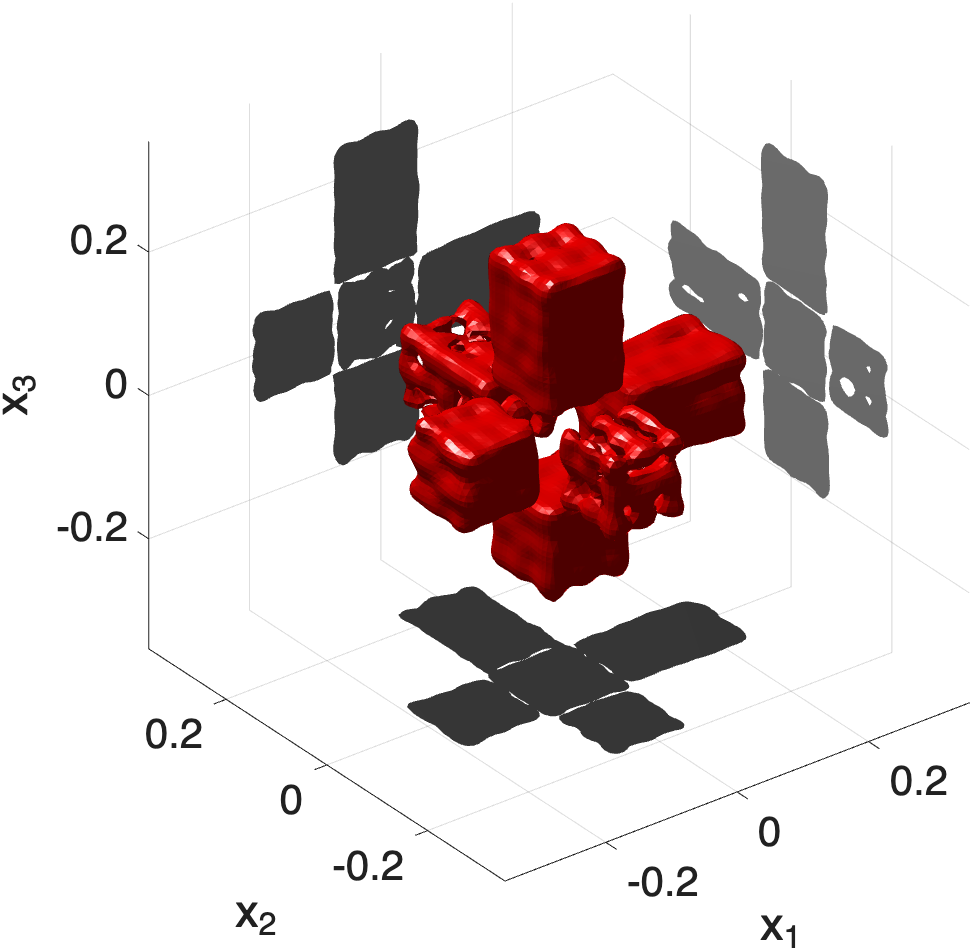}}
    \caption{Iso-surface plots of $\Re({\mathbf I}^{\infty})$ with iso-surface value $C_{\rm iso}$ for the hollow-center cross.}
    \label{fig: 3D_cross_hollow_isosurface}
\end{figure}
\end{example}

\section{Conclusion}
This paper investigates the quantitative reconstruction of inhomogeneous media from multi-frequency backscattering data, a long-standing open problem in inverse medium scattering. We prove a local uniqueness result for the contrast function under the assumption that the difference of two contrasts is at most a first-order polynomial, providing a theoretical guarantee for reconstruction. By employing the Born approximation and Fourier transform, we derive concise representations of backscattering scattered fields and far-field patterns, and propose a QDSM with two explicit indicators. The proposed method is non-iterative, initial-guess-free, and avoids inverse crime inherently. Numerical results in 2D and 3D show that QDSM can accurately and robustly recover the contrast functions and is insensitive to moderate noise. 

Future work will focus on relaxing the limitation of the Born approximation to treat high-contrast media, where the total field can no longer be simply approximated by the incident field. We will explore combining QDSM with the distorted Born approximation, regularized iterative corrections, or modern machine learning and Bayesian techniques to extend its applicability to high-contrast inhomogeneities. Further efforts will be made to extend the local uniqueness result to general contrast functions, develop adaptive frequency selection strategies for faster imaging, and extend the method to elastic and electromagnetic wave scattering.

\section*{Acknowledgment}
The research of X. Liu is supported by the National Key R\&D Program of China under grant 2024YFA1012303 and the NNSF of China under grant 12371430. The research of Y. Guo is supported by the NNSF of China under grants 11971133 and 12571452.


\bibliographystyle{plain}
\bibliography{myreferences}

@book{colton2019inverse,
  title     = {Inverse Acoustic and Electromagnetic Scattering Theory},
  author    = {Colton, D. and Kress, R.},
  edition   = {4},
  publisher = {Springer},
  year      = {2019}
}

@article{Chen2010SOMInhom,
  author  = {Chen, X.},
  title   = {Subspace-based optimization method for inverse scattering problems with an inhomogeneous background medium},
  journal = {Inverse Problems},
  year    = {2010},
  volume  = {26},
  pages   = {074007}
}

@article{kirsch2017remarks,
  title={Remarks on the {Born} approximation and the Factorization Method},
  author={Kirsch, A.},
  journal={Applicable Analysis},
  volume={96},
  number={1},
  pages={70--84},
  year={2017}
}

@article{vandenberg1997contrast,
  title   = {A contrast source inversion method},
  author  = {van den Berg, P. M. and Kleinman, R. E.},
  journal = {Inverse Problems},
  volume  = {13},
  number  = {6},
  pages   = {1607--1624},
  year    = {1997}
}

@article{hanke1999iterative,
  title   = {Iterative methods for nonlinear ill-posed problems: convergence and regularization},
  author  = {Hanke, M. and Scherzer, O.},
  journal = {Inverse Problems},
  volume  = {15},
  number  = {3},
  pages   = {R69--R106},
  year    = {1999}
}

@article{colton1996simple,
  author = {Colton, D. and Kirsch, A.},
  title = {A simple method for solving inverse scattering problems in the resonance region},
  journal = {Inverse Problems},
  volume = {12},
  pages = {383--393},
  year = {1996}
}

@article{wang1998stability,
  author = {Wang, J-N},
  title = {Stability estimate for an inverse acoustic backscattering problem},
  journal = {Inverse Problems},
  volume = {14},
  pages = {197--207},
  year = {1998}
}

@article{kirsch1999factorization,
  author = {Kirsch, A.},
  title = {Factorization of the far field operator for the inhomogeneous medium case and an application to inverse scattering theory},
  journal = {Inverse Problems},
  volume = {15},
  pages = {413--429},
  year = {1999}
}

@article{BaoTriki,
  title   = {Error estimates for the recursive linearization of inverse medium problems},
  author  = {Bao, G. and Triki, F.},
  journal = {Journal of Computational Mathematics},
  volume  = {28},
  pages   = {725--744},
  year    = {2010}
}

@article{meng2024kernel,
  title={A kernel machine learning for inverse source and scattering problems},
  author={Meng, S. and Zhang, B.},
  journal={SIAM Journal on Numerical Analysis},
  volume={62},
  number={3},
  pages={1443--1464},
  year={2024},
  publisher={SIAM}
}

@article{li2024reconstruction,
  title={Reconstruction of inhomogeneous media by an iteration algorithm with a learned projector},
  author={Li, K. and Zhang, B. and Zhang, H.},
  journal={Inverse Problems},
  volume={40},
  number={7},
  pages={075008},
  year={2024},
  publisher={IOP Publishing}
}

@article{desai2025neural,
  title={A neural network enhanced {Born} approximation for inverse scattering},
  author={Desai, A. and Ma, J. and Lahivaara, T. and Monk, P.},
  journal={arXiv preprint arXiv:2503.01596},
  year={2025}
}

@article{zhou2026exploring,
  title={Exploring low-rank structure for an inverse scattering problem with far-field data},
  author={Zhou, Y. and Audibert, L. and Meng, S. and Zhang, B.},
  journal={SIAM Journal on Applied Mathematics},
  volume={86},
  number={1},
  pages={179--205},
  year={2026},
  publisher={SIAM}
}

@article{zhou2025recovery,
  title={On the recovery of two function-valued coefficients in the {Helmholtz} equation for inverse scattering problems via neural networks},
  author={Zhou, Z.},
  journal={Advances in Computational Mathematics},
  volume={51},
  number={1},
  pages={12},
  year={2025},
  publisher={Springer}
}

@article{ItoJinZou,
  title   = {A direct sampling method to an inverse medium scattering problem},
  author  = {Ito, K. and Jin, B. and Zou, J.},
  journal = {Inverse Problems},
  volume  = {28},
  pages   = {025003},
  year    = {2012}
}

@article{LiLiuShi2025,
  title   = {Identifying strictly convex obstacles from backscattering far field data},
  author  = {Li, J. and Liu, X. and Shi, Q.},
  journal = {arxiv:2505.11850},
  year    = {2025}
}

@article{LiuIP17,
  title   = {A novel sampling method for multiple multiscale targets from scattering amplitudes at a fixed frequency},
  author  = {Liu, X.},
  journal = {Inverse Problems},
  volume  = {33},
  pages   = {085011},
  year    = {2017}
}

@article{LiuShi,
  title   = {A quantitative sampling method for elastic and electromagnetic sources},
  author  = {Liu, X. and Shi, Q.},
  journal = {Journal of Computational Physics},
  volume  = {539},
  pages   = {114251},
  year    = {2025}
}

@article{LiuWang-AML,
title = {A Radon transform-based formula for reconstructing acoustic sources from the scattered fields},
author = {Liu, X. and Wang, J.},
journal = {Applied Mathematics Letters},
pages = {109988},
year = {2026},
issn = {0893-9659},
doi = {https://doi.org/10.1016/j.aml.2026.109988},
url = {https://www.sciencedirect.com/science/article/pii/S0893965926001199},
}

@article{StefUhl,
  title   = {Inverse backscattering for the acoustic equation},
  author  = {Stefanov, P. and Uhlmann, G.},
  journal = {SIAM Journal on Mathematical Analysis},
  volume  = {28},
  pages   = {1191--1204},
  year    = {1997}
}

@article{khoo2019switchnet,
  title={Switchnet: a neural network model for forward and inverse scattering problems},
  author={Khoo, Y. and Ying, L.},
  journal={SIAM Journal on Scientific Computing},
  volume={41},
  number={5},
  pages={A3182--A3201},
  year={2019},
  publisher={SIAM}
}

@article{liu2026direct,
  title={Direct sampling methods for acoustic sources: advancing to quantitative identification from qualitative probes},
  author={Liu, X. and Shi, Q.},
  journal={Inverse Probl. Imag.},
  volume={25},
  pages={67-93},
  year={2026}
}

@article{liu2025radon,
  title={The {Radon} Transform-Based Sampling Methods for Biharmonic Sources from the Scattered Fields},
  author={Liu, X. and Shi, Q. and Wang, J.},
  journal={arXiv:2512.10332},
  year={2025}
}

@article{buergel2019ipscatt,
  title     = {Algorithm 1001: {IPscatt}—A {MATLAB} Toolbox for the Inverse Medium Problem in Scattering},
  author    = {B{\"u}rgel, F. and Kazimierski, K. S. and Lechleiter, A.},
  journal   = {ACM Transactions on Mathematical Software},
  volume    = {45},
  number    = {4},
  pages     = {1--20},
  year      = {2019},
  publisher = {Association for Computing Machinery},
  address   = {New York, NY, USA},
  doi       = {10.1145/3328525},
  url       = {https://doi.org/10.1145/3328525}
}

\end{document}